\documentclass{amsart}

\def\R{\mathbb R}

\def\Z{\mathbb Z}
\def\Q{\mathbb Q}

\newtheorem{crl}{Corollary}[section]

\newtheorem{lmm}{Lemma}[section]

\newtheorem{prp}{Proposition}[section]

\newtheorem{thm}{Theorem}[section]

\theoremstyle{definition}
\newtheorem{rem}{Remark}[section]
\newtheorem{dfn}{Definition}[section]
\newtheorem{exa}{Example}[section]
\newtheorem{assump}{Assumption}[section]
\theoremstyle{remark}

\title{Cyclic symmetry and adic convergence \\
in Lagrangian 
Floer theory
}
\author{Kenji Fukaya}

\begin{document}

\maketitle
\begin{abstract}                        
In this paper we use continuous family of multisections 
of the moduli space of pseudo holomorphic discs to partially improve the 
construction of Lagrangian Floer cohomology of \cite{FOOO080} in the case of $\R$ coefficient.
Namely we associate {\it cyclically symmetric}
filtered $A_{\infty}$ algebra to every relatively spin Lagrangian submanifold.
We use the same trick to construct a local rigid analytic family of filtered $A_{\infty}$
structure associated to a (family of) Lagrangian submanifolds.
We include the study of homological algebra of pseudo-isotopy of 
cyclic (filtered) $A_{\infty}$ algebra.
\end{abstract}

\tableofcontents                        

\section{Introduction}\footnote{Supported by JSPS Grant-in-Aid for Scientific Research
No.18104001 and Global COE Program G08.
}
In this paper we study de Rham version of Lagrangian Floer 
theory and use it to improve some of the results in \cite{FOOO080} over $\R$ coefficient. 
In particular, in this paper, we prove 
Conjectures 3.6.46, 3.6.48 and a part of Conjecture T in \cite{FOOO080}
over $\R$ coefficient.
Let $L$ be a relatively spin Lagrangian submanifold 
in a symplectic manifold $(M,\omega)$.
In this paper we always assume that $L$ is compact and 
$M$ is either compact or convex at infinity.
\par
The universal Novikov ring $\Lambda_{0,nov}^{\R}$ is defined in \cite{FOOO080}.
See also Definition \ref{Gnovikov}.
Let $H(L;\mathbb R)$ be the (de Rham) cohomology group of $L$ 
over $\R$ coefficient. We put 
$H(L;\Lambda_{0,nov}^{\R}) = H(L;\mathbb R) \otimes_{\R}\Lambda_{0,nov}^{\R}$.
\begin{thm}\label{main1}
$H(L;\Lambda_{0,nov}^{\R})$ has a structure of unital filtered {\bf cyclic} $A_{\infty}$ algebra
which is well-defined up to isomorphism.
\end{thm}
Hereafter we work over $\R$ coefficient. So we write $\Lambda_{0,nov}$ in place of 
$\Lambda_{0,nov}^{\R}$. We denote by $\Lambda_{0,nov}^+$ its maximal ideal.
\par
Let us explain the notions appearing in Theorem \ref{main1}.
We consider a (graded) anti-symmetric inner product $\langle \cdot \rangle$ on $H(L;\mathbb R)[1]$  by
\begin{equation}\label{inpro}
\langle u, v \rangle = (-1)^{\deg u\deg v+\deg u} \int_L u\wedge v
\end{equation}
Filtered $A_{\infty}$ structure 
defines a map
$$
\frak m_k : B_k(H(L;\Lambda_{0,nov})[1]) \to H(L;\Lambda_{0,nov})[1]
$$
of degree one (for $k\ge 0$) such that
\begin{equation}
\sum_{k_1+k_2=k+1}\sum_{i=1}^{k_1}
(-1)^{*}\frak m_{k_1}(\text{\bf x}_1,\ldots,\frak m_2(\text{\bf x}_i,\ldots,\text{\bf x}_{i+k_2-1}),\ldots,\text{\bf x}_k)
= 0.
\end{equation}
$* =\deg \text{\bf x}_1+\ldots+\deg \text{\bf x}_{i-1}+i-1$.
(We require $\frak m_0(1) \equiv 0 \mod \Lambda_{0,nov}^+$.)
\par
The filtered $A_{\infty}$ structure is said to be unital if $\text{\bf e} = 1 \in H^0(L;\R)$ satisfies
\begin{subequations}\label{unitalityissho}
\begin{equation}
\frak m_k(\text{\bf x}_1,\ldots,\text{\bf x}_i,\text{\bf e},\text{\bf x}_{i+2},\ldots,\text{\bf x}_k) = 0
\end{equation}
for $k\ne 2$ and 
\begin{equation}
\frak m_2(\text{\bf e},\text{\bf x}) = (-1)^{\deg \text{\bf x}}\frak m_2(\text{\bf x},\text{\bf e}).
\end{equation}
\end{subequations}
The filtered $A_{\infty}$ structure is said to be cyclically symmetric or cyclic if
\begin{equation}\label{cyclicform}
\langle 
\frak m_k(\text{\bf x}_1,\ldots,\text{\bf x}_k),\text{\bf x}_0
\rangle
= 
(-1)^*\langle 
\frak m_k(\text{\bf x}_0,\text{\bf x}_1,\ldots,\text{\bf x}_{k-1}),\text{\bf x}_{k} 
\rangle
\end{equation}
$* = (\deg \text{\bf x}_0+1)(\deg \text{\bf x}_1+\ldots + \deg \text{\bf x}_k + k)$.
Cyclically symmetric filtered $A_{\infty}$ algebra is said 
cyclic filtered $A_{\infty}$ algebra.
See Remark \ref{misremonequiv} for the definition of isomorphism of 
cyclic filtered $A_{\infty}$ algebra.
\begin{rem}
\begin{enumerate}
\item
Except the statement on cyclicity, Theorem \ref{main1} was 
proved in \cite{FOOO080} Theorem A.
Actually in that case we may take $\Q$ in place of $\R$.
They author does not know how to generalize Theorem  \ref{main1} 
to $\Q$ (or $\Lambda_{0,nov}^{\Q}$) coefficient.
\item The formula (\ref{cyclicform}) is slightly different from Proposition 8.4.8
\cite{FOOO080}. Also the sign in (\ref{inpro}) is different from those explained in Remark 8.4.7 (1) 
\cite{FOOO080}. Actually if we use sign convention in (\ref{inpro}) 
then (\ref{cyclicform}) becomes equivalent to one in Proposition 8.4.8 \cite{FOOO080}. 
(See Lemma \ref{hikakucyclicconv}.)  
The author thanks C.-H  Cho (See \cite{Cho}.) for this remark.
\end{enumerate}
\end{rem}
\par
We next state our result about adic convergence.
We take a basis $\text{\bf e}_i$ of $H(L;\R)$ 
($i=0,\ldots,b$) such that 
$\text{\bf e} = \text{\bf e}_0 = 1 \in H^0(L;\R)$, 
$\text{\bf e}_1,\ldots,\text{\bf e}_{b_1}$ are basis of $H^1(L;\R)$ 
and other $\text{\bf e}_i$'s are basis $H^k(L;\R)$, $k\ge 2$.
For $\text{\bf x} \in H^{odd}(L;\Lambda_{0,nov})$ we put
\begin{equation}
\text{\bf x} = \sum x_i\text{\bf e}_i.
\end{equation}
We introduce $y_i$ ($i=1,\ldots,b_1$) by
\begin{equation}
y_i = e^{x_i} = \sum_{k=0}^{\infty} \frac{1}{k!} x_i^k.
\end{equation}
(See Section \ref{canoalgsec} for a discussion of the convergence of the right hand side.)
We consider 
$$
\sum_{k=0}^N\frak m_k(\text{\bf x},\ldots,\text{\bf x}) = P_N(\text{\bf x}).
$$
We remark that
$
\lim_{N\to \infty} P_N(\text{\bf x})
$
converges in $T$ adic topology if $x_i \in \Lambda_{0,nov}^+$.
We used this fact to define Maurer-Cartan equation and its 
solution (bounding cochain) in \cite{FOOO06}.
We improve it as follows:
\begin{thm}\label{main2}
\begin{enumerate}
\item If $x_i \in \Lambda_{0,nov}$ then $\lim_{N\to \infty} P_N(\text{\bf x})$ 
converges.  We denote its limit by $\frak m(e^{\text{\bf x}})$.
\item  $\frak m(e^{\text{\bf x}})$ depends only on 
$y_1,\ldots, y_{b_1},x_{b_1+1},\ldots,x_b$.
\item 
There exists $\delta > 0$, such that $\frak m(e^{\text{\bf x}})$ 
extends to
\begin{equation}\label{MCteigiiki}
\aligned
\{(x_0,y_1,\ldots, y_{b_1},x_{b_1+1},\ldots,x_b) \mid 
&\,\,\delta > v(y_i) > -\delta, \\
&v(x_0),v(x_{b_1+1}),\ldots,v(x_b)>-\delta
  \}.
\endaligned
\end{equation}
\item Let $\mathcal M(L)_{\delta}$ be the set of $(x_0,y_1,\ldots, y_{b_1},x_{b_1+1},\ldots,x_b)$
in the domain $(\ref{MCteigiiki})$ such that $\frak m(e^{\text{\bf x}})=0$.
Then there exists a family of strict and unital cyclic filtered $A_{\infty}$ algebras 
parametrized by  $\mathcal M(L)_{\delta}$.
\end{enumerate}
\end{thm}
Here $v : \Lambda \to \R$ is a valuation defined by
$$
v\left(\sum a_i T^{\lambda_i}\right) = \inf\{ \lambda_i \mid a_i \ne 0\}.
$$
(Here we assume $a_i \in \R[e,e^{-1}]$ and $\lambda_i \ne \lambda_j$ for $i\ne j$.)
\par
A filtered $A_{\infty}$ algebra is said to be strict if $\frak m_0 = 0$.
\par
In \cite{FOOO06} we considered the case $\text{\bf x} \equiv 0 \mod \Lambda_{0,nov}^+$.
In that case, we defined
\begin{equation}\label{18}
\frak m_k^{\text{\bf x}}(\text{\bf x}_1,\ldots,\text{\bf x}_k)
= 
\sum_{m_0,\ldots,m_k=0}^{\infty}
\frak m_{k + \sum_{i=0}^k m_i}
(\text{\bf x}^{\otimes m_0},\text{\bf x}_1,\ldots,\text{\bf x}_k,\text{\bf x}^{\otimes m_k}).
\end{equation}
Then $(H(L;\Lambda_{0,nov})),\frak m_k^{\text{\bf x}})$ is a filtered $A_{\infty}$
algebra. It is unital or cyclic if $(H(L;\Lambda_{0,nov})),\frak m_k)$ is unital or cyclic, 
respectively. It is strict if and only if  $\frak m(e^{\text{\bf x}})=0$.
\par
Using the techinique of \cite{Cho2}, we can relax the condition $\text{\bf x} \equiv 0 \mod \Lambda_{0,nov}^+$ to 
$\text{\bf x} \equiv 0 \in H^1(L;\Lambda_{0,nov})$.
(See \cite{FOOO08I} where the case when $M$ is toric is discussed.)
\par
Theorem \ref{main2}.4 says that we can further extend this story to the case of 
$\text{\bf x}$ contained in a larger domain.
The convergence of (\ref{18}) is one on $T$-adic topology.
The convergence in Theorem \ref{main2}.1 is different from $T$-adic topology 
and is a mixture of Archimedean and non-Archimedean topology. 
See Definition \ref{mixconv}.
\par
Our main theorems, Theorems \ref{main1} and \ref{main2}, are related as follows.
In Theorem \ref{main2} it is essential that we can change variables from $x_i$ to $y_i = e^{x_i}$.
This becomes possible after we perform  the whole construction of 
Kuranishi structure and its perturbation in a way compatible with 
the forgetful map.
Then, for example, 
$
\frak m_{3,\beta}(\text{\bf e}_i,\text{\bf e}_i,\text{\bf x}) + \frak m_{3,\beta}(\text{\bf e}_i,\text{\bf x},\text{\bf e}_i)
+ \frak m_{3,\beta}(\text{\bf x},\text{\bf e}_i,\text{\bf e}_i)
$
is related to 
$
\frak m_{1,\beta}(\text{\bf x})$
by the formula
\begin{equation}\label{compforetres}
\frak m_{3,\beta}(\text{\bf e}_i,\text{\bf e}_i,\text{\bf x}) + \frak m_{3,\beta}(\text{\bf e}_i,\text{\bf x},\text{\bf e}_i)
+ \frak m_{3,\beta}(\text{\bf x},\text{\bf e}_i,\text{\bf e}_i)
= \frac{1}{2!}(\beta\cap \text{\bf e}_i)^2\frak m_{3,\beta}(\text{\bf x}).
\end{equation}
Here $\text{\bf e}_i$ is a degree 1 cohomology class and 
$\beta$ is an element of $H_2(X;\Z)$. $\frak m_{k,\beta}$ is the contribution of pseudo-holomorphic 
disc of homology class $\beta$ to $\frak m_k$.
(\ref{compforetres}) and a similar formulas make the change of coordinate 
$y_i = e^{x_i}$ possible.
\par
Compatibility with forgetful map (which is the reason why  (\ref{compforetres}) is correct) is
also used to prove cyclic symmetry (\ref{cyclicform}).
It is used also to prove unitality (\ref{unitalityissho}).
\par
Thus the main technical part of this paper is occupied to 
work out the way to construct Kuranishi structure and 
(abstract multivalued continuous family of) perturbations 
on it which is invariant of the process of forgetting boundary marked points.
This construction is performed in Sections \ref{kurareview}-\ref{fgcompfamisec}.
We use it to prove a version of Theorem \ref{main1} (namely a version 
modulo $T^{E}$) in Sections \ref{cyclicmodTEsec}-\ref{cyclicTEderamsec}.
To go from this version to Theorem \ref{main1} we use a similar 
trick as one in \cite{FOOO080} Chapter 7. We need to discuss some 
homological algebra of cyclic filtered $A_{\infty}$ algebra, 
which is in Section \ref{homalgsec}-\ref{canoalgsec}.
In Section  \ref{geoisoto}, we work out one parameter family version of the 
construction of Sections \ref{kurareview}-\ref{fgcompfamisec}, 
which is used to prove well-definedness of the structure up to 
pseudo-isotopy.
Theorem  \ref{main1} then is proved in Section \ref{cycAinfseccoh}.
(The proof of independence of choices is completed in Section \ref{pisoofpisosec}.)
In Section \ref{adicconv} we prove Theorem \ref{main2}.
\par
Our main result of this paper will have two applications. 
\par
One is to define a numerical invariant of (special) Lagrangian submanifold 
in Calabi-Yau 3 fold, which is a rational homology sphere.
Roughly speaking it counts the number of pseudo-holomorphic discs 
with appropriate weights. This is invariant of perturbation and other 
choices but depends on almost complex structure.
Existence of such an invariant was expected by several people especially 
by D. Joyce. For the rigorous construction it is esseitial
to find a perturbation which is compatible with forgetful maps. 
Such a pertrubation is provided in this paper. 
We will use it to define this invariant and discuss its 
properties in \cite{F1}.
\par
The other application is to define a (ridig analytic) family of Floer cohomologies.
Actually the space $\mathcal M(L)_{\delta}$ in Theorem \ref{main2} is a 
chart of certain rigid analytic space. We can glue them up to define appropriate 
rigid analytic space. Then, by extending the construction of Theorem  \ref{main2}
to a pair (or triple etc.) of Lagrangian submanifolds, we can show 
that another Lagrangian submanifold $L'$ gives an 
object of derived category of coherent sheaves on this 
rigid analytic space.
This is a proof of a part of Conjecture U \cite{FOOO080} and 
a step to realize the project to construct homological mirror functor 
by using family of Floer cohomology. (This project was started around 
1998 in \cite{Fuk02}. 
See also \cite{Fuk00}. Its rigid analytic version was first proposed by \cite{KoSo}.)
The construction of rigid analytic family of Floer cohomology will be given 
by extending the construction of this paper in \cite{F2}, there 
its application to mirror symmetry of torus in the form more general than 
 \cite{Fuk02} and \cite{KoSo} will be given. The story will 
 be further generalized to include the case of singular fiber 
in \cite{F3}, in dimension 2.
\par
The author would like to thank Y.-G. Oh, H. Ohta and K. Ono with whom he is 
happy to share may of the ideas of this paper.
\section{Kuranishi structure on the moduli space of pseudo-holomorphic discs: review}
\label{kurareview}
For the purpose of this paper, we need to take Kuranishi structure on 
the moduli space of pseudo-holomorphic discs with some additional 
properties. We will construct such Kuranishi structure in the 
next section. In this section we review definition of Kuranishi structure and 
the construction of it on the moduli space of pseudo-holomorphic discs, 
which was due to \cite{FOOO00}.
\par
Let $\mathcal M$ be a compact space. A Kuranishi chart  
is $(V_{\alpha},E_{\alpha},
\Gamma_{\alpha},\psi_{\alpha},s_{\alpha})$
which satisfies the following:\smallskip
\begin{enumerate}
\item $V_{\alpha}$ is a smooth manifold (with boundaries or corners) and
$\Gamma_{\alpha}$ is a finite group acting effectively on $V_{\alpha}$.
\item $\text{pr}_{\alpha} : E_{\alpha} \to V_{\alpha}$ is a finite dimensional vector bundle on which
$\Gamma_{\alpha}$ acts so that $\text{pr}_{\alpha}$ is $\Gamma_{\alpha}$ equivariant.
\item $s_{\alpha}$ is a $\Gamma_{\alpha}$ equivariant section of $E_{\alpha}$.
\item $\psi_{\alpha} : s_{\alpha}^{-1}(0)/\Gamma_{\alpha} \to \mathcal M$ is a homeomorphism
to its image, which is an open subset.
\end{enumerate}\par
\smallskip
We call $E_{\alpha}$ an {\it obstruction bundle} and $s_{\alpha}$ a {\it Kuranishi map}.
If $\text{\bf p} \in \psi_{\alpha}(s_{\alpha}^{-1}(0)/\Gamma_{\alpha})$, we call 
$(V_{\alpha},E_{\alpha},
\Gamma_{\alpha},\psi_{\alpha},s_{\alpha})$  a {\it Kuranishi neighborhood} of $\text{\bf p}$.
\par
Let $\psi_{\alpha_1}(s_{\alpha_1}^{-1}(0)/\Gamma_{\alpha_1}) 
\cap \psi_{\alpha_2}(s_{\alpha_2}^{-1}(0)/\Gamma_{\alpha_2}) \ne \emptyset$.
A {\it coordinate transformation} from $(V_{\alpha_1},E_{\alpha_1},
\Gamma_{\alpha_1},\psi_{\alpha_1},s_{\alpha_1})$
to  $(V_{\alpha_2},E_{\alpha_2},
\Gamma_{\alpha_2},\psi_{\alpha_2},s_{\alpha_2})$ is 
$(\hat\phi_{\alpha_2\alpha_1},\phi_{\alpha_2\alpha_1},h_{\alpha_2\alpha_1})$
such that
\smallskip 
\begin{enumerate}
\item
$h_{\alpha_2\alpha_1}$ is an injective homomorphism $\Gamma_{\alpha_1} \to \Gamma_{\alpha_2}$.
\par
\item $\phi_{\alpha_2\alpha_1} : V_{\alpha_1\alpha_2} \to V_{\alpha_2}$ is an
$h_{\alpha_2\alpha_1}$ equivariant smooth embedding
from a $\Gamma_{\alpha_1}$ invariant open set $V_{\alpha_1\alpha_2}$ to $V_{\alpha_1}$,
such that the induced map
$\underline{\phi}_{\alpha_2\alpha_1}:
V_{\alpha_1\alpha_2}/\Gamma_{\alpha_1} \to V_{\alpha_2}/\Gamma_{\alpha_2}$ is injective.
\par
\noindent
\item $(\hat\phi_{\alpha_2\alpha_1},\phi_{\alpha_2\alpha_1})$ is an $h_{\alpha_2\alpha_1}$
equivariant embedding of vector bundles $E_{\alpha_1}\vert_{V_{\alpha_1\alpha_2}} \to E_{\alpha_2}$.
\item $\hat\phi_{\alpha_2\alpha_1}\circ s_{\alpha_1} =
s_{\alpha_2}\circ\phi_{\alpha_2\alpha_1}$. 
\item $\psi_{\alpha_1} =
\psi_{\alpha_2} \circ \underline{\phi}_{\alpha_2\alpha_1}$ on $(s_{\alpha_2} ^{-1}(0) \cap V_{\alpha_1\alpha_2})/\Gamma_{\alpha_2} $.
Here $\underline{\phi}_{\alpha_2\alpha_1}$ is as in 2.
\item
The map $h_{\alpha_2\alpha_1}$ restricts to an isomorphism
$(\Gamma_{\alpha_1})_x \to (\Gamma_{\alpha_1})_{\phi_{\alpha_2\alpha_1}(x)}$ 
between isotopy groups, for any 
$x \in V_{\alpha_1\alpha_2}$.
\item
$\psi_{\alpha_1}(s_{\alpha_1}^{-1}(0)/\Gamma_{\alpha_1}) 
\cap \psi_{\alpha_2}(s_{\alpha_2}^{-1}(0)/\Gamma_{\alpha_2}) 
= \psi_{\alpha_1}((s_{\alpha_1}^{-1}(0)\cap V_{\alpha_1\alpha_2})/\Gamma_{\alpha_1})$
\end{enumerate}
\par\medskip
A {\it Kuranishi structure} on $\mathcal M$ assignes 
a Kuranishi neighborhood $(V_{\text{\bf p}},E_{\text{\bf p}},
\Gamma_{\text{\bf p}},\psi_{\text{\bf p}},s_{\text{\bf p}})$ 
to each $\text{\bf p} \in \mathcal M$, such that 
if $\text{\bf q} \in \psi_{\text{\bf p}}(V_{\text{\bf p}}/\Gamma_{\text{\bf p}})$
then there exists 
a coordinate transformation 
$(\hat\phi_{\text{\bf p}\text{\bf q}},\phi_{\text{\bf p}\text{\bf q}},h_{\text{\bf p}\text{\bf q}})$
from 
$(V_{\text{\bf q}},E_{\text{\bf q}},
\Gamma_{\text{\bf q}},\psi_{\text{\bf q}},s_{\text{\bf q}})$ 
to 
$(V_{\text{\bf p}},E_{\text{\bf p}},
\Gamma_{\text{\bf p}},\psi_{\text{\bf p}},s_{\text{\bf p}})$. 
We assume appropriate compatibility conditions 
among coordinate transformations, which we omit 
and refer to \cite{FO}.
\par
Let $\mathcal M$ have a Kuranishi structure.
We consider the normal bundle $N_{\phi_{\text{\bf p}\text{\bf q}}(V_{\text{\bf q}})}V_{\text{\bf p}}$.
We take the
fiber derivative of the Kuranishi map $s_{\text{\bf p}}$ and obtain a homomorphism
$$
d_{\text{fiber}}s_{\text{\bf p}} : N_{\phi_{\text{\bf p}\text{\bf q}}(V_{\text{\bf q}})}V_{\text{\bf p}} \to E_{\text{\bf p}}\vert_{\text{\rm Im}\phi_{\text{\bf p}\text{\bf q}}}
$$
which is an $h_{\psi_{\text{\bf p}\text{\bf q}}}$-equivariant bundle homomorphism.
We say that the space with Kuranishi structure
$\mathcal M$ {\it has the tangent bundle} if $d_{\text{fiber}}s_{\text{\bf p}}$
induces a bundle isomorphism
\begin{equation}\label{identifi}
N_{\phi_{\text{\bf p}\text{\bf q}}(V_{\text{\bf q}})}V_{\text{\bf p}}
\cong \frac{E_{\text{\bf p}}\vert_{\text{\rm Im}\phi_{\text{\bf p}\text{\bf q}}}}
{\hat\phi_{\text{\bf p}\text{\bf q}}(E_{\text{\bf q}})}.
\end{equation}
We call a space with Kuranishi structure which has a tangent bundle a
{\it Kuranishi space} or $K$-space. 
\par
We call a Kuranishi space to be oriented if $V_{\text{\bf p}}$ and $E_{\text{\bf p}}$  are 
oriented and if (\ref{identifi}) is orientation preserving.
\par\medskip
Let $(M,\omega)$ be a symplectic manifold of (real) dimension $2n$.
We take an almost complex structure $J$, which is tamed by $\omega$.
For each $\alpha \in H_2(M;\Z)$ and $\ell \ge 0$, we denote 
by $\mathcal M^{\text{\rm cl}}_{\ell}(\alpha;J)$ the moduli space of 
stable $J$-holomorphic curve of genus zero with $\ell$ marked points and 
of homology class $\alpha$.
(In this paper we use only genus zero pseudo-holomorphic curve.)
We sometimes write it as $\mathcal M^{\text{\rm cl}}_{\ell}(\alpha)$.
It has a Kuranishi structure of dimension 
$
2(n + \ell - 3 + c^1(M)\cap \alpha).
$
(See \cite{FO}.)

Let $L$ be a relatively spin Lagrangian submanifold of $M$.
For each $\beta \in H_2(M,L;\Z)$ and $\ell \ge 0$, $k \ge 0$,
we denote 
by $\mathcal M_{\ell,k}(\beta;L;J)$ the moduli space of 
bordered stable $J$-holomorphic curve of genus zero with $\ell$ interior 
marked points and $k$ boundary marked points, 
one boundary component and
of homology class $\beta$.
We sometimes write it as  $\mathcal M_{\ell,k}(\beta)$.
In case $\ell=0$ we write  $\mathcal M_{k}(\beta)$  also.
It has a Kuranishi structure with corner and boundary 
of dimension 
$
2n + 2\ell  + k - 3 + \mu(\beta),
$
where $\mu: H_2(M,L;\Z) \to \Z$ is a Maslov index.
(See \cite{FOOO00}.)
\begin{rem}
In  \cite{FOOO00} we used compatible almost complex structure. 
However all the constructions there work for tame almost 
complex structure. \end{rem}
We review the construction of these Kuranishi structures since 
we need to modify them so that they have some additional 
properties, in the next section.
\par
Let 
$\text{\bf p} = (\Sigma,v) \in \mathcal M_{\ell,k}(\beta;L;J)$ (resp. 
$(\Sigma,v) \in \mathcal M_{\ell}^{\text{\rm  cl}}(\alpha;J)$).
Here $\Sigma$ is a semi-stable marked bordered Riemann surface of genus $0$ 
with one boundary component
and $v : (\Sigma,\partial\Sigma) \to (M,L)$ is a $J$-holomorphic map.
(resp. $\Sigma$ is a semi-stable marked Riemann surface of genus $0$ 
and $v : \Sigma \to M$ is a $J$-holomorphic map.) 
We assume that the enumeration of the boundary marked points respects the 
counter clockwise cyclic order of the boundary of $\Sigma$.
\begin{rem}
In \cite{FOOO080} we wrote $ \mathcal M^{\text{\rm main}}_{\ell,k}(\beta;L;J)$. 
Here main means the compatibility of the enumeration of the 
boundary marked points with the 
counter clockwise cyclic order of the boundary of $\Sigma$.
We omit this symbol in this paper since we only consider such $\Sigma$.
\end{rem}
We consider decomposition 
\begin{equation}\label{irdec}
\Sigma = \bigcup_{a \in A} \Sigma_a
\end{equation}
to irreducible components. Here $\Sigma$ is either a disc or a sphere component.
(resp. is a sphere component.)
Let $\Gamma_{\text{\bf p}}$ be a finite group 
consisting of bi-holomorphic maps $\varphi : \Sigma \to \Sigma$ such that 
$v\circ \varphi = v$.
\par
We choose an open set $U_a$ for each $\Sigma_a$.
We may choose them so that they are $\Gamma_{\text{\bf p}}$ 
invariant in an obvious sense.
We also assume that the closure $U_a$ does not intersect with 
the boundary, marked points or singular points.
Let $\Lambda^{01}$ be a bundle on the disjoint union $\sqcup \Sigma_a$
of $(0,1)$ forms. 
We are going to take 
a finite dimensional vector spaces 
\begin{equation}\label{whereEa}
E_a \subset C^{\infty}_0(U_a;v^*TM\otimes \Lambda^{01})
\end{equation}
and take its direct sum $E_a^0 = \bigoplus_{a\in A} E_a$ as 
a `tentative choice' of the obstruction bundle.
The explanation of this choice of $E_a$ is in order.
(The right hand side of (\ref{whereEa}) denotes 
the space of compactly supported smooth sections of the pull back bundle 
$v^*TM$ on $U_a$.)
\par
We take a positive integer $p$ and consider 
the space of sections $W^{1,p}(\Sigma_a,v^*TM)$ 
of the pull-back bundle $v^*TM$ on $\Sigma_a$ 
of $W^{1,p}$-class.
We choose $p$ sufficiently large so that 
element of $W^{1,p}(\Sigma_a,v^*TM)$ 
is continuous on $\Sigma_a$.
In case when $\Sigma_a$ has a boundary 
we put boundary condition 
\begin{equation}\label{bdrycondxi}
\xi \vert_{\partial \Sigma_a} \subset  W^{1-1/p,p}(\partial \Sigma_a,v^*TL)
\end{equation}
and denote by $W^{1,p}(\Sigma_a,v^*TM;v^*TL)$ the 
space of $\xi $ satisfying (\ref{bdrycondxi}).
By taking $p$ large we may assume the restriction to 
$\partial \Sigma_a$ is also continuous.
We write 
$W^{1,p}(\Sigma_a,v^*TM;v^*TL) 
= W^{1,p}(\Sigma_a,v^*TM)$ in case $\Sigma_a$ is a sphere component.
\par
We now consider the sum 
\begin{equation}\label{sumW1pa}
\bigoplus_a W^{1,p}(\Sigma_a,v^*TM;v^*TL).
\end{equation}
We require the following additional condition at the singular 
points. Let $p$ be a singular point. 
There exist $\Sigma_{a_1}$ and $\Sigma_{a_2}$
($a_1 \ne a_2$) with $p \in \Sigma_{a_1} \cap \Sigma_{a_2}$.
(Using the fact that $\Sigma$ has only nodal point as a singularity 
and that genus is $0$, there exist such $a_1, a_2$ uniquely.)
We now require 
$(\xi_a) \in \bigoplus_a W^{1,p}(\Sigma_a,v^*TM;v^*TL)$
to satisfy
\begin{equation}\label{compatsing}
\xi_{a_1}(p) = \xi_{a_2}(p).
\end{equation}
We  denote by
\begin{equation}
W^{1,p}(\Sigma,v^*TM;v^*TL) 
\end{equation}
the subspace of (\ref{sumW1pa}) of the elements $(\xi_a)$ 
satisfying (\ref{compatsing}) at every singular point $p$.
\par
Let $\Lambda^{01}$ be a bundle on $\cup \Sigma_a$
of $(0,1)$ forms. We put
\begin{equation}
L^{p}(\Sigma,v^*TM\otimes\Lambda^{01}) 
= \bigoplus_a L^{p}(\Sigma_a,v^*TM\otimes\Lambda^{01}).
\end{equation}
The linearization of the pseudo-holomorphic curve equation 
defines an operator
\begin{equation}\label{lineq}
D_v\overline{\partial} 
: W^{1,p}(\Sigma,v^*TM;v^*TL) \to L^{p}(\Sigma,v^*TM) .
\end{equation}
(\ref{lineq}) is a Fredholm operator.
\begin{dfn}\label{Fredreg1}
We say $(\Sigma,v)$ is {\it Fredholm regular} if (\ref{lineq}) is surjective.
\end{dfn}
In case $(\Sigma,v)$ is Fredholm regular we can take its neighborhood 
in $\mathcal M_{\ell,k}(\beta;L;J)$ (resp. $\mathcal M_{\ell}(\alpha;J)$) so that it is 
an orbifold with boundary/corner (resp. orbifold).
In general we need to take obstruction bundle $E_{\text{\bf p}}$.
We consider the finite dimenional subspaces $E_a$ as in (\ref{whereEa}).
We require that $ \bigoplus_a E_a$ is $\Gamma_{\text{\bf p}}$
invariant.
\begin{dfn}\label{Fredreg2}
We say that $((\Sigma,v),(E_a))$ is {\it Fredholm regular} if 
\begin{equation}
\text{\rm Im}\,D_v\overline{\partial} 
+ \bigoplus_a E_a
= L^{p}(\Sigma,v^*TM).
\end{equation}
\end{dfn}
Under this assumption we will construct a Kuranishi 
neighborhood of $\text{\bf p}$.
We introduce the following notation 
(Definition \ref{defcompnentwise}) for convenience.
We attach each of the tree of sphere components 
to the disc component to which the component is rooted. 
Let
\begin{equation}\label{discdec}
\Sigma = \bigcup_{b \in B}\Sigma_b = \bigcup_{b \in B}\bigcup_{a \in A_b} \Sigma_a
\end{equation}
be the resulting decomposition.
(Here $\bigcup_{b}A_b = A$.) 
\begin{dfn}\label{defcompnentwise}
The choice of $(E_a)$ is said to be 
{\it component-wise} if $E_a$ depends only 
on $(\Sigma_a,v\vert_{\Sigma_a})$ and 
is independent of other components or the restriction of 
$v$ to other components.
\par
The choice of $(E_a)$ is said to be 
{\it disc-component-wise} if $E_a$ depends only 
on $(\Sigma_{a'},v\vert_{\Sigma_{a'}})$ 
with $a,a' \in A_b$ so some $b$, and 
is inependent of the components $(\Sigma_{a'},v\vert_{\Sigma_{a'}})$ 
with $a' \notin A_b$ where $a \in A_b$.
\end{dfn}
Before reviewing the construction of Kuranishi neighborhood 
for Fredholm regular $((\Sigma,v),(E_a))$ we need to 
discuss the stability of the domain. 
We here follow \cite{FO} appendix.
Let $\Sigma_a$ be a component of $\Sigma$.
We remark that we include marked or singular points 
to the notation $\Sigma_a$. Namely 
$\Sigma_a$ is a marked disc or marked sphere 
where its marked points are either singular or marked 
point of $\Sigma$ which is on $\Sigma_a$.
We say that $\Sigma_a$ is {\it stable} if 
its automorphism group is of finite order.
(In case of disc component it is equivalent to 
$k_a + 2\ell_a \ge 3$ where $k_a$ is the number of 
boundary marked points and $\ell_a$ is the number of 
interior marked points.
In case of sphere component it is equivalent to 
$\ell_a \ge 3$.)
\par
If $\Sigma_a$ is not stable we add some interior marked points 
and make it stable. We denote it by $\Sigma^+_a$.
Note that $v\vert_{\Sigma_a}$ is nontrivial 
(in case $\Sigma_a$ is unstable), 
since $(\Sigma,v)$ is stable.
Therefore $v\vert_{\Sigma_a}$ 
is an immersion at the generic point.
We assume that $v$ is an immersion at the addtional marked points. 
We glue them and obtain $\Sigma^+$.
Let $\text{\rm mkadd}(\Sigma^+)$ be the 
set of the marked points we add.
\par
We require that $\text{\rm mkadd}(\Sigma^+)$ 
are invariant of the $\Gamma_{\text{\bf p}}$ action.
Namely we assume that none of the  nontrivial element of 
$\Gamma_{\text{\bf p}}$ fix any of $p \in \text{\rm mkadd}(\Sigma^+)$
and $\Gamma_{\text{\bf p}}$ exchanges elements of 
$\text{\rm mkadd}(\Sigma^+)$.
\par
For each added marked point $p \in \text{\rm mkadd}(\Sigma^+)$
we take $2n-2$ dimensional submanifold $XM_p \subset M$ 
such that $XM_p \subset M$ intersect with 
$v : \Sigma_a \to M$ transversally at $v(p)$.
(Here $\Sigma_a$ is a component containing $p$.)
We also require that if $\gamma \in \Gamma_{\text{\bf p}}$ then 
$XM_p = XM_{\gamma(p)}$.
\par
Now we start with $(\Sigma,v)$ and consider $(\Sigma^+,v)$.
We take $E(\Sigma^+,v) = \bigoplus E_{a}(\Sigma^+,v)$  such that 
$((\Sigma^+,v),E(\Sigma^+,v))$ is Fredholm regular.
We remark that all the components of $\Sigma^+$ are now stable.
Let $U(\text{\rm defresolv};\Sigma^+)$ be a neighborhood of $\Sigma^+$ in 
the moduli space of marked bordered Riemann surface of genus zero and 
with one boundary component
(resp. marked Riemann surface of genus $0$).
It is an orbifold with boundary or corners.
We write
$$
U(\text{\rm defresolv};\Sigma^+) = V(\text{\rm defresolv};\Sigma^+)/\text{Aut}(\Sigma^+)
$$
where $V(\text{\rm defresolv};\Sigma^+)$ is an manifold with boundary or corner and 
$\text{Aut}(\Sigma^+)$ is a finite group of automorphism of $\Sigma^+$.
By our choice of $\text{\rm mkadd}(\Sigma^+)$ the group 
$\Gamma_{\text{\bf p}}$ of automorphisms of $(\Sigma,v)$ is a subgroup of 
$\text{Aut}(\Sigma^+)$.
\par
Let $\frak v \in  V(\text{\rm defresolv};\Sigma^+)$ and 
$\Sigma^+(\frak v)$ be the corresponding marked bordered Riemann surface of 
genus $0$. 
$\Sigma^+(\frak v)$ minus small neighborhood of `neck region' is canonically isomorphic to 
$\Sigma^+$ minus small neighborhood of singular point. (See \cite{FO} Section 12.)
Therefore the support $U_a$ of $E_a$ for each $E_a = E_a(\Sigma^+,v)$ has 
a canonical embedding to $\Sigma^+(\frak v)$.
(Note we use here the fact  $\frak v$ is an element of $V(\text{\rm defresolv};\Sigma^+)$ 
not just an element of $U(\text{\rm defresolv};\Sigma^+)$.)
We consider $v' : (\Sigma^+(\frak v),\partial \Sigma^+(\frak v) \to (M,L)$ 
which is $C^0$ close to $v$. 
\par
To define this $C^0$ closed-ness precisely we proceed as follows.
We decompose
$$
\Sigma^+ = \Sigma^+_{\text{reg}} \cup \Sigma^+_{\text{sin}}
$$
where the second term is a small neighborhood ($\epsilon$-neighborhood)  of the singular point.
We then decompose
$$
\Sigma^+(\frak v) = \Sigma^+_{\text{reg}}(\frak v) \cup \Sigma^+_{\text{sin}}(\frak v)
$$
where $\Sigma^+_{\text{reg}}(\frak v)$ is biholomorphic to $\Sigma^+_{\text{reg}}$.
(For each $\epsilon$ we may take $V(\text{\rm defresolv};\Sigma^+)$ small 
enough so that such decomposition exists.)
We may assume $U_a(\Sigma^+,v) \subset \Sigma^+_{\text{reg}}$ and 
hence $U_a(\Sigma^+,v) \subset \Sigma^+_{\text{reg}}(\frak v)$.
Now $v' $ is said to be $\epsilon$ close to $v$ if:
\smallskip
\begin{enumerate}
\item $dist(v(x),v'(x)) \le \epsilon$, if $x \in  \Sigma^+_{\text{reg}}= \Sigma^+_{\text{reg}}(\frak v)$.
\item The diameter of the $v'$  image of each connected component of $\Sigma^+_{\text{sin}}(\frak v)$
is smaller than $\epsilon$.
\end{enumerate}
\par\smallskip
Now for $v'$ which is $\epsilon$ close to $v$ we define 
$E(\Sigma^+(\frak v),v')$ as follows.
For $U_a$ we consider the isomorphism
\begin{equation}\label{Obstpara}
C^{\infty}(U_a;v^*TM) \cong C^{\infty}(U_a;{v'}^*TM)
\end{equation}
by taking parallel transport along the minimal geodesic joining $v(x)$ to $v'(x)$.
(If $\epsilon$ is smaller than injectivity radius of $M$ such minimal geodesic exists 
uniquely.)
Using (\ref{Obstpara}) we regard
$E_a(\Sigma^+,v)$ as a subspace of $L^p(\Sigma^+(\frak v);{v'}^*TM)$.
Now we consider
\begin{equation}
\aligned
V^+(\text{\bf p}) = \{(\frak v,v') 
\mid v' : (\Sigma^+(\frak v),\partial \Sigma^+(\frak v) \to (M,L) 
&\text{  is $\epsilon$ close to $v$} 
\\
&\hskip-1cm\text{and satisfies (\ref{perturbeq})  below.}\}
\endaligned\end{equation}
\begin{equation}\label{perturbeq}
\overline{\partial} v' \in \bigoplus_a E_a(\Sigma^+,v).
\end{equation}
The following is a consequence of standard glueing analysis. (See \cite{FOOO00}.)
\begin{prp}\label{gluing}
If $((\Sigma^+,v),E(\Sigma^+,v))$ is Fredholm regular and $\epsilon$ is sufficiently small
then $V^+(\text{\bf p})$ is a smooth manifold with boundary and corner.
\end{prp}
We next define evaluation map at added marked point.
Namely we define
\begin{equation}\label{addeval}
ev^{\text{int,add}} : V^+(\text{\bf p}) \to M^{\#(\text{\rm mkadd}(\Sigma^+))}
\end{equation}
by
$$
(\frak v,v') \mapsto (v'(p_1),\ldots,v'(p_{\text{\rm mkadd}(\Sigma^+)})).
$$
Using the fact that 
$v$ is immersion and is transversal to $XM_p$ for each 
$p\in \text{\rm mkadd}(\Sigma^+)$ it follows that the fiber produce
\begin{equation}
V_{\text{\bf p}} = ev^{\text{int,add}} {}_{ev^{\text{int,add}}}\times_{M^{\#(\text{\rm mkadd}(\Sigma^+))}} 
\prod_{p\in \text{\rm mkadd}(\Sigma^+)} XM_{p}
\end{equation}
is transversal. In particular $V_{\text{\bf p}}$ is a smooth manifold with corners.
We define $E_{\text{\bf p}}$ by 
$$
E_{\text{\bf p}}(\frak v,v') = \bigoplus_a E_a(\Sigma^+,v)
$$
using the isomorphism (\ref{Obstpara}).
The section (Kuranishi map) 
$s_{\text{\bf p}}$ is defined by
$$
s_{\text{\bf p}}(\frak v,v') = \overline{\partial} v' \in E_{\text{\bf p}}(\frak v,v').
$$
They are $\Gamma_{\text{\bf p}}$ equivariant by construction.
\par
Let $s_{\text{\bf p}}(\frak v,v')= 0$. Then $v' : 
(\Sigma^+(\frak v),\partial \Sigma^+(\frak v) \to (M,L)$ is pseudo-holomorphic.
We forget added marked point and obtain 
$\overline v' : (\Sigma(\frak v),\partial \Sigma(\frak v)) \to (M,L)$.
We put
$$
\psi_{\text{\bf p}}(\frak v,v') = (\Sigma(\frak v),\overline v') 
\in \mathcal M_{\ell,k}(\beta).
$$
We thus described a construction of Kuranishi neighborhood for 
each given choice of $E_a(\Sigma^+,v)$, $\text{\rm mkadd}(\Sigma^+)$.
We next review how to glue them.
\par
For each element $\text{\bf p} \in  \mathcal M_{\ell,k}(\beta)$
we fix $E_{\text{\bf p}}$ and $\text{\rm mkadd}(\text{\bf p})$.
We write them as $E^0_{\text{\bf p}}$ and $\text{\rm mkadd}^0(\text{\bf p})$ 
since later we will change them. We also take sufficiently small $\epsilon$
depending on  $\text{\bf p}$ and 
construct a (tentative) Kuranishi neighborhood, 
which we denote by 
$(V^0_{\text{\bf p}},E^0_{\text{\bf p}},\Gamma^0_{\text{\bf p}},s^0_{\text{\bf p}},\psi^0_{\text{\bf p}})$.
We consider the covering
$$
\mathcal M_{\ell,k}(\beta;L;J) 
= \bigcup_{\text{\bf p}}\psi^0_{\text{\bf p}}((s^0_{\text{\bf p}})^{-1}(0))/\Gamma^0_{\text{\bf p}}).
$$
We now take a finite set $\text{\bf p}_{\frak a}$, $\frak a \in \frak A$ and closed subset 
$W(\text{\bf p}_{\frak a})$ of 
$\psi^0_{\text{\bf p}_{\frak a}}((s^0_{\text{\bf p}_{\frak a}})^{-1}(0))/\Gamma^0_{\text{\bf p}_{\frak a}})$
such that 
$$
\bigcup_{\frak a \in \frak A}\text{\rm Int}\,W(\text{\bf p}_{\frak a}) = \mathcal M_{\ell,k}(\beta;L;J) .
$$
Now for each $\text{\bf p} = (\Sigma,v) \in \mathcal M_{\ell,k}(\beta;L;J)$ we choose 
$E_{\text{\bf p}}$ and $\text{\rm mkadd}(\text{\bf p})$ as follows.
We put 
$$
\frak A(\text{\bf p}) = \{ \frak a \in \frak A \mid \text{\bf p}  
\in W(\text{\bf p}_{\frak a})\}.
$$
then 
\begin{equation}
\text{\rm mkadd}(\text{\bf p}) 
= \bigcup_{\frak a \in \frak A(\text{\bf p})} \text{\rm mkadd}^0(\text{\bf p}_{\frak a}) 
\end{equation}
and
\begin{equation}
E_{\text{\bf p}} = \bigoplus_{\frak a \in \frak A(\text{\bf p})} E^0{\text{\bf p}}.
\end{equation}
We remark that we may regard
$\text{\rm mkadd}^0(\text{\bf p}_{\frak a})  \subset \Sigma
$, 
and 
$
E^0_{\text{\bf p}} \subset L^{p}(\Sigma,v^*TM\otimes\Lambda^{01}) 
$
by using $\psi_{\text{\bf p}_{\frak a}}$.
We can use them to define 
Kuranishi neighborhood 
$(V_{\text{\bf p}},E_{\text{\bf p}},\Gamma_{\text{\bf p}},s_{\text{\bf p}},\psi_{\text{\bf p}})$ 
of each point in $\mathcal M_{\ell,k}(\beta;L;J)$.
\par
Using the closed-ness of the set $W(\text{\bf p}_{\frak a})$ we can prove that 
if $\text{\bf q} \in \psi_{\text{\bf p}}((s_{\text{\bf p}})^{-1}(0))/\Gamma_{\text{\bf p}})$ then
$
E_{\text{\bf q}} \subseteq E_{\text{\bf p}}
$
and 
$\text{\rm mkadd}^0(\text{\bf q}) 
\subseteq \text{\rm mkadd}^0(\text{\bf p})$.
We can use it to construct coordinate change and then 
obtain Kuranishi structure.
\par
We remark that here we construct Kuranishi structure on $\mathcal M_{\ell,k}(\beta;L;J)$ 
for each of $\ell,k,\beta$ individually. We actually 
need to construct them so that they are related to each other at their boundaries.
In next section we will do it, in a way slightly different from \cite{FOOO080}.
Namely in \cite{FOOO080} the constructions are not compatible with 
forgetful map.  We will modify it so that it is compatible with the forgetful map 
of boundary marked points. For this purpose we include $k=0$ (the case 
of no boundary marked points.)
This is the technical heart of the whole construction of this paper.
\par
We will use the following notions in the next section.
\begin{dfn}\label{stcontwksub}
Let $\mathcal M$ be a Kuranishi space and $N$ be a smooth manifold.
A {\it strongly continuous smooth} map $f : \mathcal M \to N$ is a 
family $f = \{f_{\text{\bf p}}\}$ of $\Gamma_{\text{\bf p}}$ equivariant smooth maps 
$
f_{\text{\bf p}} : V_{\text{\bf p}} \to N
$
which induces 
$
f_{\text{\bf p}} : V_{\text{\bf p}}/\Gamma_{\text{\bf p}} \to N
$
and such that
$
f_{\text{\bf p}} \circ \overline{\varphi}_{\text{\bf p}\text{\bf q}}
= f_{\text{\bf q}}
$
on $V_{\text{\bf q}\text{\bf p}}/\Gamma_{\text{\bf q}}$.
\par
We say the  $\{f_{\text{\bf p}}\}$ is {\it weakly submersive} 
if each of $f_{\text{\bf p}}$ is 
a submersion.
\end{dfn}
In case Kuranishi space has boundary or corner, 
for the map to be submersive, we require that the restriction 
to each stratum is submersive.
\par
We consider the moduli spaces $\mathcal M_{\ell,k}(\beta;L;J)$
and $\mathcal M_{\ell}^{\text{\rm cl}}(\alpha;J)$. 
The evaluation at marked points induces maps
$$
\aligned
&ev = (ev^{\text{\rm int}},ev) = ((ev^{\text{\rm int}}_1,\ldots,ev^{\text{\rm int}}_{\ell}),(ev_0,\ldots,ev_{k-1})) : 
\mathcal M_{\ell,k}(\beta) \to M^{\ell} \times L^k, \\
& ev = ev^{\text{\rm int}} = (ev^{\text{\rm int}}_1,\ldots,ev^{\text{\rm int}}_{\ell}) : 
\mathcal M_{\ell}^{\text{\rm cl}}(\tilde\beta) \to M^{\ell}.
\endaligned$$
These maps are strongly continuous. In \cite{FOOO080} the Kuranishi structure is chosen 
so that they are weakly submersive. 
In the next section we do {\it not} choose so. 
We will explain the reason in Remark \ref{remwksubmer}.
\par
We next review fiber product of Kuranishi structures.
Let $\mathcal M_1$, $\mathcal M_2$ be Kuranishi spaces and 
$f = \{f_{\text{\bf p}}\}$,  $g = \{g_{\text{\bf p}'}\}$ be strongly 
continuous maps from them to a manifold $N$.
We consider 
\begin{equation}\label{fiberprodform}
\mathcal M_1 \,{}_{f}\times_g \mathcal M_2 
= \{ (x,y) \in \mathcal M_1 \times \mathcal M_2 
\mid f(x) = g(y)\}.
\end{equation}
\begin{lmm}\label{kurafiber}
If either $f$ or $g$ are weakly submersive then the fiber product 
$(\ref{fiberprodform})$ has a Kuranishi structure.
\end{lmm}
See \cite{FOOO080} Section A1.3 for its proof. 
Let $h = \{h_{\text{\bf p}'}\} : \mathcal M_2 \to N'$ be another strongly 
continuous map. 
It induces a strongly continuous map 
\begin{equation}\label{fiberprodformh}
\mathcal M_1 \,{}_{f}\times_g \mathcal M_2 \to N'
\end{equation}
in an obvious way.
\begin{lmm}\label{kurafibermap}
If $f$ and $h$ are weakly submersive then the map 
$(\ref{fiberprodformh})$ is also weakly submersive.
\end{lmm}
The proof is immediate from the corresponding statement 
for the submersion of smooth maps between manifolds 
and from the construction of fiber product in \cite{FOOO080}.
\section{Forgetfulmap compatible Kuranishi structure}
\label{Seckuranishiforget}
Let $(M,\omega)$ be a symplectic manifold and $L$ its Lagrangian submanifold.
For $\beta \in H_2(M,L;\Z)$, we defined the moduli space $\mathcal M_{\ell,k}(\beta)$
as in Section \ref{kurareview}.
In Section \ref{kurareview}, we wrote $\mathcal M_{\ell,k}(\beta;L;J)$ or 
$\mathcal M_{\ell}^{\text{\rm cl}}(\alpha;J)$. 
Hereafter we omit $J$ from notation when no confusion can occur.
\par
We include the case $k=0$. In that case we compactify it by adding 
$\mathcal M^{\text{\rm cl}}_{1+\ell}(\tilde{\beta}) \times L$ if $\partial\beta = 0$ as is explained in 
\cite{FOOO080} Subsection 7.4.1. Namely we have an embedding
\begin{equation}\label{closeopencllapse}
\mathfrak{clop} : \mathcal M^{\text{\rm cl}}_{1+\ell}(\tilde{\beta}) \,{}_{ev^{\text{\rm int}}_0}\times_M L 
\to  \mathcal M_{\ell,0}(\beta).
\end{equation}
Note $\mathcal M^{\text{\rm cl}}_{1+\ell}(\tilde{\beta})$ is a moduli space 
of pseudo-holomorphic map from genus 0 stable map (without boundary).
Here $\tilde{\beta} \in H_2(X;\Z)$ is a class which goes to the class 
$\beta$ by the natural homomorphism $H_2(X;\Z) \to H_2(X,L;\Z)$.
We fix compatible system of Kuranishi structures on them so that evaluation maps are 
submersion. And we will not change them.
\par
We have a forgetful map
\begin{equation}
\mathfrak{forget} : \mathcal M_{\ell,1}(\beta) \to \mathcal M_{\ell,0}(\beta).
\end{equation} 
We also consider the embedding
\begin{equation}\label{gluemap}
\mathfrak{glue} : \mathcal M_{\ell_1,1}(\beta_1)\, {}_{ev_0} \times_{ev_0} \mathcal M_{\ell_2,1}(\beta_2)
\to \partial\mathcal M_{\ell_1+\ell_2,0}(\beta_1+\beta_2).
\end{equation}
\begin{dfn}\label{forgetcomp}
Kuranishi structures of $\mathcal M_{\ell,1}(\beta)$ and of $\mathcal M_{\ell,0}(\beta)$ are 
said to be {\it forgetfulmap compatible} to each other if the following holds.
\par
Let $\tilde{\text{\bf p}} = [(\Sigma,z_0,\vec z\,^{\text{int}}),u] \in \mathcal M_{\ell,1}(\beta)$ and 
$\text{\bf p} =\mathfrak{forget}([(\Sigma,z_0,\vec z\,^{\text{int}}),u]) = [(\Sigma,\vec z\,^{\text{int}}),u] \in \mathcal M_{\ell,0}(\beta)$.
We first consider the case when $\text{\bf p}$ is not in the image of (\ref{closeopencllapse}).
\par
Let $(V_{\tilde{\text{\bf p}}},E_{\tilde{\text{\bf p}}},\Gamma_{\tilde{\text{\bf p}}},
\psi_{\tilde{\text{\bf p}}},s_{\tilde{\text{\bf p}}})$ and 
$(V_{\text{\bf p}},E_{{\text{\bf p}}},\Gamma_{{\text{\bf p}}},
\psi_{\text{\bf p}},s_{{\text{\bf p}}})$ be 
a Kuranishi neighborhoods of them, respectively. Then:
\par\smallskip
\begin{enumerate}
\item $V_{\tilde{\text{\bf p}}} = V_{\text{\bf p}} \times (0,1)$.
\item $E_{\tilde{\text{\bf p}}} = E_{\text{\bf p}} \times (0,1)$.
\item $\Gamma_{\tilde{\text{\bf p}}} = \Gamma_{{\text{\bf p}}}$. 
The action of  $\Gamma_{\tilde{\text{\bf p}}}$ preserves identifications 
given in 1 and 2, where the action to the factor $(0,1)$ is trivial.
\item $s_{\tilde{\text{\bf p}}}(x,t) = (s_{{\text{\bf p}}}(x),t)$ by the identification given in 1 and 2.
\item $\mathfrak{forget} \circ \psi_{\tilde{\text{\bf p}}}$ coincides 
with the composition of $\psi_{{\text{\bf p}}}$ and the projection to the 
first factor.
\end{enumerate}
\par\smallskip
We next consider the case when $\text{\bf p}$ is in the image of (\ref{closeopencllapse}).
This implies that on
the disc component $\Sigma_0$ containing $z_0$, 
$u$ is a constant map, and that $\Sigma_0$ has exactly one 
singular point and has no interior or boundary marked points other than $z_0$.
Let $\overline{\Sigma}$ be a closed semistable curve of genus $0$ 
without boundary, which is obtained 
by removing $\Sigma_0$ from $\Sigma$.
Let $z_0^{\text{int}}$ be the point $\overline{\Sigma} \cap \Sigma_0 
\subset \overline{\Sigma}$.
We put $\vec z_+^{\,\,\text{int}} = (z_0^{\text{int}},\vec z\,^{\text{int}})$.
Then $(((\overline{\Sigma},\vec z_+^{\,\,\text{int}},u),u(z_0^{\text{int}}))$ is an element of 
$\mathcal M^{\text{\rm cl}}_{1+\ell}(\tilde{\beta}) \times L$ such that
$$
\mathfrak{clop}(((\overline{\Sigma},\vec z_+^{\,\,\text{int}},u),u(z_0^{\text{int}}))
= \text{\bf p}.
$$
\par
Let $(V_{\tilde{\text{\bf p}}},E_{\tilde{\text{\bf p}}},\Gamma_{\tilde{\text{\bf p}}},\psi_{\tilde{\text{\bf p}}},s_{\tilde{\text{\bf p}}})$ be 
a Kuranishi neighborhood of $\tilde{\text{\bf p}}$.
\par
Let $\overline{\text{\bf p}} = [\overline{\Sigma},\vec z_+^{\,\,\text{int}},u]$ and 
$(V_{\overline{\text{\bf p}}},E_{\overline{\text{\bf p}}},\Gamma_{\overline{\text{\bf p}}},
\psi_{\overline{\text{\bf p}}},s_{\overline{\text{\bf p}}})$
be its Kuranishi neighborhood in $\mathcal M^{\text{\rm cl}}_{1+\ell}(\tilde{\beta})$.
Since $ev_{0}^{\text{\rm int}} : \mathcal M^{\text{\rm cl}}_{1+\ell}(\tilde{\beta}) \to M$ 
is strogly continuous and weakly submersive, it induces a submersion
$
ev_{0}^{\text{\rm int}} : V_{\tilde{\text{\bf p}}} \to M.
$
Let
$
V_{\tilde{\text{\bf p}}} \cap L = (ev_{0}^{\text{\rm int}})^{-1}(L) \subset V_{\tilde{\text{\bf p}}}.
$
The fiber product Kuranishi structure on 
$\mathcal M^{\text{\rm cl}}_{1+\ell}(\tilde{\beta}) \,{}_{ev^{\text{\rm int}}_0}\times_M L$
is 
$(V_{\overline{\text{\bf p}}}\cap L,E_{\overline{\text{\bf p}}}\cap L,\Gamma_{\overline{\text{\bf p}}},
\psi_{\overline{\text{\bf p}}},s_{\overline{\text{\bf p}}})$.
(Here we write $E_{\overline{\text{\bf p}}} \cap L$ for the restrictionof 
$E_{\overline{\text{\bf p}}}$  to 
$V_{\overline{\text{\bf p}}}\cap L$ by abuse of notation. We also 
write $s_{\overline{\text{\bf p}}}$ or $\psi_{\overline{\text{\bf p}}}$ for its appropriate restriction.)
\smallskip
\begin{enumerate}
\item $V_{\tilde{\text{\bf p}}} = (V_{\overline{\text{\bf p}}}\cap L) \times [0,1)$.
\item $E_{\tilde{\text{\bf p}}} = E_{\overline{\text{\bf p}}}\cap L \times (0,1)$.
\item $\Gamma_{\tilde{\text{\bf p}}} = \Gamma_{\overline{\text{\bf p}}}$. 
The action of  $\Gamma_{\tilde{\text{\bf p}}}$ preserves identifications 
given in 1 and 2, where the action to the factor $(0,1)$ is trivial.
\item $s_{\tilde{\text{\bf p}}}(x,t) = (s_{{\text{\bf p}}}(x),t)$ by the identification given in 1 and 2.
\item $\mathfrak{forget} \circ \psi_{\tilde{\text{\bf p}}}$ coincides with 
the composition of $\psi_{\overline{\text{\bf p}}}$
and the projection to the 
first factor.
\end{enumerate}
\end{dfn}
The main result of this section is:
\begin{thm}\label{forgetcompKUra}
There exists a system of Kuranishi structures on $\mathcal M_{\ell,1}(\beta)$ and 
$\mathcal M_{\ell,0}(\beta)$,
such that: 
\smallskip
\begin{enumerate}
\item They are forgetfulmap compatible in the sense of Definition 
\ref{forgetcomp}.
\item $ev_0 : \mathcal M_{\ell,1}(\beta) \to L$ is strongly submersive.
\item They are compatible with $(\ref{gluemap})$. Namely the fiber 
product Kuranishi structure on 
$\mathcal M_{\ell_1,1}(\beta_1)\, {}_{ev_0} \times_{ev_0} \mathcal M_{\ell_2,1}(\beta_2)$
(which is well-defined by $2$) coincides with the pull back of the 
$\partial\mathcal M_{\ell_1+\ell_2,0}(\beta_1+\beta_2)$ by $\mathfrak{glue}$.
\item They are component-wise in the sense of Definition \ref{defcompnentwise}.
\end{enumerate}
\end{thm}
\begin{proof}
The most essential part of the proof is the following lemma.
Let $\tilde{\text{\bf p}} = [(\Sigma,z_0,\vec z\,^{\text{int}}),u] \in \mathcal M_{\ell,1}(\beta)$ and 
$\text{\bf p} =\mathfrak{forget}([(\Sigma,z_0,\vec z\,^{\text{int}}),u]) = [(\Sigma,\vec z\,^{\text{int}}),u] \in \mathcal M_{\ell,0}(\beta)$.
We consider the case when $\text{\bf p}$ is not in the image of (\ref{closeopencllapse}).
\par
We consider the linearization of the pseudo-holomorphic curve equation (\ref{lineq}).
We consider the case $\Sigma = D^2$.
\begin{lmm}\label{ateachpointkuracomp}
For any open subset $U$ of $\text{\rm Int}\, D^2$, 
there exists a finite dimensional linear subspace $E(u)$ of 
sections of $u^*TM\otimes \Lambda^{01}$ such that 
the following holds.
\begin{enumerate}
\item Each of elements of $E(u)$ is smooth and supported in $U$.
\item We put
\begin{equation}
K(u) = (D_u\overline{\partial})^{-1}(E(u)).
\end{equation}
Then for {\bf any} $z_0 \in \partial D^2$ the map
$
Ev_{z_0} :  K(u) \to T_{u(z_0)}L
$
defined by
\begin{equation}\label{evalatz0}
Ev_{z_0}(v) = v(z_0),
\end{equation}
is surjective.
\end{enumerate}
\end{lmm}
We remark that elements of $K(u)$ are smooth by elliptic regularity. 
Therefore (\ref{evalatz0}) is well-defined.
\begin{proof}
For each {\it fixed} $z_0$ we can find $E(u;z_0)$ satisfying 1 and 
such that for 
$
K(u;z_0) = (D_u\overline{\partial})^{-1}(E(u;z_0))
$ 
the map
$
Ev_{z_0;z_0} :  K(u;z_0) \to T_{u(z_0)}L
$
defined by (\ref{evalatz0}) is surjective.
(This is a consequence of unique continuation. See \cite{FOOO080}.)
\par
Then there exists a neighborhood $W(z_0)$ of $z_0$ in 
$\partial D^2$ such that if $z \in W(z_0)$ then the map
$
Ev_{z_0;z} :  K(u;z_0) \to T_{u(z)}L
$
defined by $Ev_{z_0;z}(v) = v(z)$ is a submersion.
We cover $\partial D^2$ by finitely many $W(z_0)$ 
say $W(z_i)$, $i=1,\ldots,N$. Then 
$E(u) = \bigoplus_{i=1}^N E(u;z_i)$ has the required property.
\end{proof}
We next consider the case of the boundary point 
corresponding to 
$\mathcal M_{\ell+1}^{\text{\rm cl}}(\tilde \beta) \,{}_{ev^{\text{\rm int}}_1}\times_M\, L$.
For this purpose, we need to take Kuranishi structure on 
the moduli space of pseudo-holomorphic sphere.
\begin{lmm}\label{closedmoduliKura}
There exists a system of Kuranishi structures on 
$\mathcal M_{\ell}^{\text{\rm cl}}(\alpha)$ for various $\ell \ge 0$ and 
$\alpha \in \pi_2(M)$ with the following properties.
\par\smallskip
\begin{enumerate}
\item The action of permutation group of order $\ell !$ on $\mathcal M^{\text{\rm cl}}_{\ell}(\alpha)$ exchanging 
marked points is extended to an action of the Kuranishi structures.
\item The evaluation map $ev^{\text{\rm int}} : \mathcal M_{\ell}^{\text{\rm cl}}(\alpha) \to M^{\ell}$ 
is strongly continuous and weakly submersive.
\item 
Let $\ell_1 + \ell_2 = \ell + 2$, $\alpha_1 + \alpha_2 = \alpha$. Then 
by the embedding 
$$
\mathcal M^{\text{\rm cl}}_{\ell_1}(\alpha_1) {}_{ev^{\text{\rm int}}_1}\times_{{ev^{\text{\rm int}}_1}}
\mathcal M^{\text{\rm cl}}_{\ell_2}(\alpha_2)
\subset \mathcal M^{\text{\rm cl}}_{\ell}(\alpha)
$$
the Kuranishi structure in the right hand side restricts to the 
fiber product Kuranishi structure in the left hand side.
In particular, our Kuranishi structures are component-wise.
\end{enumerate}
\end{lmm} 
This is proved in \cite{FO}.
\par\smallskip
We now consider the boundary point $
(v,x) \in \mathcal M_{\ell+1}^{\text{\rm cl}}(\tilde \beta) \,{}_{ev^{\text{\rm int}}_1}\times_M\, L$,
where $v \in \mathcal M^{\text{\rm cl}}_{\ell+1}(\tilde \beta)$ and $x = ev^{\text{\rm int}}_1(v) \in L$.
Let $(V,\Gamma,E,s,\psi)$ be a Kuranishi neighborhood of $v$ we have taken 
in Lemma \ref{closedmoduliKura}. 
W take 
$$
V' = (ev^{\text{\rm int}}_1)^{-1}(L).
$$
This is smooth and 
$ev : V'  \to L$, (which is the restriction of $ev^{\text{\rm int}}$ to $V'$ is a submersion.
Therefore we take
$(V',\Gamma,E\vert_{V'},s\vert_{V'},\psi\vert_{(s\vert_{V'})^{-1}(0)/\Gamma})$
can be regarded as the Kuranishi neighborhood of $(v,x)$ 
in $\mathcal M^{\text{\rm cl}}_{\ell+1}(\widetilde{\beta}) \,{}_{ev^{\text{\rm int}}_1}\times_M\, L$.
We put
$$
\operatorname{Conf}^{o}(k;(\partial D^2,0)) 
= \{(z_1,\ldots,z_k) \in(\partial D^2)^k \mid z_i \text{ respects cyclic order} \}/S^1
$$
and let $\operatorname{Conf}(k;(\partial D^2,0))$ be its compactification.
Then 
$$
\operatorname{Conf}(k;(\partial D^2,0)) 
\times
V'
$$
can be regarded as a stratum of a Kuranishi neighborhood of 
the pull back of $(v,x)$ in $\mathcal M^{\text{\rm cl}}_{\ell+1}(\widetilde{\beta})$.
We can extend it to a Kuranishi neighborhood of $(v,x)$.
Then the properties 1,2 of Lemma \ref{ateachpointkuracomp} are satisfied.
\par
We thus described a way to obtain 
a space $E(u)$ satisfying the properties 1,2 of Lemma \ref{ateachpointkuracomp} 
at each point of the compactification of $\mathcal M_1(\beta)$.
\par
Now we are in the position to complete the proof of Theorem \ref{forgetcompKUra}. 
The proof is by induction on $\beta\cap \omega$.
We assume that we have chosen already 
the Kuranishi structure satisfying the conclusion of Theorem 3.1
for $\beta'$ with $\beta'\cap \omega <\beta\cap \omega$.
\par
We consider $\mathcal M_{0}(\beta)$.
A component of its boundary is
$$
\mathcal M_{1}(\beta_1) \times_L \mathcal M_{1}(\beta_2)
$$
with $\beta_1 + \beta_2 = \beta$.
We already fixed a Kuranishi structure on each of the factors.
We take the fiber product Kuranishi sructure on it.
We then lift it to a Kuranishi structure of 
a boundary component of $\mathcal M_{1}(\beta)$.
We claim that
$ev_0$ is weakly submersive for this Kuranishi structure.
In fact
the boundary component we are studying is one of the following three cases:
\begin{subequations}
\begin{equation}\label{bdrycomp1}
\mathcal M_{2}(\beta_1) \,{}_{ev_0}\times_{ev_0} \mathcal M_{1}(\beta_2)
\end{equation}
where $ev_0 : \mathcal M_{1}(\beta) \to L$ is the map $ev_1$ of the first factor.
\begin{equation}\label{bdrycomp2}
\mathcal M_{1}(\beta_1) \,{}_{ev_0}\times_{ev_0} \mathcal M_{2}(\beta_2),
\end{equation}
where $ev_0 : \mathcal M_{1}(\beta) \to L$ is the map $ev_1$ of the second factor.
\begin{equation}\label{bdrycomp3}
\mathcal M_{1}(\beta_1) \,{}_{ev_0}\times_{ev_0} \mathcal M_{1}(\beta_2),
\end{equation}
\end{subequations}
where $ev_0$ is the map 
$((\Sigma_1,v_1),(\Sigma_2,v_2)) \mapsto ev_0(\Sigma_1,v_1) = ev_0(\Sigma_2,v_2)$.
\par
We remark that (\ref{bdrycomp3}) is identified with
$$
(\mathcal M_{1}(\beta_1) \times \mathcal M_{1}(\beta_2)) 
\,{}_{(ev_0,ev_0)}\times_{(ev_1,ev_2)}\mathcal M_{3}(0). 
$$
Here $0$ in $\mathcal M_{3}(0)$ is $0 \in H_2(M,L;\Z)$.
\par
We first consdier the case (\ref{bdrycomp1}).
By induction hypothesis
$$
(ev_1,ev_0) : 
\mathcal M_{2}(\beta_1) \times \mathcal M_{1}(\beta_2) \to L^2
$$
is weakly submersive. 
It follows from Lemma \ref{kurafibermap} that $ev_0$ is submersive on (\ref{bdrycomp1}).
The case (\ref{bdrycomp2}), (\ref{bdrycomp3}) are similar.
\par
In case $[\partial \beta] = 0$
there is another boundary component 
$\mathcal M_{\ell+1}^{\text{\rm cl}}(\tilde\beta)\times_M L$ of 
$\mathcal M_{\ell,0}(\beta)$.
We already explained the way 
to impose obstruction bundle and then Kuranishi structure  on
$\mathcal M_{\ell+1}^{\text{\rm cl}}(\tilde\beta)$ so that the map $ev_0$ 
is weakly submersive.
Therefore we can 
define a Kuranishi structure on a neighborhood of this 
boundary component of
$\mathcal M_{\ell,0}(\beta)$ by extending 
fiber product Kuranishi structure.
We thus constructed the required Kuranishi structure on a 
neighborhood of the boundary of $\mathcal M_{\ell,0}(\beta)$.
Because of the inductive way to constructing our Kuranishi structures 
they are compatible at their intersections.
\par
Now we can use Lemma \ref{ateachpointkuracomp} in the same way as 
\cite{FOOO080} Section 7.2 to extend it the whole $\mathcal M_{\ell,0}(\beta)$.
We define Kuranishi structure on $\mathcal M_{\ell,1}(\beta)$ by 
taking pull back of the Kuranishi structure of $\mathcal M_{\ell,0}(\beta)$
via forgetful map. The proof of Theorem \ref{forgetcompKUra}
is now complete.
\end{proof}
We consider the forgetful map
$$
\mathfrak{forget}_{k+1,1}
: \mathcal M_{\ell,k+1}(\beta)  \to \mathcal M_{\ell,1}(\beta)
$$
forgetting 2nd,\dots,$k+1$-th marked points.
Note we enumerate marked points as $z_0,\ldots,z_k$.  We forget 
$z_1,\ldots,z_k$.
\begin{crl}\label{Corkura}
There exists a system of Kuranishi structures on 
$\mathcal M_{\ell,k+1}(\beta)$ $k\ge 0$, $\ell\ge 0$,
with the following properties.
\par\smallskip
\begin{enumerate}
\item It is compatible with $\mathfrak{forget}_{k+1,1}$.
\item It is invariant under the cyclic permutation of the 
boundary marked points.
\item It is invariant of the permutation of interior marked points.
\item $ev_0 : \mathcal M_{\ell,k+1}(\beta) \to L$ is strongly 
submersive.
\item We consider the decomposition of the boundary:
\begin{equation}\label{bdcompati0}
\aligned
\partial \mathcal M_{\ell,k+1}(\beta)
=
\bigcup_{1\le i\le j+1 \le k+1} 
&\bigcup_{\beta_1+\beta_2=\beta}
\bigcup_{L_1\cup L_2=\{1,\ldots,\ell\}} \\
&\mathcal M_{\#L_1,j-i+1}(\beta_1) {}_{ev_0} \times_{ev_i} 
\mathcal M_{\#L_2,k-j+i}(\beta_2).
\endaligned
\end{equation}
(See {\rm \cite{FOOO080}} Subsection {\rm 7.1.1.}) Then the restriction of the Kuranishi structure of 
$\mathcal M_{\ell,k+1}(\beta)$ in the left hand side 
coincides with the fiber product Kuranishi structure in 
the right hand side.
\end{enumerate}
\end{crl}
We first explain the statement. 
\par
1 is similar to the Definition \ref{forgetcomp}.
The only difference is we replace 
$(0,1)$ appearing there with 
$(0,1)^{k-m} \times [0,1)^m$ for some appropriate $m$.
\par
The cyclic permutation of boundary marked points is defined as follows.
Let $(\Sigma,v)$ be a point in  $\mathcal M_{\ell,k+1}(\beta)$.
Let $z_0,z_1,\ldots,z_k$ be boundary marked points of $\Sigma$.
We change them to $z_1,\ldots,z_k,z_0$ to obtain 
$(\Sigma',v)$. 2 claims that this action extends to the Kuranishi structure.
(See \cite{FOOO080} Subsection A1.3 for the definition of finite group action to Kuranishi structure.)
\par
The meaning of 3 is similar. Here we consider not only cyclic permutation 
but also an arbitrary permutation of the interior marked points.
\par
We remark that 4 and Lemma \ref{fiberprodform} imply that the right hand side of (\ref{bdcompati0})
has a Kuranishi structure.
Then 5 claims that the boundary of the moduli space of $\mathcal M_{\ell,k+1}(\beta)$ 
as Kuranishi space decomposes as in the right hand side of  (\ref{bdcompati0}).
We remark that the $i$-th, \dots, $j$-th boundary marked points of
 $\mathcal M_{\ell,k+1}(\beta)$ correspond to the $1$-st, \dots, $j-i+1$-th 
 marked points of the first factor and 
 the other boundary marked points (except the $0$-th) of  $\mathcal M_{\ell,k+1}(\beta)$ correspond 
 to the boundary marked points (except $i$-th) of the right hand side.
Also the interior marked points correspond  to each other in an obvious way. 
We then have a compatibility statement of evaluation maps in an obvious way.
It is a part of the statement 5.
\par
Note $k+1\ge 1$. Therefore the extra boundary component 
$\mathcal M_{\ell+1}^{\text{\rm cl}}(\tilde\beta) \times_M L$ of 
$\mathcal M_{\ell,0}(\beta)$ does not appear in (\ref{bdcompati0}).
\begin{proof}
We already defined the Kuranishi structure on $\mathcal M_{\ell,1}(\beta)$ 
in Theorem \ref{forgetcompKUra}.
We define the Kuranishi structure on $\mathcal M_{\ell,k+1}(\beta)$ 
so that 1 holds. (This determines the Kuranishi structure uniquely.)
Then, by Theorem \ref{forgetcompKUra}.1, our Kuranishi structure of 
$\mathcal M_{\ell,k+1}(\beta)$  is pulled back 
from one on $\mathcal M_{\ell,0}(\beta)$. 
2 follows immediately.
3 is a consequence of Theorem \ref{forgetcompKUra}.4.
4 is a consequence of Theorem \ref{forgetcompKUra}.2.
5 is a consequence of Theorem \ref{forgetcompKUra}.3.
\end{proof}
\begin{rem}
The Kuranishi structure we constructed is {\it not} compatible with the 
forgetful map of the {\it interior} marked points.
The reason is rather technical. Namely we require the support of 
the obsruction bundle $E_a$ to be disjoint from marked or singular points.
This is automatic for the boundary marked points since we also 
assume that it is disjoint from the boundary.
However if we try to imitate the proof of Theorem \ref{forgetcompKUra} 
to obtain a Kuranishi structure which is compatible with the forgetful map, this causes a 
problem. Namely if we construct 
Kuranishi structure of $\mathcal M_{0,0}(\beta)$ and 
$\mathcal M^{\text{\rm cl}}_{0}(\alpha)$ so that the evaluation map 
is submersive 
for the pull back Kuranishi structure on 
$\mathcal M_{1,0}(\beta)$ and 
$\mathcal M^{\text{\rm cl}}_{1}(\alpha)$, then, 
a priori, we can not assume the support of obsruction bundle $E_a$
to be disjoint from marked points.
(In fact we need to fix $E_a$ for elements of 
$\mathcal M^{\text{\rm cl}}_{0}(\alpha)$. So 
the interior marked point of the pull back Kuranishi structure 
$\mathcal M^{\text{\rm cl}}_{1}(\alpha)$ can be arbitrary. 
Therefore, we can not exclude that it is in the support of $E_a$.)
\par
The reason why the support of obsruction bundle $E_a$
is assumed to be disjoint from singular point
is as follows:
\par\smallskip
\begin{enumerate}
\item The glueing analysis is easier. 
If $E_a$ hits the singular point we need to study the case 
when perturbation is put on the neck region also.  
\item 
We need to identify obstruction bundle of the pieces 
$\Sigma_a$ as a section of appropriate bundles 
after resolving singularity of $\cup \Sigma_a$.
(See (\ref{Obstpara}).) This is easier in case when the support of $E_a$ is 
away from singular points.
\end{enumerate}
\par\smallskip
We remark that we can not distinguish marked points from singular points when we want to 
make perturbation component-wise.
This is the reason why it is assumed that the support of 
$E_a$ is disjoint from marked points. 
However the description above also shows that by 
working harder we might remove this restriction and 
then can find a Kuranishi structure on $\mathcal M_{\ell,k}(\beta)$
which is compatible with the forgetful map of the interior marked 
points also. Since we do not need it in this paper, we do not try to prove it 
here.
\end{rem}
\begin{rem}\label{remwksubmer}
For the Kuranishi structure we constructed in Corollary \ref{Corkura},
the evaluation map $ev_0$ is weakly submersive.
As a consequence of cyclic symmetry it implies that 
$ev_i$ is submersive for each but fixed $i$.
On the other hand, in \cite{FOOO080}, we used a Kuranishi structure 
such that 
$(ev_0,\ldots,ev_{k-1}) : \mathcal M_{\ell,k}(\beta) \to L^k$
is weakly submersive.
We remark that there does not exist a system of Kuranishi structures 
such that $(ev_0,\ldots,ev_{k-1})$ is weakly submersive 
and is compatible with the forgetful map in the sense of 
Corollary \ref{Corkura}.1, at the same time.
In fact if $d$ is a  dimension of the Kuranishi neighborhood 
$V_{\text{\bf p}}$
of a point in $\mathcal M_{\ell,1}(\beta)$, 
then the dimension of the Kuranishi neighborhood 
$V_{\widetilde{\text{\bf p}}}$ of compatible 
Kuranishi structure in $\mathcal M_{\ell,k}(\beta)$ 
is $k+d-1$. 
(We remark that dimension here is one of the manifold 
$V_{\text{\bf p}}$. It is different from the 
dimension as the Kuranishi space.)
If $k$ is large then dimension of $L^k$, 
which is $nk$, is certainly bigger than
$k+d-1$. Therefore $(ev_0,\ldots,ev_{k-1})$ can 
not be weakly submersive.
\par\smallskip
A key idea of the proof of Corollary \ref{Corkura} 
is the observation that the 
submersivity of $ev_0$ is enough to carry out 
the inductive construction. 
(This observation was due to \cite{FOOO08I}.)
We remark also that for this observation to hold the assumption,
genus $=0$ is essential. Namely 
the same method does not work for its generalization 
to higher genus.
\end{rem}
\section{Continuous family of multisections: review}
In the next section we will construct perturbation of the Kuranishi map of the 
Kuranishi structure constructed in the last section so that it is compatible 
with the forgetful map of the boundary marked points and is 
cyclically symmetric (as its consequence). 
The author does not know how to do it using 
multi (but finitely many) valued sections. 
So we use continuous family of multisections.
The notion of continuous family of multisections had already been used in 
\cite{FOOO06} Section 33,  \cite{Fuk07I},\cite{FOOO08II} etc.. We first review them in this section.
\par
We first recall the notion of good coordinate system.
Let  $(V_{\alpha},E_{\alpha},
\Gamma_{\alpha},\psi_{\alpha},s_{\alpha})$ be Kuranishi charts 
parametrized by $\alpha \in \mathfrak A$.
We assume that the index set $\mathfrak A$ has a partial order $<$, where 
either $\alpha_1 \le \alpha_2$ or $\alpha_2 \le \alpha_1$ holds 
for $\alpha_1, \alpha_2 \in \mathfrak A$
if 
$$
\psi_{\alpha_1}(s_{\alpha_1}^{-1}(0)/\Gamma_{\alpha_1}) \cap
\psi_{\alpha_2}(s_{\alpha_2}^{-1}(0)/\Gamma_{\alpha_2}) \ne \emptyset.
$$
Moreover we assume that if  $\alpha_1, \alpha_2 \in \mathfrak A$ and $\alpha_1 \le \alpha_2$ 
then there exists a coordinate transformation from 
$(V_{\alpha_1},E_{\alpha_1},
\Gamma_{\alpha_1},\psi_{\alpha_1},s_{\alpha_1})$ 
to 
$(V_{\alpha_2},E_{\alpha_2},
\Gamma_{\alpha_2},\psi_{\alpha_2},s_{\alpha_2})$, 
in the sense described in Section \ref{kurareview}.
We assume compatibility between coordinate transformations 
in the sense of \cite{FO}, \cite{FOOO080}. 
Existence of good coordinate system is proved in \cite{FO} Lemma 6.3.
\par
We next review multisection. (See \cite{FO} section 3.)
Let $(V_{\alpha},E_{\alpha},\psi_{\alpha},s_{\alpha},\Gamma_{\alpha})$ be a Kuranishi 
chart of $\mathcal M$. For $x \in V_{\alpha}$ we consider the fiber $E_{\alpha,x}$ of 
the bundle $E_{\alpha}$ at $x$. We take its $l$ copies and consider the direct product
$
E_{\alpha,x}^{l}
$. We divide it by the action of symmetric group of order $l!$ and let 
$\mathcal S^l(E_{\alpha,x})$ be the quotient space.
There exists a map
$
tm_m : \mathcal S^l(E_{\alpha,x}) 
\to \mathcal S^{lm}(E_{\alpha,x}),
$
which sends $[a_1,\ldots,a_l]$ to
$
[\,\underbrace {a_1,\ldots,a_1}_{\text{$m$ copies}},\ldots,
\underbrace {a_l,\ldots,a_l}_{\text{$m$ copies}}].
$
A {\it multisection} $s$ of the orbibundle 
$
E_{\alpha} \to V_{\alpha}
$
consists of an open covering 
$
\bigcup_iU_i = V_{\alpha}
$ 
and 
$s_i$ which sends $x\in U_i$ to $s_i(x) \in \mathcal S^{l_i}(E_{\alpha,x})$.
They are required to have the following properties.
\begin{enumerate}
\item $U_i$ is $\Gamma_{\alpha}$ invariant. $s_i$ is $\Gamma_{\alpha}$ 
equivariant. (We remark that there exists an obvious map 
$
\gamma : \mathcal S^{l_i}(E_{\alpha,x}) \to \mathcal S^{l_i}(E_{\alpha,\gamma x})
$
for each $\gamma \in \Gamma_{\alpha}$.)
\item If $x \in U_i \cap U_j$ then we have
$
tm_{l_j}(s_i(x)) = tm_{l_i}(s_j(x)) \in \mathcal S^{l_il_j}(E_{\alpha,\gamma x}).
$
\item $s_i$ is {\it liftable and smooth}  in the following sense. 
For each $x$ there exists a smooth section $\tilde s_i$ of 
$\underbrace{E_{\alpha} \oplus \ldots \oplus E_{\alpha}}_{\text{$l_i$ times}}$ 
in a neighborhood of $x$ such that 
\begin{equation}\label{locallift}
\tilde s_i(y) = (s_{i,1}(y),\ldots,s_{i,l_i}(y)), 
\quad
s_i(y) = [s_{i,1}(y),\ldots,s_{i,l_i}(y)].
\end{equation}
\end{enumerate}
We identify two multisections $(\{U_i\},\{s_i\},\{l_i\})$,  $(\{U'_i\},\{s'_i\},\{l'_i\})$
if 
$$
tm_{l_j}(s_i(x)) = tm_{l'_i}(s'_j(x)) \in \mathcal S^{l_il'_j}(E_{\alpha,\gamma x}).
$$
on $U_i \cap U'_j$. 
We say $s_{i,j}$ to be a {\it branch} of $s_i$ in the situation of (\ref{locallift}).
\par
We next discuss continuous family of multisections and its transversality.
Let $W_{\alpha}$  be
a finite dimensional manifold and consider the pull-back bundle
$$
\pi_\alpha^*E_\alpha \to V_\alpha \times W_\alpha
$$
under $\pi_\alpha: V_\alpha \times W_\alpha \to V_\alpha$. 
The action of $\Gamma_{\alpha}$ on $W_{\alpha}$ is, by definition, trivial.
\begin{dfn}\label{defnmultisec}
\begin{enumerate}
\item
A $W_{\alpha}$ parametrized family $\frak s_{\alpha}$ of multisections is by definition a 
multisection of 
$\pi_\alpha^*E_\alpha$.
\item
We fix a metric of our bundle $E_{\alpha}$.
We say $\frak s_{\alpha}$ is $\epsilon$ close to $s_{\alpha}$ in 
$C^0$ topology if the following holds.
Let $(x,w) \in V_{\alpha} \times W_{\alpha}$. Then for any branch 
$\frak s_{\alpha,i,j}$ of $\frak s_{\alpha}$ we have
$$
dist(\frak s_{\alpha,i,j}(w,y),s_{\alpha}(y)) < \epsilon
$$
if $y$ is in a neighborhood of $x$.
\item
$\frak s_{\alpha}$ is said to be transversal to $0$ if the following holds :
Let $(x,w) \in  V_{\alpha}\times W_{\alpha} $. Then any branch 
$\frak s_{\alpha,i,j}$ of $\frak s_{\alpha}$ is transversal to $0$.
\item
Let $f_{\alpha} : V_{\alpha} \to M$ be a $\Gamma_{\alpha}$ equivariant smooth map. 
We assume that $\frak s_{\alpha}$ is transversal to $0$.
We then say that $f_{\alpha}\vert_{\frak s_{\alpha}^{-1}(0)}$ is a submersion 
if the following holds. 
Let $(x,w) \in  V_{\alpha}\times W_{\alpha}$. Then for any branch 
$\frak s_{\alpha,i,j}$ of $\frak s_{\alpha}$ the restriction of 
$
f_{\alpha} \circ \pi_{\alpha} : V_{\alpha} \times W_{\alpha} \to M
$
to 
\begin{equation}\label{fras-1(0)}
\{ (x,w) \mid \frak s_{\alpha,i,j}(x,w) = 0\}
\end{equation}
is a submersion. We remark that (\ref{fras-1(0)}) is a smooth manifold by 3. 
\end{enumerate}
\end{dfn}
\begin{rem}
In case $\mathcal M$ has boundary or corner, (\ref{fras-1(0)}) has 
boundary or corner. In this case we require that the restriction of $f_{\alpha}$ 
to each of the stratum of (\ref{fras-1(0)}) is a submersion.
\end{rem}
\begin{lmm}\label{existencesubmersion}
We assume that $f_{\alpha} : V_{\alpha} \to M$ is a submersion. Then 
there exists $W_{\alpha}$ such that for any  $\epsilon$ there 
exists a $W_{\alpha}$ parametrized family $\frak s_{\alpha}$ of multisections 
which is $\epsilon$ close to $s_{\alpha}$, transversal to $0$ and 
such that $f_{\alpha}\vert_{\frak s_{\alpha}^{-1}(0)}$ is a submersion.
\par
Moreover there exists $0\in W_{\alpha}$ such that 
the restriction of $\frak s_{\alpha}$ to 
$V_{\alpha} \times \{0\}$ coincides with $s_{\alpha}$.
\par
If $\frak s_{\alpha}$ is already given and satisfies the required condition 
on a neighborhood of a compact set $K_{\alpha}$, then we may 
extend it to the whole $W_{\alpha}$ without changing it on $K_{\alpha}$.
\end{lmm}
We omit the proof. See \cite{FOOO08II}.
We next describe the compatibility conditions among the $W_\alpha$-parametrized 
families of
multisections for various $\alpha$.
During the construction we need to shrink $V_{\alpha}$ a bit several times.
We do not mention it explicitely below.
\par
Let $\alpha_1 < \alpha_2$.
We consider the normal bundle 
$N_{\phi_{\alpha_2\alpha_1}(V_{\alpha_1\alpha_2})}V_{\alpha_2}$ 
of the embedding $\phi_{\alpha_2\alpha_1} : 
V_{\alpha_1\alpha_2} \to V_{\alpha_2}$.
We take a small neighborhood $U_{\epsilon}(\phi_{\alpha_2\alpha_1}(V_{\alpha_1\alpha_2}))$ 
of its image and identify it with an $\epsilon$-neighborhood 
$B_{\epsilon}N_{\phi_{\alpha_2\alpha_1}(V_{\alpha_1\alpha_2})}V_{\alpha_2}$ of 
the zero section $\cong V_{\alpha_1\alpha_2})$ of 
$N_{\phi_{\alpha_2\alpha_1}(V_{\alpha_1\alpha_2})}V_{\alpha_2}$.
We then have a projection 
$\text{Pr} : U_{\epsilon}(\phi_{\alpha_2\alpha_1}(V_{\alpha_1\alpha_2})) 
\to V_{\alpha_1\alpha_2}$.
\par
We pull back the bundle $E_{\alpha_1}\vert_{V_{\alpha_1\alpha_2})}$ by 
$\text{Pr}$ to obtain 
$\text{Pr}^*E_{\alpha_1} \to V_{\alpha_1\alpha_2}$.
We extend the bundle embedding 
$\hat{\phi}_{\alpha_2\alpha_1} : E_{\alpha_1}\vert_{V_{\alpha_1\alpha_2}} 
\to E_{\alpha_2} $ to a bundle embedding:
$$
\hat{\phi}_{\alpha_2\alpha_1} 
: \text{Pr}^*E_{\alpha_1} \to E_{\alpha_2}\vert_{U_{\epsilon}(\phi_{\alpha_2\alpha_1}(V_{\alpha_1\alpha_2}))}. 
$$
We consider the section (Kuranishi map) $s_{\alpha_2} : U_{\epsilon}(\phi_{\alpha_2\alpha_1}(V_{\alpha_1\alpha_2})) 
\to E_{\alpha_2}$ and compose it with the projection to obtain:
\begin{equation}\label{piesalph}
\pi \circ s_{\alpha_2} : U_{\epsilon}(\phi_{\alpha_2\alpha_1}(V_{\alpha_1\alpha_2}))  
\to \frac{ E_{\alpha_2}}{\text{Pr}^*E_{\alpha_1}}.
\end{equation}
We remark that  (\ref{piesalph}) is zero on the zero section 
$=\phi_{\alpha_2\alpha_1}(V_{\alpha_1\alpha_2})$
and the fiber derivative of it there induces an isomorphism.
\par
Let $\text{\rm Exp} : B_{\epsilon}N_{\phi_{\alpha_2\alpha_1}(V_{\alpha_1\alpha_2})}V_{\alpha_2}
\to U_{\epsilon}(\phi_{\alpha_2\alpha_1}(V_{\alpha_1\alpha_2}))$ be the 
isomorphism we mentioned above. By modifying it using fiber-wise 
diffeomorphism we may assume that 
\begin{equation}\label{Expiesalph}
\pi \circ s_{\alpha_2} \circ \text{\rm Exp} 
: B_{\epsilon}N_{\phi_{\alpha_2\alpha_1}(V_{\alpha_1\alpha_2})}V_{\alpha_2}
\to \frac{ E_{\alpha_2}}{\text{Pr}^*E_{\alpha_1}}
\end{equation}
is a restriction of (linear) isomorphism of vector bundles.
\par
Now let $U_{i,\alpha_1} \subset V_{\alpha_1}$ and 
 $\{\frak s_{\alpha_1,i,j}\mid j=1,\ldots,l_i\}$ be a multisection on 
 $U_{i,\alpha_1}\times W_{\alpha_1}$.
 We take $W_{\alpha_2,\alpha_1}$ and put 
 $W_{\alpha_2} = W_{\alpha_1} \times W_{\alpha_2,\alpha_1}$.
 We define 
 $$
 \text{Pr}^{-1}(\phi_{\alpha_2\alpha_1}(V_{\alpha_1\alpha_2}))
 = U_{i,\alpha_2}\subset U_{\epsilon}(\phi_{\alpha_2\alpha_1}(V_{\alpha_1\alpha_2})).
 $$
 \begin{dfn}\label{multiseccomp}
$W_{\alpha_2}$ parametrized family of multisections 
$\{\frak s_{\alpha_2,i,j}\mid j=1,\ldots,l_i\}$ on $U_{i,\alpha_2}$  is said to be 
{\it compatible} with  $\{\frak s_{\alpha_1,i,j}\mid j=1,\ldots,l_i\}$ 
if the following holds.
\par
$\{\frak s_{\alpha_2,i,j}\mid j=1,\ldots,l_i\}$ is a multisection of $E_{\alpha_2}$ 
on $U_{i,\alpha_2} \times W_{\alpha_2}$. 
\par
Let $y = \text{\rm Exp}(x,\xi)$ with $x \in U_{i,\alpha_1} \cap U_{\alpha_1,\alpha_2}$, 
$\xi \in (N_{\phi_{\alpha_2\alpha_1}(V_{\alpha_1\alpha_2})}V_{\alpha_2})_{x}$
($\Vert\xi\Vert < \epsilon$) and $w =(w_1,w_2) \in W_{\alpha_2} = 
W_{\alpha_1} \times W_{\alpha_2,\alpha_1}$. Then, we have:
\begin{equation}
\frak s_{\alpha_2,i,j}(y,w) 
\equiv
(\pi \circ s_{\alpha_2})(y) \mod \text{Pr}^*E_{\alpha_1}.
\end{equation}
We assume also that 
\begin{equation}
\frak s_{\alpha_2,i,j}(\text{\rm Exp}(x,0),w) = \frak s_{\alpha_1,i,j}(x,w_1).
\end{equation}
 \end{dfn}
 We remark that for given  $\{\frak s_{\alpha_1,i,j}\mid j=1,\ldots,l_i\}$ 
 we can always find $\{\frak s_{\alpha_2,i,j}\mid j=1,\ldots,l_i\}$ 
 which is compatible to it. In fact we can use the splitting
 $$
 E_{\alpha_2} =  \text{Pr}^*E_{\alpha_1} \oplus \frac{ E_{\alpha_2}}{\text{Pr}^*E_{\alpha_1}}
 $$
 to construct it.
 \par
 Moreover if $f = \{f_{\alpha}\} : \mathcal M \to N$ is strongly continuous and weakly
 submersive map and $f_{\alpha_1}\vert_{\frak s_{\alpha_1}^{-1}(0)}$ is a submersion, 
 then $f_{\alpha_2}\vert_{\frak s_{\alpha_2}^{-1}(0)}$ is a submersion, 
 for small $\epsilon$.
 \par
 Thus we can prove the following by induction of $\alpha$ with respect to the 
 order $<$ and by using Lemma \ref{existencesubmersion}.
 \begin{prp}\label{globalfamikura}
Let $\mathcal M$ be a Kuranishi space with a good coordinate system 
$(V_{\alpha},E_{\alpha},
\Gamma_{\alpha},\psi_{\alpha},s_{\alpha})$ ($\alpha \in \frak A$).
Let $f = \{f_{\alpha}\} : \mathcal M \to N$  be a strongly continuous and weakly
submersive map.
Then there exists a system of continuous family of multisections 
$\frak s_{\alpha} = \{\frak s_{\alpha,i,j}\mid j=1,\ldots,l_i\}$ of 
$E_{\alpha}$ such that:
\smallskip 
\begin{enumerate}
\item They are compatible in the sense of Definition \ref{multiseccomp}.
\item They are transversal to $0$ in the sense of Definition \ref{defnmultisec}.$3$.
\item $f_{\alpha}\vert_{\frak s_{\alpha}^{-1}(0)}$ is a submersion
 in the sense of Definition \ref{defnmultisec}.$4$.
\item They are $C^0$ close to the original Kuranishi map in the sense of 
Definition \ref{defnmultisec}.$2$.
\end{enumerate}
\end{prp}
\begin{rem}
We can prove a relative version of Proposition \ref{globalfamikura}.
Namely if there exists an open set $\mathcal U \subset \mathcal M$ and 
a compact set $K \subset \mathcal U$ and $\frak s_{\alpha}$ 
satisfying $1,2,3,4$ above are given on $\mathcal U$, then 
we can extend it with required properties, without changing it on $K$.
\end{rem}
We next review the way to use family of multisections to define 
smooth correspondence.
\par
We work on the following situation.
Let $\mathcal M$ be a Kuranishi space, 
$f^{s} =\{f^s_{\alpha}\} : \mathcal M \to N_s$  a strongly continuous map 
and $f^{t} =\{f^t_{\alpha}\} : \mathcal M \to N_t$  a strongly continuous and 
weakly submersive map.
Here $N_s,N_t$ are smooth manifolds. (Here $s$ and $t$ stand for source and 
target, respectively.)
Let $\Lambda^d(M)$ denote the set of smooth $d$ forms on $M$.
We will define
\begin{equation}\label{correspondence}
\operatorname{Corr}_*(\mathcal M;f^s,f^t) : \Lambda^d(N_s) \to \Lambda^{\ell}(N_t)
\end{equation}
where 
$
\ell = d + \dim N_t - \dim \mathcal M.
$
\par
We first take a continuous family of multisections 
$\frak s_{\alpha} = \{\frak s_{\alpha,i,j}\mid j=1,\ldots,l_i\}$ 
on $\mathcal M$ such that 
$f^t_{\alpha}\vert_{\frak s_{\alpha}^{-1}(0)}$ is a submersion.
Let $\rho \in \Lambda^d(N_s)$.
 We consider a representative $\frak s_{\alpha,i,j}$ of $\frak s_{\alpha}$ 
 it is a section of $E_{\alpha}$ on $U_{i,\alpha} \times W_{\alpha}$.
 We take a top dimensional smooth form $\omega_{\alpha}$ on $W_{\alpha}$ of compact support 
 such that its total mass is $1$.
 Let $\chi_i$ be a partition of unity subordinate to the covering 
 $\{U_{i,\alpha}\}$. We consider
 \begin{equation}\label{localdefofcorr}
\sum_i  f^{t}_{\alpha !} \frac{1}{l_{i}}\chi_i \left( (f^s_{\alpha})^* \rho \wedge \omega_{\alpha} \right)
\vert_{\frak s_{\alpha,i,j}}.
\end{equation}
 Here $l_i$ is the number of branches and $f^{t}_{\alpha !}$ are integration along fiber.
 This defines `$U_{\alpha}$ part' of $\operatorname{Corr}_*(\mathcal M;f^s,f^t)$.
 We use partition of unity again to glue them for various $\alpha$ 
 to obtain a map (\ref{correspondence}).
 See \cite{FOOO08II} for detail.
 \begin{rem}
 The map (\ref{correspondence}) depends on the choice of multisection $\frak s_{\alpha}$ and the 
 smooth form $\omega_{\alpha}$. But it is independent of the choice 
 of partition of unity.
 \end{rem}
 The smooth correspondence we defined above has the following two properties.
 \begin{prp}[Stokes]\label{stokes}
 We have
 \begin{equation}
d \circ \operatorname{Corr}_*(\mathcal M;f^s,f^t) 
- \operatorname{Corr}_*(\mathcal M;f^s,f^t) \circ d 
= \operatorname{Corr}_*(\partial\mathcal M;f^s,f^t).
\end{equation}
\end{prp}
\begin{proof}
On each chart this is a consequence of Stokes' theorem. 
By using partition of unity in a standard way we obtain the 
proposition.
\end{proof}
To state the next proposition, we need some notation.
We consider oriented Kuranishi spaces $\mathcal M_1$, $\mathcal M_2$.
Let $f^{1,s} : \mathcal M_1 \to N_s^1$, 
$f^{2,s} : \mathcal M_2 \to N_t^1 \times N_s^2$ be 
strongly continuous maps and 
$f^{1,t} : \mathcal M_1 \to N_t^1$, 
$f^{2,t} : \mathcal M_2 \to N_t^2$ be 
strongly continuous and weakly submersive maps. 
We put
\begin{equation}
\mathcal M = \mathcal M_1\,\, {}_{f^{1,t}}\times_{M_t^1} \mathcal M_2
\end{equation}
where we use the second factor $: \mathcal M_2 \to N_t^1$ of the map $f^{2,s}$
to define the above fiber product. 
$f^{1,s}$ and $f^{2,s}$ induces a strongly continuous map
$f^s : \mathcal M \to N_s^1 \times N_s^2$.
$f^{2,t}$ induces a strongly continuous and weakly submersive
map $f^t : \mathcal M \to N_t^2$.
(Lemma \ref{kurafiber}.)
\begin{prp}[Composition formula]\label{comprop}
If $\rho_i \in \Lambda(N^2_i)$ $(i=1,2)$ then we have
\begin{equation}
\aligned
&\operatorname{Corr}_*(\mathcal M;f^s,f^t)(\rho_1 \times \rho_2) \\
&= 
\operatorname{Corr}_*(\mathcal M_2;f^{2,s},f^{2,t})
(\operatorname{Corr}_*(\mathcal M_1;f^{1,s},f^{1,t})(\rho_1) \times \rho_2).
\endaligned\end{equation}
\end{prp}
 \begin{proof}
 On each chart this is obvious. So we can use partition of unity in a standard way 
 to prove the propositions.
 \end{proof}
 \begin{rem}\label{Corandint}
 We need to take and fix appropriate orientation in Propositions  \ref{stokes},  \ref{comprop}.
 We will do it later but only in case we use. See the end of Section \ref{cyclicTEderamsec}.
 \end{rem}
 We also remark the next lemma.
 Let $(\mathcal M;f^s,f^t)$ be as above. Here $f^t : \mathcal M \to N_t$ is weakly submersive.
We consider $(\mathcal M;f^s\times f^t,\text{\rm const})$ where $\text{\rm const} : \mathcal M \to \{\text{\rm point}\}$
is a constant map to a point.
 \begin{lmm}\label{Corandint2}
Let $\rho_s \in \Lambda(N_s)$ and $\rho_t \in \Lambda(N_t)$. Then 
\begin{equation}
\int_{N_t}\operatorname{Corr}(\mathcal M;f^s,f^t)(\rho_s) \wedge \rho_t
= \operatorname{Corr}(\mathcal M;f^s \times f^t,\text{\rm const})(\rho_s \wedge \rho_t).
\end{equation}
 \end{lmm}
 Note the right hand side is an element of $\Lambda( \{\text{point}\}) = \R$.
\section{Forgetfulmap compatible continuous family of multisections}
\label{fgcompfamisec}
In this section, we define a system of continuous family of multisections on 
the Kuranishi space produced in Section \ref{Seckuranishiforget}.
We will construct the system so that it is compatible with forgetful maps.
\par
We first define this compatibility precisely.
We consider the moduli spaces $\mathcal M_{\ell,1}(\beta)$, 
$\mathcal M_{\ell,0}(\beta)$ and consider their Kuranishi structures 
which are compatible in the sense of Definition \ref{forgetcomp}.
We take their good coordinate systems which are 
compatible in a similar sense. 
More precise definition of this compatibility is in order.
We have Kuranishi charts  $(V_{\alpha},E_{\alpha},\psi_{\alpha},s_{\alpha},\Gamma_{\alpha})$ 
($\alpha \in \frak A$) of $\mathcal M_{\ell,0}(\beta)$ 
and 
$(V_{\tilde\alpha},E_{\tilde\alpha},\psi_{\tilde\alpha},s_{\tilde\alpha},\Gamma_{\tilde\alpha})$ 
($\tilde\alpha \in \tilde{\frak A}$) of $\mathcal M_{\ell,1}(\beta)$. 
Here $\frak A$ and $\tilde{\frak A}$ are partially ordered sets.
We require that there exists an order preserving map 
$\tilde{\frak A} \to \frak A$, $\tilde\alpha \mapsto \alpha$ such that:
\par\smallskip
\begin{enumerate}
\item $V_{\tilde\alpha} = V_{\alpha} \times (0,1)$.
\item $E_{\tilde\alpha} = E_{\alpha} \times (0,1)$.
\item $\Gamma_{\tilde\alpha} = \Gamma_{\alpha}$. 
The action of  $\Gamma_{\tilde\alpha}$ preserves identifications 
given in 1 and 2, where the action to the factor $(0,1)$ is trivial.
\item $s_{\tilde\alpha}(x,t) = (s_{\alpha}(x),t)$ by the identification given in 1 and 2.
\item $\mathfrak{forget} \circ \psi_{\tilde\alpha}$ coincides 
with the composition of $\psi_{\alpha}$ and the projection to the 
first factor.
\end{enumerate}
\begin{dfn}\label{forgetmulticomp}
Let $U_{\alpha,i}$, $W_{\alpha}$. $\{\frak s_{\alpha,i,j}\}$ define a 
compatible system of family of multisections on $\mathcal M_{\ell,0}(\beta)$ 
and 
$U_{\tilde\alpha,i}$, $W_{\tilde\alpha}$. $\{\frak s_{\tilde\alpha,i,j}\}$ define a 
compatible system of family of multisections on $\mathcal M_{\ell,1}(\beta)$. 
\par
We say that they are {\it compatible} if the following conditions are satisfied:
\par\smallskip
\begin{enumerate}
\item $U_{\tilde\alpha,i} = U_{\alpha,i} \times (0,1)$. 
\item $W_{\tilde\alpha} = W_{\alpha}$.
\item $\frak s_{\tilde\alpha,i,j}(w,x,t) = \frak s_{\alpha,i,j}(w,(x,t))$.
\end{enumerate}
\end{dfn}
The main result of this section is as follows:
\begin{thm}\label{multicontforgetmain}
For each $E_0>0$, $\ell_0 \in \Z_{\ge 0}$ and $\epsilon > 0$, there exists a 
compatible systems of familis of multisections 
$(U_{\tilde\alpha,i},W_{\tilde\alpha},\{\frak s_{\tilde\alpha,i,j}\})$, 
$(U_{\alpha,i},W_{\alpha},\{\frak s_{\alpha,i,j}\})$ on 
$\mathcal M_{\ell,1}(\beta)$, $\mathcal M_{\ell,0}(\beta)$ 
for $\beta \cap \omega \le E_0$ and $\ell \le \ell_0$, with the following properties:
\par\smallskip
\begin{enumerate}
\item They are $\epsilon$ close to the Kuranishi map.
\item They are compatible in the sense of Definition \ref{forgetmulticomp}.
\item They are transversal to $0$ in the sense of Definition \ref{defnmultisec}.$3$.
\item $ev_0 : \mathcal M_{\ell,1}(\beta) \to L$ induces submersions 
$(ev_0)_{\tilde\alpha}\vert_{\frak s_{\tilde\alpha}^{-1}(0)} : \frak s_{\tilde\alpha}^{-1}(0) \to L$.
\item They are compatible with $(\ref{gluemap})$ in the sense we describe below.
\end{enumerate}
\end{thm}
We describe the above Condition 5 precisely below.
We consider $U_{\alpha_1,i_1}$, $\frak s_{\tilde\alpha_1,i_1,j_1}$ 
of $\mathcal M_{\ell_1,1}(\beta_1)$ and 
$U_{\alpha_2,i_2}$, $\frak s_{\tilde\alpha_2,i_2,j_2}$ 
of $\mathcal M_{\ell_2,1}(\beta_2)$.
We put $\alpha = (\alpha_1,\alpha_2)$. The fiber product
\begin{equation}\label{Vfiberprod}
V_{\alpha} := V_{\alpha_1} \,{}_{ev_0}\times_{ev_0} V_{\alpha_2}
\end{equation}
is a Kuranishi neighborhood of the Kuranishi space 
$$
\mathcal M_{\ell_1,1}(\beta_1)\,{}_{ev_0}\times_{ev_0}\mathcal M_{\ell_2,1}(\beta_2).
$$
(\ref{Vfiberprod}) contains a fiber product 
$U_{\alpha,(i_1,i_2)} = U_{\alpha_1,i_1} \,{}_{ev_0}\times_{ev_0} U_{\alpha_2,i_2}$.
An obstruction bundle on (\ref{Vfiberprod})  is a 
restriction of $E_{\alpha_1} \times E_{\alpha_2}$.
We put $W_{\alpha} = W_{\alpha_1} \times W_{\alpha_2}$.
Now we define
\begin{equation}\label{fprocontmulti}
\frak s_{\alpha,(i_1,i_2),(j_1,j_2)} = (\frak s_{\alpha_1,i_1,j_1},\frak s_{\alpha_2,i_2,j_2})\vert
_{U_{\alpha,(i_1,i_2)}}.
\end{equation}
(\ref{fprocontmulti}) defines a compatible system of continuous family of 
multisections on 
$\mathcal M_{\ell_1,1}(\beta_1)\,{}_{ev_0}\times_{ev_0}\mathcal M_{\ell_2,1}(\beta_2)$.
Condition 5 requires that this system coincides with the restriction 
of one in $\mathcal M_{\ell,0}(\beta)$. Here $\beta = \beta_1+\beta_2$, 
$\ell = \ell_1+\ell_2$.
\begin{proof}
The strategy of the proof is similar to the proof of Theorem \ref{forgetcompKUra}.
We first consider the situation of Lemma \ref{ateachpointkuracomp}
and use the notations there.
\begin{lmm}\label{lemmulticonst}
Let $V_0 \subset \overline V_0 \subset V_1 \subset \overline V_1 \subset V$
be neighborhoods of $u$ in $V_{\alpha}$, where $V_{\alpha}$ 
is a Kuranishi neighborhood of $u$.
\par
Then there exists a finite dimensional vector space $W(u)$ and a section
$
s_u : V \times W(u) \to E(u)\times W(u)
$
with the following properties.
\smallskip
\begin{enumerate}
\item $s_u(v,0) = D_v\overline{\partial}$.
\item $s_u(v,w) = D_v\overline{\partial}$ if $v \notin V_1$.
\item If $D_v\overline{\partial} = 0$ and $v \in V_0$ then the map
$$
D_{(v,0)} s_u : T_vV \times T_0W(u) \to E(u)
$$
is surjective.
\item If $v$ is as in $3$ and if $z_0 \in \partial D^2$, then 
$$
D_{(v,0)} s_u  \oplus d_{v,0}Ev_{z_0} : T_vV \times T_0W(u) \to E(u)
\oplus T_{v(z_0)}L
$$
is surjective.
\end{enumerate}
\end{lmm}
\begin{proof}
Let $\chi : V \to [0,1]$ be a smooth function such that $\chi \equiv 1$ on $V_0$ 
and $\chi \equiv 0$ on the complement of $V_1$. 
We put $W(u) = E(u)$ and 
$$
s_u(v,w) = D_v\overline{\partial} + \chi(v)w.
$$
The lemma then follows easily from (\ref{evalatz0}).
\end{proof}
We consider a Kuranishi chart $V_{\alpha}$ of $\mathcal M_0(\beta)$.
For simplicity of notation, we assume that each element of $V_{\alpha}$ is 
represented by a map from a disc (without singular point).
Let $u\in V_{\alpha}$.
We apply Lemma \ref{lemmulticonst} and obtain $V_0$, $V_1$, $V$, 
which we denote by $V_0(u)$, $V_1(u)$, $V(u)$.
We also have $E(u)$ and obtain $s_u$, $W(u)$. 
(We remark that $E(u)$ here is the restriction of $E_{\alpha}$, the 
obstruction bundle of our Kuranishi chart.)
Let $\Gamma(u) \subset \Gamma_{\alpha}$ be the isotropy group of $u$.
We may assume that 
$U(u) \cap \gamma V(u) \ne \emptyset$, $\gamma \in \Gamma_{\alpha}$ 
implies that $\gamma \in \Gamma(u)$ and that 
$V_0(u)$, $V_1(u)$, $V_2(u)$ are $\Gamma(u)$ invariant.
(We do not (and can not) require $s_u$ to be $\Gamma(u)$ equivariant.)
\par
We take $\Gamma_{\alpha}$ invariant relatively compact subset $V'_{\alpha} \subset V_{\alpha}$ such that 
$V'_{\alpha}/\Gamma_{\alpha}$ still cover $\mathcal M_{0}(\beta)$.
$V'_{\alpha}/\Gamma_{\alpha}$ is covered by finitely many of $V_0(u)/\Gamma(u)$'s, which 
we write $V_0(u_c)/\Gamma(u_c)$, $c=1,\ldots, N$. 
We put
$$
W_{\alpha} = W(u_1) \times \ldots \times W(u_N). 
$$
For each $c$, the map $s_{u_c}$ determines a multisection on $V_{\alpha}/\Gamma_{\alpha}$
which coincides with $s_{\alpha}$ ourside $V_1(u_c)/\Gamma(u_c)$. 
(Lemma \ref{lemmulticonst}.2.) We denote it by 
$s_{u_c}$ by abuse of notation. 
Let $\chi_c : V_{\alpha}/\Gamma(u_c) \to [0,1]$ be a 
partition of unity subordinate to the covering $\{V_1(u_c)/\Gamma(u_c)\mid c=1,
\ldots, N\}$. We now put
\begin{equation}\label{stasu}
\frak s_{\alpha}(v;w_1,\ldots,w_I) 
= s_{\alpha}(v) + \sum_i \chi_c(v) \left( s_c(v,w) - s_{\alpha}(v) \right).
\end{equation}
(See \cite{FO} Definition 3.4 for the sum of multisections, which appear in the right hand side of 
(\ref{stasu}).)  Now we have:
\begin{lmm}\label{famimultilocconst}
There exists $\epsilon >0$ such that the following holds.
\smallskip
\begin{enumerate}
\item $\frak s_{\alpha}$ is transversal to $0$.
\item  We take $U_i \subset V'_{\alpha}/\Gamma_{\alpha}$ such that 
$\frak s_{\alpha}$ has a representative 
$(\frak s_{\alpha,i,j})_{j=1}^{l_j}$ on $U_i \times W$. We put
$$
\frak s_{\alpha,i,j}^{-1}(0) 
= \{ (v,w) \in U_i \times W \mid 
\frak s_{\alpha,i,j}(v,w) = 0, \,\,\, \Vert w\Vert \le \epsilon \}.
$$
This is a manifold by $1$. Then for each $z_0 \in \partial D^2$ the map
$$
Ev_{z_0} : \frak s_{\alpha,i,j}^{-1}(0)  \to L
$$
defined by $(v,w) \mapsto Ev_{z_0}(v,w) = v(z_0)$, is a submersion.
\item $\frak s_{\alpha,i,j}(v,0) = s_{\alpha}(v)$.
\end{enumerate}
\end{lmm}
\begin{proof}
1,2,3 follows from Lemma \ref{lemmulticonst} 3,4,1 respectively.
\end{proof}
Lemma \ref{famimultilocconst} enables us to construct a 
family of multisections required in Theorem \ref{multicontforgetmain} locally. 
We can prove the relative version of Lemma \ref{famimultilocconst} 
also in the same way.
\begin{rem}
We remark that we can use Lemma \ref{famimultilocconst}.3 to prove the property 1 in 
Theorem  \ref{multicontforgetmain}. Namely we can shrink $W$ to its 
small neighborhood of $0$. Then Lemma \ref{famimultilocconst}.3
implies that $\frak s_{\alpha,i,j}$ is $C^0$ close to the Kuranishi map 
$s_{\alpha}$.
\end{rem}
We next state the analog of Lemma \ref{closedmoduliKura}.
\begin{lmm}\label{closedmoduliKurapert}
For each $E_0$ and $\ell_0$, 
there exists a system of continuous family of 
multisections on various 
$\mathcal M_{\ell}^{\text{\rm cl}}(\alpha)$ for various $\ell \ge 0$ and 
$\alpha \in \pi_2(M)$ with $\alpha\cap \omega \le E_0$ and $\ell \le \ell_0$.
It has the following properties.
\par\smallskip
\begin{enumerate}
\item The action of permutation group of order $\ell !$ on $\mathcal M^{\text{\rm cl}}_{\ell}(\alpha)$ exchanging 
marked points preserves the family of 
multisections.
\item It is transversal to $0$.
\item The evaluation map $ev^{\text{\rm int}} : \mathcal M_{\ell}^{\text{\rm cl}}(\alpha) \to M^{\ell}$ 
is a submersion on the zero set.
\item 
Let $\ell_1 + \ell_2 = \ell + 2$, $\alpha_1 + \alpha_2 = \alpha$. Then 
the embedding 
$$
\mathcal M^{\text{\rm cl}}_{\ell_1}(\alpha_1) {}_{ev^{\text{\rm int}}_1}\times_{{ev^{\text{\rm int}}_1}}
\mathcal M^{\text{\rm cl}}_{\ell_2}(\alpha_2)
\subset \mathcal M^{\text{\rm cl}}_{\ell}(\alpha)
$$
is compatible with our system of continuous families of multisections.
\end{enumerate}
\end{lmm} 
\begin{proof}
The proof is by induction on $(\alpha,\ell)$. We can 
organize the order of the induction in the same way as \cite{FOOO080} Section 7.2.
4 determines the family of multisections in the singular locus of each 
of $ \mathcal M^{\text{\rm cl}}_{\ell}(\alpha)$.
By 2,3 of the fiber product factors, we can prove that 
2,3 are satisfied on the singular locus. We can then 
use the argument of the last section to extend it to 
$ \mathcal M^{\text{\rm cl}}_{\ell}(\alpha)$.
\end{proof}
Note we do not require compatibility with forgetful map here. 
So the proof of Lemma \ref{closedmoduliKurapert} is easier than that 
of Theorem \ref{multicontforgetmain}.
\par
Now we are in the position to complete the proof of Theorem \ref{multicontforgetmain}.
The proof is by induction on $(\beta,\ell)$. 
We assume that we have constructed the required family of multisections 
on $\mathcal M_{\ell',0}(\beta')$ for $\beta'\cap \omega 
< \beta\cap \omega$ ad $\ell' \le \ell$.
We study the case of  $\mathcal M_{\ell,0}(\beta)$.
We consider the boundary  $\partial\mathcal M_{\ell,0}(\beta)$.
One of its boundary component is 
\begin{equation}\label{discfpro}
\mathcal M_{\ell_1,1}(\beta_1) {}_{ev_0}\times_{ev_0}
\mathcal M_{\ell_2,1}(\beta_2)
\end{equation}
where $\ell_1 + \ell_2 = \ell + 2$, $\alpha_1 + \alpha_2 = \alpha$.
Continuous family of multisections on the factors of (\ref{discfpro}) 
is already given by induction hypothesis.
Moreover the properties 3,4 of the factors of (\ref{discfpro})  imply that 
the same properties for the fiber product family of multisections.
\par
For the other type of boundary component
\begin{equation}
\mathcal M^{\text{\rm cl}}_{\ell}(\tilde\beta) {}_{ev^{\text{\rm int}}_1}\times_M L
\end{equation}
we can apply Lemma \ref{closedmoduliKurapert} to 
obtain the required family of multisections.
We thus obtain family of multisections which has required 
properties.
\par
Those families of the multisections on the components of 
the boundary $\partial \mathcal M_{\ell,0}(\beta)$ 
is consistent at the overlapped part.
(See \cite{FOOO080} Lemma 7.2.55.)
\par
We can then extend it to a neighborhood of the boundary.
It is easy to see that this extended one 
still has the required transversal properties.
Therefore we can use relative version of Lemma \ref{famimultilocconst}
to discuss in the same way as the last section 
to extend this family of multisections to 
$\mathcal M_{\ell,0}(\beta)$.
Since we have only finitely many steps to work out, 
we can choose our family so that 1 of Theorem \ref{multicontforgetmain}
is satisfied. The proof of Theorem \ref{multicontforgetmain} is now complete.
\end{proof}
\begin{rem}
See \cite{FOOO080} Subsection 7.2.3 for the reason why we need to 
fix $E_0$, $\ell_0$ and stop the construction of the 
multisections after $\beta\cap \omega > E_0$ or $\ell > \ell_0$. 
\end{rem}
In the same way as Section \ref{Seckuranishiforget}, Theorem \ref{multicontforgetmain} has 
the following corollary.
\begin{crl}\label{Corkuramulti}
For each $\epsilon$, $E_0$ and $\ell_0$, 
there exists a system of continuous family of multisections on 
$\mathcal M_{\ell,k+1}(\beta)$ $k\ge 0$, $\ell_0\ge \ell\ge 0$,
$\beta \cap \omega \le E_0$, 
with the following properties.
\par\smallskip
\begin{enumerate}
\item It is $\epsilon$-close to the Kuranishi map in $C^0$ sense.
\item It is compatible with $\mathfrak{forget}_{k+1,1}$.
\item It is invariant under the cyclic permutation of the 
boundary marked points.
\item It is invariant by the permutation of interior marked points.
\item $ev_0 : \mathcal M_{\ell,k+1}(\beta) \to L$ induces a submersion on its zero set.
\item We consider the decomposition of the boundary:
\begin{equation}\label{bdcompati}
\aligned
\partial \mathcal M_{\ell,k+1}(\beta)
=
\bigcup_{1\le i\le j+1 \le k+1} 
&\bigcup_{\beta_1+\beta_2=\beta}
\bigcup_{L_1\cup L_2=\{1,\ldots,\ell\}} \\
&\mathcal M_{\#L_1,j-i+1}(\beta_1) {}_{ev_0} \times_{ev_i} 
\mathcal M_{\#L_2,k-j+i}(\beta_2).
\endaligned
\end{equation}
(See {\rm \cite{FOOO080}} Subsection {\rm 7.1.1}.) Then the restriction of our family of multisections of 
$\mathcal M_{\ell,k+1}(\beta)$ in the left hand side 
coincides with the fiber product family of multisections in 
the right hand side.
\end{enumerate}
\end{crl}
\section{Cyclic filtered $A_{\infty}$ algebra and cyclic filtered $A_{\infty}$ algebra modulo $T^E$}
\label{cyclicmodTEsec}
Let $\overline C$ be a graded vector space over $\R$ and $n$ be a positive integer.
We conider $\R$ bilinear map
\begin{equation}\label{innerprod}
\langle \cdot \rangle : \overline C^k \otimes \overline C^{n-k} \to \R.
\end{equation}
such that
\begin{equation}\label{gradedansym}
\langle x,y \rangle = (-1)^{1+\deg'x\deg'y}\langle y,x \rangle.
\end{equation}
Here and hereafter
$$
\deg' x = \deg x -1.
$$
We consider a sequence of operators
\begin{equation}\label{mkopewithfil}
\overline{\frak m}_k : B_k(\overline C[1]) \to \overline C[1]
\end{equation}
of degree $1$, for $k=1,2,\ldots$.
(Here $B_k(\overline C[1])$ is the tensor product of $k$ copies 
of $\overline C[1]$.)
\begin{dfn}\label{cyclicainf}
We say $(\overline C,\langle \cdot \rangle,\{\overline{\frak m}_k\}_{k=1}^{\infty})$
is a {\it cyclic $A_{\infty}$ algebra} of dimension $n$ if:
\smallskip
\begin{enumerate}\label{ainfty1}
\item $\{\overline{\frak m}_k\}_{k=1}^{\infty}$ satisfies the 
following $A_{\infty}$ relation:
\begin{equation}\label{ainfty12}
\sum_{k_1+k_2=k+1}\sum_{i=1}^{k-k_2+1}
(-1)^*\overline{\frak m}_{k_1}(x_1,\ldots,\overline{\frak m}_{k_2}(x_i,\ldots,x_{i+k_2-1}),\ldots,x_{k}) = 0,
\end{equation}
where $* = \deg'x_1+\ldots+\deg'x_{i-1}$.
\item We have:
 \begin{equation}\label{cyclicform2}
\langle 
\overline{\frak m}_k(\text{\bf x}_1,\ldots,\text{\bf x}_k),\text{\bf x}_0
\rangle
= 
(-1)^*\langle 
\overline{\frak m}_k(\text{\bf x}_0,\text{\bf x}_1,\ldots,\text{\bf x}_{k-1}),\text{\bf x}_{k} 
\rangle,
\end{equation}
where $* = \deg'x_0(\deg' x_1+\ldots +\deg'x_k)$.
\item $\langle \cdot \rangle$ is non-degenerate and induces a perfect 
pairing on $H(\overline C)= \text{\rm Ker}\overline{\frak m}_{1}/\text{\rm Im}\overline{\frak m}_{1}$.
\end{enumerate}
\end{dfn}
We remark that $\overline{\frak m}_{1} \circ \overline{\frak m}_{1}=0$ by 1. Therefore $H(\overline C)$ 
is well-defined. 2 implies 
$
\langle \overline m_1(x),y\rangle = \pm \langle \overline m_1(y),x\rangle
$
Therefore $\langle \cdot \rangle$ induces one on $H(\overline C)$ which 
satisfies  (\ref{gradedansym}).
\begin{exa}\label{derhamAinift}
Let $M$ be an $n$-dimensional oriented closed manifold.
Let $\overline C= \Lambda(M)$ be the de-Rham complex. 
We put
$$
\overline{\frak m}_1(u) = (-1)^{\deg u}du, \quad
\overline{\frak m}_2(u,v) = (-1)^{\deg u\deg v + \deg u}u \wedge v,
$$
$\overline{\frak m}_k = 0$ for $k\ne 1,2$, and 
$$
\langle u,v\rangle = (-1)^{\deg u\deg v + \deg u}\int_M u \wedge v.
$$
\end{exa}
\begin{dfn}\label{discmonoid}
We say a subset $G \subset \R_{\ge 0}\times 2\Z$ a {\it discrete submonoid} if 
the following holds. 
We denote by $E : G  \to \R$ and $\mu : G  \to 2\Z$ the projections to each of 
the component.
\smallskip
\begin{enumerate}
\item If $\beta_1,\beta_2 \in G$ then $\beta_1 + \beta_2 \in G$.
$(0,0) \in G$.
\item The image $E(G) \subset \R_{\ge 0}$ is discrete.
\item For each $\lambda \in \R_{\ge 0}$ the inverse image 
$G \cap E^{-1}(\lambda)$ is a finite set.
\end{enumerate}
\end{dfn}
We remark that 1,2,3 imply that $E^{-1}(0) \cap G = \{(0,0)\}$.
\begin{dfn}\label{Gnovikov}
We put
$$
\Lambda_0^G = \left\{ \sum a_{\beta}T^{E(\beta)}e^{\mu(\beta)/2} \mid \beta \in G,\,\, a_{\beta} 
\in \R \right\},
$$
where the sum may be either finite or infinite sum. For each $E_0 \in \R_{\ge 0}$
we have a filtration
$$
F^{E_0}\Lambda_0^G = \left\{ \sum a_{\beta}T^{E(\beta)}e^{\mu(\beta)/2}\in \Lambda^G \mid E(\beta)\ge E_0 
\right\}.
$$
It determines a topology on $\Lambda_0^G$ with which $\Lambda_0^G$ is complete.
\par
We define a grading on $\Lambda_0^G$ by $\deg T = 0$, $\deg e =2$.
\end{dfn}
We remark that $\Lambda_0^G$ contains a (semi)-group ring of the monoid $G$ that is 
nothing but the element of $\Lambda_0^G$ with only finitely many nonzero $a_{\beta}$.
$\Lambda_0^G$ is its completion.
\par
The universal Novikov ring $\Lambda_{0,nov}$ which was introduced by \cite{FOOO00} is 
the union (inductive limit) of all $\Lambda_0^G$ for various discrete submonoids $G$.
\begin{dfn}
A $G$-gapped cyclic filtered $A_{\infty}$ algebra structure on $\overline C$ is a 
sequence of opeartors
\begin{equation}\label{mkbeta}
\frak m_{k,\beta} : B_k(\overline C[1]) \to \overline C[1]
\end{equation}
for each $\beta\in G$ and $k\in \Z_{\ge 0}$ of degree $1-\mu(\beta)$ 
and $\langle \cdot \rangle$ with the following properties.
\smallskip
\begin{enumerate}
\item
$\frak m_{0,\beta} = 0$ for $\beta = (0,0)$.
\item 
\begin{equation}\label{Ainfinityrelbeta}
\aligned
\sum_{k_1+k_2=k+1}&\sum_{\beta_1+\beta_2=\beta}\sum_{i=1}^{k-k_2+1} \\
&(-1)^*{\frak m}_{k_1,\beta_1}(x_1,\ldots,{\frak m}_{k_2,\beta_2}(x_i,\ldots,x_{i+k_2-1}),\ldots,x_{k}) = 0,
\endaligned\end{equation}
holds for any $\beta \in G$ and $k$. Here the sign is as in (\ref{ainfty1}).
\item $\langle \cdot \rangle$ satisfies
(\ref{innerprod}), (\ref{gradedansym}) and 
\begin{equation}\label{cyclicform22}
\langle 
{\frak m}_{k,\beta}(\text{\bf x}_1,\ldots,\text{\bf x}_k),\text{\bf x}_0
\rangle
= 
(-1)^*\langle 
{\frak m}_{k,\beta}(\text{\bf x}_0,\text{\bf x}_1,\ldots,\text{\bf x}_{k-1}),\text{\bf x}_{k} 
\rangle,
\end{equation}
holds for any $\beta$ and $k$. Here the sign is as in (\ref{cyclicform2}).
3 of Definition \ref{cyclicainf} also holds.
\end{enumerate}
\end{dfn}
\begin{dfn}
A G-gapped cyclic filtered $A_{\infty}$ algebra structure modulo $T^{E_0}$ 
is sequence of operators (\ref{mkbeta}) for $E(\beta) < E_0$ 
which satisfies the same properties, except (\ref{Ainfinityrelbeta})  and (\ref{cyclicform22})
are assumed only for $E(\beta) < E_0$.
\end{dfn}
\begin{dfn}
$\text{\bf e} \in \overline C$ of degree $0$ is said to be a 
(strict) unit of $A_{\infty}$ algebra if
\begin{equation}\label{unitalityeq}
\aligned
\frak m_2(\text{\bf e},x) &= (-1)^{\deg x} \frak m_2(x,\text{\bf e}) = x,  \\
\frak m_k(\ldots,\text{\bf e},\ldots) &= 0\quad\text{for $k\ne 2$}.
\endaligned\end{equation}
$\text{\bf e}$ is said to be a unit of filtered $A_{\infty}$ algebra if
\begin{equation}\aligned
\frak m_{2,(0,0)}(\text{\bf e},x) &= (-1)^{\deg x} \frak m_{2,(0,0)}(x,\text{\bf e}) = x,  \\
\frak m_{k,\beta}(\ldots,\text{\bf e},\ldots) &= 0\quad\text{for $(k,\beta)\ne (2,(0,0))$}.
\endaligned\end{equation}
Unit for filtered $A_{\infty}$ algebra modulo $T^E$ is defined in the same way.
(filtered) $A_{\infty}$ algebra with unit is said to be unital.
\end{dfn}
We put
$$
\frak m_k = \sum_{\beta} T^{E(\beta)} e^{\mu(\beta)/2} \frak m_{k,\beta} : 
B_k(C[1]) \to C[1]. 
$$
It satisfies (\ref{ainfty12}), (\ref{cyclicform2}).
\begin{rem}
The sign convention of this paper is slightly different from one in \cite{FOOO080}.
We explain the difference below.
In \cite{FOOO080} the inner product 
\begin{equation}\label{()and<>}
(u,v) = (-1)^{\deg u\deg v}\int_L u\wedge v = (-1)^{\deg u} \langle u,v\rangle
\end{equation}
is used. 
(See \cite{FOOO080} Remark 8.4.7 (1).)
Then the cyclic symmetry of $\frak m_k$ takes the form
\begin{equation}\label{()cyclic}
(x_0,\frak m_{k,\beta}(x_1,\ldots,x_k)) = (-1)^*(x_1,\frak m_{k,\beta}(x_2,\ldots,x_{k},x_0)),
\end{equation}
where $* = \deg'x_0(\deg' x_1+\ldots+\deg'x_k)$.
(See \cite{FOOO080} Proposition 8.4.8.)
\begin{lmm}\label{hikakucyclicconv}
If $(\cdot)$ and $\langle \cdot \rangle$ are related by $(\ref{()and<>})$, 
then $(\ref{cyclicform22})$ is equivalent to $(\ref{()cyclic})$.
\end{lmm}
\begin{proof}
We assume (\ref{()and<>}) and (\ref{()cyclic}). Then we calculate
$$\aligned
&\langle \frak m_{k,\beta}(x_1,\ldots,x_k),x_0\rangle \\
&= (-1)^{\sum_{i=1}^k\deg'x_i}
(\frak m_{k,\beta}(x_1,\ldots,x_k),x_0)\\
&= (-1)^{\left(\sum_{i=1}^k\deg'x_i\right)(\deg'x_0)}
(x_0, \frak m_{k,\beta}(x_1,\ldots,x_k)) \\
&= 
(x_1, \frak m_{k,\beta}(x_2,\ldots,x_k,x_0)) \\
&= 
(-1)^{*_1} (\frak m_{k,\beta}(x_2,\ldots,x_k,x_0),x_1) \\
&= (-1)^{*_2} \langle\frak m_{k,\beta}(x_2,\ldots,x_k,x_0),x_1\rangle.
\endaligned$$
Here 
$$
*_1 = \left(\sum_{i\ne 1}\deg'x_i\right)(\deg'x_1 + 1),
$$
and
$$
*_2 = *_1 + \sum_{i\ne 1}\deg'x_i
= \left(\sum_{i\ne 1}\deg'x_i\right)\deg'x_1.
$$
We thus have $(\ref{cyclicform22})$.
The proof of the converse is similar.
\end{proof}
\end{rem}
By Lemma \ref{hikakucyclicconv} we can apply the discussion of 
orientation and sign in \cite{FOOO080}  Subsection 8.4.2 for the purpose of 
this paper.
\section{Cyclic filtered $A_{\infty}$ structure modulo $T^E$ on de-Rham complex}
\label{cyclicTEderamsec}
The main result of this section is as follows.
\begin{thm}\label{existifilaAinfmodTE}
For any relatively spin Lagrangian submanifold $L$ of $(M,\omega)$. We can assign $G$ such that for each
$E_0>0$ the de Rham complex $\Lambda(L)$ has a structure of $G$ gapped cyclic unital filtered $A_{\infty}$ algebra modulo $T^{E_0}$.
\end{thm}
\begin{proof}
This theorem follows from Corollary \ref{Corkuramulti} as follows.
We fix tame almost complex structure $J$ and let $G$ be the 
submonoid of $\R_{\ge 0} \times 2\Z$ generated by the set
$$
\{(\beta\cap \omega,\mu(\beta)) \mid \beta \in H_2(X,L;\Z), 
\,\, \mathcal M_0(\beta) \ne \emptyset \}.
$$
Gromov compactness implies that $G$ satisfies the conditions in Definition \ref{discmonoid}.
\par
We apply Corollary \ref{Corkuramulti} by putting $E_0$ as above and $\ell =0$.
Then for $\rho_1,\ldots,\rho_k \in \Lambda(L)$, we define
\begin{equation}\label{geodef}
\aligned
&\frak m_{k,\beta}(\rho_1,\ldots,\rho_k) \\
&= \text{\rm Corr}(\mathcal M_{k+1}(\beta);(ev_1,\ldots,ev_k),ev_0)(\rho_1\times \ldots \rho_k)
\in \Lambda(L),
\endaligned\end{equation}
for $\beta \ne (0,0)$. 
Since $ev_0$ is weakly submersive by Corollary \ref{Corkuramulti}.5, the right hand 
side is well-defined and is a smooth form.
We define $\frak m_{k,\beta_0}$ for $\beta_0 = (0,0)$ and $\langle \cdot \rangle$ 
as in (\ref{derhamAinift}).
\par
The cyclic symmetry follows from Corollary \ref{Corkuramulti}.4 (up to sign) as follows.
By Lemma \ref{Corandint2} and definitions we have:
\begin{equation}\label{tobecycilc}
\aligned
\langle 
&{\frak m}_{k,\beta}(\rho_1,\ldots,\rho_k),\rho_0
\rangle\\
&= \text{\rm Corr}(\mathcal M_{k+1}(\beta),(ev_1,\ldots,ev_k,ev_0),\text{\rm const})
(\rho_1\times \dots \times \rho_k\times \rho_0)
\endaligned
\end{equation}
By Corollary \ref{Corkuramulti}.4 the right hand side of  (\ref{tobecycilc}) 
is cyclically symmetric.
\par
Filtered $A_{\infty}$ relation (\ref{Ainfinityrelbeta}) is a consequence of 
Corollary \ref{Corkuramulti}.6 (up to sign) and
Propositions \ref{stokes} and \ref{comprop}. (See \cite{FOOO080} Section 3.5.)
\par
The (strict) unitality (\ref{unitalityeq}) follows from Corollary \ref{Corkuramulti}.2
as follows.
We consider $\beta \ne (0,0)$ and $k\ge 1$ and studies
\begin{equation}\label{zeroifunit}
\frak m_{k,\beta}(\rho_1,\ldots,\rho_{i-1},1,\rho_{i+1},\ldots,\rho_k).
\end{equation}
Here $1 \in \Lambda^0(L)$ is the 0 form $\equiv 1$.
We consider the forgetful map 
$$
\frak{forget}_i : \mathcal M_{k}(\beta) \to \mathcal M_{k-1}(\beta)
$$
which forgets $i$-th marked point.
We have chosen the Kuranishi structure and the (family of) multisections 
are invariant of forgetful map.
We consider the expression (\ref{localdefofcorr}).
 \begin{equation}\label{recall4.8}
\sum_i  f^{t}_{\alpha !} \frac{1}{l_{i}}\chi_i \left( (f^s_{\alpha})^* \rho \wedge \omega_{\alpha} \right)
\vert_{\frak s_{\alpha,i,j}}
\end{equation}
In our case $\rho = \rho_1 \times \ldots \times \rho_{i-1} \times 1\times \rho_{i+1} \times \ldots
\times \rho_k$.
Let $X$ be the vector field tangent to the fiber of $\frak{forget}_i$. Then we have
$$
i_X((f^s_{\alpha})^* \rho \wedge \omega_{\alpha})= 0.
$$
On the other hand, $f^t = ev_0$ factors through $\frak{forget}_i$.
Therefore, (\ref{recall4.8}) is zero in our case.
Since (\ref{zeroifunit}) is obtained from  (\ref{recall4.8}) by summing up using partition of 
unity, it follows that (\ref{zeroifunit}) is zero also. We thus proved the strict unitality.
\par
We finally remark that we can handle the sign in the same way as \cite{FOOO080} Subsection 8.10.3 
and \cite{Fuk07I} Section 12. Namely we can reduce the sign in the de Rham version to one in singular 
homology version.
\end{proof}
\begin{rem}
We remark that the evaluation map $ev_0$ from our perturbed moduli space 
is submersive. This is enough to work with de Rham theory since 
we can pull back differential forms by $ev_1,\ldots,ev_k$ without assuming 
its submersivity. 
\par
In case we work with (singular) chains, we need to take 
fiber product of singular chains $P_i$ of $L$ with our perturbed moduli space  
by evaluation  maps $ev_1,\ldots,ev_k$. 
This requires submersivity of $ev_1,\ldots,ev_k$. As we explained in 
Remark \ref{remwksubmer} it is impossible to do so while keeping compatibility with 
forgetful map.
\end{rem}
\section{Homological algebra of cyclic filtered $A_{\infty}$ algebra: statement}
\label{homalgsec}
In Sections \ref{homalgsec}-\ref{canoalgsec} we study homological 
algebra of cyclic filtered $A_{\infty}$ algebra. 
We follow \cite{Ka} in various places. 
(See also \cite{Cho}.) However our 
discussion is different from \cite{Ka} not only we study 
filtered case also but also in several other points.
In \cite{Ka} $\langle u,v\rangle \ne 0$ only when 
$\deg u+\deg v$ is odd. Thus our situation is included in  \cite{Ka} 
when $\dim L$ is odd. In that case 
according to Remark 2.1.2 \cite{Ka}, the sign convention 
of cyclic symmetry in \cite{Ka} is similar to  (\ref{()cyclic})
(that is the convention of \cite{FOOO080}).
The notion of pseudo-isotopy of cyclic filtered $A_{\infty}$ algebra 
which we introduce in this section is also new.
\par
We first review the definition of filtered $A_{\infty}$ homomorphism from 
\cite{FOOO080} Chapter 4.
Let $G \subset \R_{\ge 0} \times 2\Z$ be a discrete submonoid and 
$(C,\{\frak m_{k,\beta}\})$, $(C',\{\frak m'_{k,\beta}\})$ be $G$-gapped 
filtered $A_{\infty}$ algebras.
\begin{dfn}
A sequence of $\R$-linear maps
$$
\frak f_{k,\beta} : B_k(\overline C[1]) \to \overline C[1]
$$
of degree $-\mu(\beta)$ for $k=0,1,2,\ldots$, $\beta \in G$ is said to be a 
$G$-{\it gapped filtered $A_{\infty}$ homomorphism} if:
\smallskip
\begin{enumerate}
\item $\frak f_{0,(0,0)} = 0$.
\item 
For each $\beta$, $k =0,1,2,\ldots$, with $(k,\beta) \ne (0,(0,0))$ and 
$x_1,\ldots,x_k \in \overline C[1]$, we have
\begin{equation}\label{inftyhomoeq}
\aligned
&\sum \frak m_{\ell,\beta_0}
\left(
\frak f_{k_1,\beta_1}(x_1,\ldots),\ldots,
\frak f_{k_{\ell},\beta_{\ell}}(\ldots,x_k)
\right)
\\
&= \sum_{k_1+k_2=k+1}\sum_{\beta_1+\beta_2=\beta}
\sum_{i=1}^{k-k_2+1}
(-1)^*\frak f_{\beta_1,k_1}(x_1,\ldots,\frak m_{k_2,\beta_2}(x_i,\ldots),\ldots,x_k)
\endaligned
\end{equation}
Here the sum in the left hand side is taken over $\ell$, $\beta_i$, $k_i$ such that 
$\beta_0+\beta_1+\ldots+\beta_{\ell} = \beta$ and $k_1+\ldots+k_{\ell} = k$, and 
the sign in the right hand side is $*=\deg'x_1+\ldots+\deg'x_{i-1}$.
\end{enumerate}
\par\smallskip
In case $\frak f_{k,\beta}$ is defined only for $E(\beta) < E_0$ and 
(\ref{inftyhomoeq}) holds only for  $E(\beta) < E_0$, 
we call it a filtered $A_{\infty}$ homomorphism modulo $T^{E_0}$.
This is defined if $(C,\{\frak m_{k,\beta}\})$, $(C',\{\frak m'_{k,\beta}\})$ are filtered $A_{\infty}$
algebra modulo $T^{E_0}$.
\end{dfn}
\begin{dfn}
\begin{enumerate}
\item A filtered $A_{\infty}$ homomorphism is said to be {\it strict} if $\frak f_{0,\beta}=0$ for any $\beta$.
\item Suppose $(C,\{\frak m_{k,\beta}\})$, $(C',\{\frak m'_{k,\beta}\})$ are unital. We say a filtered $A_{\infty}$
homomorphism to be {\it unital} if:
\begin{equation}
\aligned
&\frak f_{1,\beta}(\text{\bf e}) 
= \begin{cases}
\text{\bf e}   &\text{if $\beta= (0,0)$}, \\
0   &\text{if $\beta\ne (0,0)$}.
\end{cases}
\\
&\frak f_{k,\beta}(\ldots,\text{\bf e},\ldots) = 0, \qquad k> 1.
\endaligned
\end{equation}
\item We say a filtered algebra or homomorphism to be {\it gapped} in case 
it is $G$-gapped for some $G$, which we do not specify.
\end{enumerate}
\end{dfn}
\begin{dfn}[\cite{Ka} Definition 2.13]
Let $(C,\langle \cdot \rangle,\{\frak m_{k,\beta}\})$, $(C',\langle \cdot \rangle,\{\frak m'_{k,\beta}\})$ be
$G$-gapped cyclic filtered $A_{\infty}$ algebras.
A $G$-gapped filtered $A_{\infty}$ homomorphism $\frak f = \{\frak f_{k,\beta}\} : C \to C'$ is said to be 
{\it cyclic} if the following holds.
\begin{equation}\label{homocyc1}
\langle \frak f_{1,(0,0)}(x),\frak f_{1,(0,0)}(y) \rangle = \langle x,y \rangle
\end{equation}
for any $x,y$.
\begin{equation}\label{homocyc2}
\sum_{\beta_1+\beta_2=\beta}\sum_{k_1+k_2=k} \langle \frak f_{k_1,\beta_1}(x_1,\ldots,x_{k_1}),\frak f_{k_2,\beta_2}(x_{k_1+1},\ldots,x_k) \rangle =
0,
\end{equation}
holds  for $(k,\beta) \ne (2,(0,0))$ and $x_1,\ldots,x_k$.
\par
In case (\ref{homocyc2}) holds only for $E(\beta) < E_0$ we say that it is cyclic modulo $T^{E_0}$.
\par
We define composition of filtered $A_{\infty}$ homomorphism as in \cite{FOOO080} Definition 3.2.31.
It is easy to see that composition of cyclic  filtered $A_{\infty}$ homomorphisms are cyclic.
\end{dfn}
\begin{rem}\label{misremonequiv}
\begin{enumerate}
\item
We say a (cyclic) filtered $A_{\infty}$ homomorphism $\{\frak f_{k,\beta}\}$ to be a {\it weak homotopy equivalence} 
if $\frak f_{1,(0,0)}$  induces isomorpism on $\frak m_{1,(0,0)}$.
\item
We say a  filtered $A_{\infty}$ homomorphism $\{\frak f_{k,\beta}\}$ is an {\it isomorphism}
if it has an inverse, that is a filtered $A_{\infty}$ homomorphism $\{\frak g_{k,\beta}\}$ 
such that the compositions of them are identity. (An identity morphism 
$\{\frak h_{k,\beta}\}$ is defined by $\frak h_{k,\beta} = 0$ for $(k,\beta) \ne (1,(0,0))$ and 
$\frak h_{1,(0,0)} = $identity. 
It is easy to see that an inverse of cyclic filtered $A_{\infty}$ homomorphism is automatically cyclic.
\end{enumerate}
\end{rem}
\par
We next define and study the properties of pseudo-isotopy of cyclic filtered $A_{\infty}$ algebras.
Let $\overline C$ be a graded $\R$ vector space. 
We take a basis $\frak e_i$ of $\overline C^k$ and define 
$C^{\infty}([0,1],\overline C^k)$ to be a finite sum
$$
\sum_i a_i(t) \frak e_i
$$
such that $a_i : [0,1] \to \R$ are smooth functions.
\begin{rem}
In case $\overline C = \Lambda(L)$ the de Rham complex, we need to 
take into acount the Fr\'echet topology of $\Lambda(L)$ and 
define 
$C^{\infty}([0,1],\overline C)$ in a different way. See Section \ref{geoisoto}.
\end{rem}
We consider the set of formal sums
\begin{equation}\label{elementdt}
a(t) + dt \wedge b(t)
\end{equation}
where $a(t) \in C^{\infty}([0,1],\overline C^k)$, 
$b(t) \in C^{\infty}([0,1],\overline C^{k-1})$. We write the totality of 
such (\ref{elementdt}) by
$
C^{\infty}([0,1] \times \overline C)^k.
$
We will consider filtered $A_{\infty}$ structure on it. More precisely we proceed as follows.
\par
We assume that, for each $t\in [0,1]$, we have operations:
\begin{equation}\label{param}
\frak m^t_{k,\beta} : B_k(\overline C[1]) \to \overline C[1]
\end{equation}
of degree $-\mu(\beta)+1$ and
\begin{equation}\label{parac}
\frak c^t_{k,\beta} : B_k(\overline C[1]) \to \overline C[1]
\end{equation}
of degree $-\mu(\beta)$.
\begin{dfn}\label{smoothmtandct}
We say $\frak m^t_{k,\beta}$ is {\it smooth} if for 
each $x_1,\ldots,x_k$
$$
t\mapsto \frak m^t_{k,\beta}(x_1,\ldots,x_k)
$$
is an element of $C^{\infty}([0,1],\overline C)$.
\par
The smoothness of $\frak c^t_{k,\beta}$ is defined in the same way.
\end{dfn}
\begin{dfn}\label{pisotopydef}
We say $(C,\langle \cdot \rangle,\{\frak m^t_{k,\beta}\},\{\frak c^t_{k,\beta}\})$
is a {\it pseudo-isotopy} of $G$-gapped cyclic filtered $A_{\infty}$ algebras if the following holds:
\smallskip
\begin{enumerate}
\item  $\frak m^t_{k,\beta}$ and $\frak c^t_{k,\beta}$ are smooth.
\item For each (but fixed) $t$, the triple  $(C,\langle \cdot \rangle,\{\frak m^t_{k,\beta}\})$ defines a cyclic fitered $A_{\infty}$
algebra.
\item For each (but fixed) $t$, and $x_i \in \overline C[1]$, we have
\begin{equation}
\langle 
\frak c^t_{k,\beta}(x_1,\ldots,x_k),x_0
\rangle
= 
(-1)^*\langle 
\frak c^t_{k,\beta}(x_0,x_1,\ldots,x_{k-1}),x_{k} 
\rangle
\end{equation}
$* = (\deg \text{\bf x}_0+1)(\deg \text{\bf x}_1+\ldots + \deg \text{\bf x}_k + k)$.
\item  For each $x_i \in \overline C[1]$
\begin{equation}\label{isotopymaineq}
\aligned
&\frac{d}{dt} \frak m_{k,\beta}^t(x_1,\ldots,x_k) \\
&+ \sum_{k_1+k_2=k}\sum_{\beta_1+\beta_2=\beta}\sum_{i=1}^{k-k_2+1}
(-1)^{*}\frak c^t_{k_1,\beta_1}(x_1,\ldots, \frak m_{k_2,\beta_2}^t(x_i,\ldots),\ldots,x_k) \\
&- \sum_{k_1+k_2=k}\sum_{\beta_1+\beta_2=\beta}\sum_{i=1}^{k-k_2+1}
\frak m^t_{k_1,\beta_1}(x_1,\ldots, \frak c_{k_2,\beta_2}^t(x_i,\ldots),\ldots,x_k)\\
&=0.
\endaligned
\end{equation}
Here $* = \deg' x_1 + \ldots + \deg'x_{i-1}$.
\item 
$\frak m_{k,(0,0)}^t$  is independent of $t$. $\frak c_{k,(0,0)}^t = 0$.
\end{enumerate}
\end{dfn}
\begin{rem}
The condition 5 above may be a bit more restrictive than optional definition. We assume it here since it suffices for 
the purpose of this paper.
\end{rem}
We consider $x_i(t) + dt \wedge y_i(t) = \text{\bf x}_i \in C^{\infty}([0,1],\overline C)$.
We define
$$
\hat{\frak m}_{k,\beta}(\text{\bf x}_1,\ldots,\text{\bf x}_k)
= x(t) + dt \wedge y(t)
$$
where
\begin{subequations}\label{combineainf}
\begin{equation}
x(t) = {\frak m}^t_{k,\beta}(x_1(t),\ldots,x_k(t))
\end{equation}
\begin{equation}\label{combineainfmain}
\aligned
y(t) 
= 
& \frak c^t_{k,\beta}
(x_1(t),\ldots,x_k(t)) \\
&-\sum_{i=1}^k (-1)^{*_i} \frak m^t_{k,\beta}
(x_1(t),\ldots,x_{i-1}(t),y_i(t),x_{i+1}(t),\ldots,x_k(t))
\endaligned\end{equation}
if $(k,\beta) \ne (1,(0,0))$ and 
\begin{equation}
y(t) = \frac{d}{dt} x_1(t) + \frak m_{1,(0,0)}^t (y_1(t))
\end{equation}
if $(k,\beta) = (1,(0,0))$. Here $*_i$ in (\ref{combineainfmain}) is 
$*_i = \deg' x_1 +\ldots+\deg'x_{i-1}$.
\end{subequations}
\begin{lmm}\label{dtcomiAinf}
The equation $(\ref{isotopymaineq})$ is equivalent to the filtered $A_{\infty}$
relation of $\hat{\frak m}_{k,\beta}$ defined by $(\ref{combineainf})$.
\end{lmm}
The proof is straightforward and is omitted.
(See \cite{FOOO080} Lemma 4.2.13.)
We define 
$\langle \cdot \rangle_{t_0}$ on $C^{\infty}([0,1],\overline C)$ by
$$
\langle x_1(t) + dt \wedge y_1(t),x_2(t) + dt \wedge y_2(t) \rangle_{t_0}
= \langle x_1(t_0),x_2(t_0)\rangle.
$$
Then $(C^{\infty}([0,1],\overline C),\langle \cdot \rangle_{t_0},\{\hat{\frak m}_{k,\beta}\})$ 
is a cyclic filtered $A_{\infty}$ algebra for any $t_0$.
\begin{dfn}
A pseudo-isotopy 
$(C,\langle \cdot \rangle,\{\frak m^t_{k,\beta}\},\{\frak c^t_{k,\beta}\})$
is said to be {\it unital} 
if there exists $\text{\bf e} \in \overline C^0$ such that
$\text{\bf e}$ is a unit of $(C,\{\frak m^t_{k,\beta}\})$ for each $t$ and if 
$$
\frak c^t_{k,\beta}(\ldots,\text{\bf e},\ldots) = 0
$$
for each $k,\beta$ and $t$.
\end{dfn}
\begin{dfn}
Let $(\overline C,\langle \cdot \rangle,\{\frak m_{k,(0,0)}\})$ be a
(unfiltered) cyclic $A_{\infty}$ algebra. 
We consider two 
$G$-gapped filtered cyclic $A_{\infty}$ algebras 
$(\overline C,\langle \cdot \rangle,\{\frak m^{i}_{k,\beta}\})$ ($i=0,1$)
such that $\frak m^{i}_{k,(0,0)} = \frak m_{k,(0,0)}$ for $i=0,1$.
\par
We say $(\overline C,\langle \cdot \rangle,\{\frak m^{0}_{k,\beta}\})$ is 
pseudo-isotopic to 
$(\overline C,\langle \cdot \rangle,\{\frak m^{1}_{k,\beta}\})$ 
if there exists a pseudo-isotopy
$(C,\langle \cdot \rangle,\{\frak m^t_{k,\beta}\},\{\frak c^t_{k,\beta}\})$
with given boundary value at $t=0,1$.
\par
Modulo $T^{E_0}$ and/or unital version is defined in a similar way.
\end{dfn}
\begin{lmm}\label{pisoequiv}
Pseudo-isotopy of $G$-gapped filtered cyclic $A_{\infty}$ algebras is an equivalence relation.
\par
Modulo $T^{E_0}$ version also holds.
\end{lmm}
\begin{proof}
Let $(C,\langle \cdot \rangle,\{\frak m^t_{k,\beta}\},\{\frak c^t_{k,\beta}\})$ 
be a pseudo-isotopy. First we show that we can 
modify them so that 
$\frak m^t_{k,\beta}$ is locally constant of $t$ and that
$\frak c^t_{k,\beta} = 0$, in a neighborhood of $t \in \partial [0,1]$,
as follows. 
Let $t = t(s)$ be a smooth map $[0,1] \to [0,1]$ which is constant in a neighborhood of $0,1$ and 
$t(1) = 1$, $t(0) =0$. We put: 
\begin{equation}\label{pisopullback}
\frak m^s_{k,\beta} = \frak m^{t(s)}_{k,\beta},
\qquad
\frak c^s_{k,\beta} = \frac{dt}{ds}(s) \cdot \frak c^{t(s)}_{k,\beta}.
\end{equation}
It is easy to see that $(C,\langle \cdot \rangle,\{\frak m^s_{k,\beta}\},\{\frak c^s_{k,\beta}\})$ is a required 
pseudo-isotopy.
\par
Now we can easily join two pseudo-isotopies satisfying the above additional condition. 
It implies that `pseudo-isotopic' is transitive.
The other properties are easier to check.
\end{proof}
\begin{thm}\label{pisotopyextention}
Let $E_0 < E_1$ and $(\overline C,\langle \cdot \rangle,\{\frak m^{i}_{k,\beta}\})$ $(i=0,1)$
be $G$-gapped cyclic filtered $A_{\infty}$ algebra modulo $T^{E_i}$.
Let $(C,\langle \cdot \rangle,\{\frak m^t_{k,\beta}\},\{\frak c^t_{k,\beta}\})$  be a pseudo-isotopy 
modulo $T^{E_0}$ between them. 
Then:
\smallskip
\begin{enumerate}
\item We can extend $(\overline C,\langle \cdot \rangle,\{\frak m^{0}_{k,\beta}\})$ to a 
$G$-gapped cyclic filtered $A_{\infty}$ algebra modulo $T^{E_1}$.
\item We can extend  $(C,\langle \cdot \rangle,\{\frak m^t_{k,\beta}\},\{\frak c^t_{k,\beta}\})$ 
to a pseudo-isotopy 
modulo $T^{E_1}$ between them. 
\end{enumerate}
\smallskip
\par Unital version also holds.
\end{thm}
The proof will be given in Section \ref{homisoto}.
\begin{thm}\label{pisotohomotopyequiv}
If  $(C,\langle \cdot \rangle,\{\frak m^t_{k,\beta}\},\{\frak c^t_{k,\beta}\})$ is a psudo-isotopy then 
there exists a filtered $A_{\infty}$ homomorphism from $(C,\langle \cdot \rangle,\{\frak m^{0}_{k,\beta}\})$ 
to $(C,\langle \cdot \rangle,\{\frak m^{1}_{k,\beta}\})$ which is cyclic and has an 
inverse.
\par
Modulo $T^{E}$ and/or unital version also holds. 
\end{thm}
The proof will be given in Section \ref{homisoto}.
\begin{rem}
It is possible to prove that gapped cyclic filtered $A_{\infty}$ homomorphism 
which is homotopy equivalence (as gapped filtered $A_{\infty}$ homomorphism) has a homotopy inverse
which is cyclic. See \cite{Ka} Theorem 5.17. 
The explicite construction of homotopy inverse given in \cite{AkJo08} proves it also.
\par
One reason why we built our story without using this theorem but used 
pseudo-isotopy more is that it seems that `pseudo-isotopic' is strictly stronger that `homotopy equivalent'.
(An invariant of a kind of `Reidemeister torsion' may distinguish them.)
So for future application (especially one in \cite{F1}) to keep track of pseudo-isotopy type rather than homotopy type seems 
essential.
\par
This does not seem to be the case when we do not include cyclic symmetry and inner product 
in the story.
\end{rem}
Let $(C,\langle \cdot \rangle,\{\frak m_{k,\beta}\})$ be a $G$-gapped cyclic filtered $A_{\infty}$ algebra.
We have $\frak m_{1,(0,0)}\circ \frak m_{1,(0,0)} = 0 : \overline C \to \overline C$. We put
$$
\overline H = \frac{\operatorname{Ker} \frak m_{1,(0,0)}}{\operatorname{Im} \frak m_{1,(0,0)}}.
$$
In \cite{FOOO080} Theorem 5.4.2', $G$-gapped filtered $A_{\infty}$ structure $\{\frak m_{k,\beta}\}$
on $\overline H$ is defined. Moreover $G$-gapped filtered $A_{\infty}$ homomorphism 
$\frak f :  H \to  C$, (which is a homotopy equivalence) is defined.
By (\ref{Ainfinityrelbeta}) the inner product $\langle \cdot \rangle$ on $\overline C$ 
induces one on $\overline H$, which we denote also by $\langle \cdot \rangle$.
\begin{thm}\label{cancyclic}
We assume either $\overline C$ is finite dimensional or 
is a de Rham complex. Then,
$(H,\langle \cdot \rangle,\{\frak m_{k,\beta}\})$ is cyclic. Moreover $\frak f :  H \to  C$ is cyclic.
\par
The modulo $T^{E_0}$ and/or unital version is also true.
\end{thm}
We will prove Theorem \ref{cancyclic} in Section \ref{canoalgsec}.
\begin{dfn}
We call $(H,\langle \cdot \rangle,\{\frak m_{k,\beta}\})$ the {\it canonical model} 
of cyclic filtered $A_{\infty}$ algebra $(C,\langle \cdot \rangle,\{\frak m_{k,\beta}\})$.
\end{dfn}
Weak homotoy equivalence between (cyclic) canonical  filtered $A_{\infty}$ algebras is 
an isomorphism. (\cite{FOOO080} Proposition 5.4.5.)
\begin{thm}\label{canopseudo}
If $(C,\langle \cdot \rangle,\{\frak m^{0}_{k,\beta}\})$ is pseudo-isotopic to 
$(C,\langle \cdot \rangle,\{\frak m^{1}_{k,\beta}\})$ then their canonical models are 
also pseudo-isotopic to each other.
\end{thm}
We prove Theorem \ref{canopseudo} in Section \ref{canoalgsec}.
\section{Pseudo-isotopy of cyclic filtered $A_{\infty}$ algebras}
\label{homisoto}
In this section we prove Theorems \ref{pisotopyextention} and \ref{pisotohomotopyequiv}.
We begin with the proof of Theorem \ref{pisotohomotopyequiv}.
We will construct the required isomorphism by taking appropriate sum over trees 
with some additional data, which we describe below.
\par
A {\it ribbon tree} is a tree $T$ together with isotopy type of an embedding 
$T \to \R^2$. (This is equivalent to fix cyclic order of the set of edges containing given vertex.)
A {\it rooted ribbon tree} is a pair $(T,v_0)$ of a ribbon tree $T$ and its vertex $v_0$ such that 
$v_0$ has exactly one edges.
Let $G \subset \R_{\ge 0} \times 2\Z$ be a discrete submonoid.
We consider triple $\Gamma = (T,v_0,\beta(\cdot))$ together with some other date which has the following properties:
\smallskip
\begin{enumerate}
\item $(T,v_0)$ is a rooted ribbon tree.
\item The set of vertices $C_0(T)$ is divided into disjoint union 
$C_0^{\text{\rm int}}(T) \cup C_0^{\text{\rm ext}}(T)$. $v_0 \in C_0^{\text{\rm ext}}(T)$.
Each of $v\in C_0^{\text{\rm ext}}(T)$ has exactly one edge.
\item $\beta(\cdot) : C_0^{\text{\rm int}}(T) \to G$ is a map.
\item If $\beta(v) = (0,0)$ then, $v$ has at least 3 edges.
\end{enumerate}
\begin{dfn}
We write $Gr(\beta,k)$ the set of all $(T,v_0,\beta(\cdot))$ as above such that
\smallskip
\begin{enumerate}
\item $\sum_{v\in C_0^{\text{\rm int}}(T)}\beta(v) = \beta.$
\item $\# C_0^{\text{\rm ext}}(T) = k+1$.
\end{enumerate}
\par\smallskip
We call an element of $C_0^{\text{\rm ext}}(T)$ an {\it exterior vertex} and
an element of $C_0^{\text{\rm int}}(T)$ an {\it interior vertex}.
An edge is said {\it exterior} if its contains an exterior edge. 
It is called {\it interior} otherwise. 
The set of exterior edges an interior edges are denoted by 
$C_1^{\text{\rm ext}}(T)$ and $C_1^{\text{\rm int}}(T)$, respectively.
\par
We call $v_0$ the {\it root} of $(T,v_0)$.
\end{dfn}
\begin{dfn}[See \cite{FOOO080} Definition 4.6.6.]
For a rooted ribbon tree $(T,v_0)$ we define a partial order $<$ on 
$C_0(T)$ as follows.
\par
$v<v'$ if all the paths joining $v$ with $v_0$ contains $v'$.
\end{dfn} 
\begin{dfn}[See \cite{FOOO080} Definition 7.1.53]
The {\it time allocation} of an element $(T,v_0,\beta(\cdot)) \in Gr(\beta,k)$
is a map $\tau : C_0^{\text{\rm int}}(T) \to [0,1]$ such that
if $v<v'$ then $\tau(v) \le \tau(v')$.
\par
We denote by $\frak M(T,v_0,\beta(\cdot);\tau_1,\tau_0)$ 
the set of all time allocations $\tau$ such that 
$\tau(v) \in [\tau_0,\tau_1]$ for all $v$. 
We write $\frak M(T,v_0,\beta(\cdot)) = \frak M(T,v_0,\beta(\cdot);1,0)$.
\par
We may regard 
\begin{equation}
\frak M(T,v_0,\beta(\cdot)) \subseteq [0,1]^{\# C_0^{\text{\rm int}}(T)}.
\end{equation}
\end{dfn}
For $(T,v_0,\beta(\cdot)) \in Gr(\beta,k)$ and $\tau \in \frak M(T,v_0,\beta(\cdot))$
we will associate an $\R$ linear map
$$
\frak c(T,v_0,\beta(\cdot),\tau)
: B_k(\overline C[1]) \to \overline C[1]
$$
of degree $-\mu(\beta)$ by induction on $\# C_0^{\text{\rm int}}(T)$ as follows.
\par
Suppose $\# C_0^{\text{\rm int}}(T)=0$. Then $T$ has only one edge and two (exterior) vertices.
So $\beta(\cdot)$ is void. We put
\begin{equation}
\frak c(T,v_0,\beta(\cdot),\tau) = \text{\rm identity},
\end{equation}
in this case.
\par
Suppose $\# C_0^{\text{\rm int}}(T)=1$. Let $v$ be the unique interior vertex. 
$\beta = \beta(v)$. $v$ has exactly $k+1$ edges. 
We put
\begin{equation}\label{93}
\frak c(T,v_0,\beta(\cdot),\tau)(x_1,\ldots,x_k)
= 
-\frak c^{\tau(v)}_{k,\beta}(x_1,\ldots,x_k).
\end{equation}
\par
Let $\# C_0^{\text{\rm int}}(T)>1$. We take the unique edge $e_0$ containing $v_0$. 
Let $v'_0$ be the edge of $e_0$ other than $v_0$. $v'_0$ is necessarily interior. 
We remove $v_0$, $e_0$ and $v'_0$ from $T$ and then
obtain $\ell$ components $T_1,\ldots,T_{\ell}$. 
Here $\ell+1$ is the number of edges of $v'_0$. 
We number them so that $v'_0, T_1, \ldots, T_{\ell}$ respects counter clockwise cyclic 
order induced by the canonical orietation of  $\R^2$.
We take the closures of $T_i$ and denote it by the same symbol, by an 
abuse of notation.
Together with the other data which is induced in an obvious way from one of $(T,v_0,\beta(\cdot),\tau)$
we obtain 
$(T_i,v'_0,\beta_i(\cdot),\tau_i)$ for $i=1,\ldots,\ell$.
We now put
\begin{equation}\label{94}
\aligned
&\frak c(T,v_0,\beta(\cdot),\tau) \\
&= -\frak c^{\tau(v'_0)}_{\ell,\beta(v'_0)} \circ 
\left(
\frak c(T_1,v'_0,\beta_1(\cdot),\tau_1) 
\otimes \ldots\otimes 
\frak c(T_{\ell},v'_0,\beta_{\ell}(\cdot),\tau_{\ell})
\right).
\endaligned
\end{equation}
Note the right hand side is already defined by induction hypothesis.
\par
Now we integrate on $\frak M(T,v_0,\beta(\cdot))$ and define:
\begin{equation}
\frak c(T,v_0,\beta(\cdot)) 
= \int_{\tau \in \frak M(T,v_0,\beta(\cdot))}\frak c(T,v_0,\beta(\cdot),\tau)
d\tau.
\end{equation}
Here we regard $ \frak M(T,v_0,\beta(\cdot)) \subset [0,1]^{\# C_0^{\text{\rm int}}(T)}$
and use standard measure $d\tau$ to integrate.
We define $\frak c(T,v_0,\beta(\cdot);\tau_1,\tau_0)$ in the same way by integrating on 
$\frak M(T,v_0,\beta(\cdot);\tau_1,\tau_0)$.
\begin{dfn}
\begin{equation}
\frak c(k,\beta) = \sum_{(T,v_0,\beta(\cdot)) \in Gr(k,\beta)} \frak c(T,v_0,\beta(\cdot)) .
\end{equation}
We define $\frak c(k,\beta;\tau_1,\tau_0)$ in a similar way.
\end{dfn}
\begin{prp}\label{Ainftyhomobyint}
$\{\frak c(k,\beta;\tau_1,\tau_0)\}$ defines a $G$ gapped filtered $A_{\infty}$ homomorphism
from  $(C,\{\frak m^{\tau_0}_{k,\beta}\})$ to $(C,\{\frak m^{\tau_1}_{k,\beta}\})$.
\end{prp}
\begin{proof}
Let $E(G) = \{0,E_1,\ldots,E_k,\ldots\}$ with $E_i < E_{i+1}$.
We are going to prove that $\{\frak c(k,\beta;\tau_1,\tau_0)\}$ is a filtered $A_{\infty}$
homomorphism modulo $E_j$ by induction on $j$.
\par
We remark that 
\begin{equation}
\frak c(k,(0,0))
= \begin{cases}
\text{identity}  &\text{if $k=1$}, \\
0 &\text{otherwise},
\end{cases} 
\end{equation}
by Definition \ref{pisotopydef}.5.
The case $j = 0+1=1$ follows immediately.
\par
We assume that $\{\frak c(k,\beta;\tau,\tau_0)\}$ is a filtered $A_{\infty}$
homomorphism modulo $E_j$.
Let $E(\beta) = E_j$. We will study the following two maps (\ref{cainfleft}), 
(\ref{cainfright}).
\begin{equation}\label{cainfleft}
\sum \frak m^{\tau}_{\ell,\beta_0} \circ
\left(
\frak c({k_1,\beta_1};\tau,\tau_0)\otimes\ldots
\otimes \frak c({k_{\ell},\beta_{\ell}};\tau,\tau_0)
\right)
\end{equation}
where the sum is taken over all $\ell$, $k_i$, $\beta_i$ with
$\beta_0 + \beta_1 + \ldots + \beta_{\ell} = \beta$, 
$k_1+\ldots+k_{\ell} = k$.
\begin{equation}\label{cainfright}
x_1 \otimes \ldots \otimes x_k \mapsto
\sum
(-1)^*\frak c(\beta_1,k_1;\tau,\tau_0)(x_1,\ldots,\frak m^{\tau_0}_{k_2,\beta_2}(x_i,\ldots),\ldots,x_k)
\end{equation}
where the sum is taken over $k_1$, $k_2$, $\beta_1$, $\beta_2$, $i$ with 
$k_1+k_2=k+1$, $\beta_1+\beta_2 = \beta$ and $i = 1,\ldots,k-k_2+1$ and 
$* = \deg' x_1 + \ldots + \deg'  x_{i-1}$.
\par
We denote (\ref{cainfleft}) by $\frak P(k,\beta;\tau,\tau_0)$ 
and (\ref{cainfright}) by $\frak Q(k,\beta;\tau,\tau_0)$.
To prove Proposition \ref{Ainftyhomobyint} it suffcies to show 
$\frak P(k,\beta;\tau,\tau_0)=\frak Q(k,\beta;\tau,\tau_0)$.
\par
We calculate
\begin{equation}\label{derivativefraP}
\aligned
&\frac{d}{d\tau} \frak P(k,\beta;\tau,\tau_0)
\\
&=
\sum \left(\frac{d}{d\tau}\frak m^{\tau}_{\ell,\beta_0}\right) \circ
\left(
\frak c({k_1,\beta_1};\tau,\tau_0)\otimes\ldots
\otimes \frak c({k_{\ell},\beta_{\ell}};\tau,\tau_0)
\right)
\\
&\qquad\quad+ 
\frak m^{\tau}_{\ell,\beta_0} \circ
\frac{d}{d\tau}\left(
\frak c({k_1,\beta_1};\tau,\tau_0)\otimes\ldots
\otimes \frak c({k_{\ell},\beta_{\ell}};\tau,\tau_0)
\right)
\endaligned\end{equation}
By using (\ref{isotopymaineq}) the first term of (\ref{derivativefraP}) becomes
the sum of the following two formulas:
\smallskip
\begin{enumerate}
\item The sum of the composition of
\begin{equation}\label{frakctensor}
\frak c({k_1,\beta_1};\tau,\tau_0)\otimes\ldots
\otimes \frak c({k_{\ell},\beta_{\ell}};\tau,\tau_0)
\end{equation}
and
\begin{equation}
\aligned
x_1 \otimes \ldots \otimes x_{\ell} \mapsto
(-1)^{*}\frak c^{\tau}_{\ell_1,\beta_0}(\ldots, \frak m_{\ell_2,\beta'_0}^{\tau}(x_i,\ldots),\ldots). 
\endaligned
\end{equation}
Here the sum is taken over all $\beta_0$, $\beta'_0$, $\beta_1,\ldots,\beta_{\ell}$, $\ell_1$, $\ell_2$, 
$k_1,\ldots,k_{\ell}$ such that
$\beta = \beta_0+\beta'_0 +\beta_1+\ldots +\beta_{\ell}$, $\ell_1 + \ell_2 = \ell+1$, 
$k_1+\ldots+k_{\ell} = k$. The sign is
$* = \deg' x_1+\ldots+\deg' x_{i-1}$.
\item
The composition of (\ref{frakctensor}) and
\begin{equation}\label{913}
\aligned
x_1 \otimes \ldots \otimes x_{\ell} \mapsto
-\frak m^{\tau}_{\ell_1,\beta_0}(\ldots, \frak c_{\ell_2,\beta'_0}^{\tau}(x_i,\ldots),\ldots). 
\endaligned
\end{equation}
\end{enumerate}
We remark that the minus sign in (\ref{913}) is induced by the 
minus sign in the 3rd line of (\ref{isotopymaineq}).
\smallskip\par
Using induction hypothesis, we can show that 1 above is equal to (\ref{cainfright}).
(Note the minus sign in (\ref{93}), (\ref{94}) is essential here.)
\par
On the other hand, by definition we can show that 2 above cancels with the second term of 
(\ref{derivativefraP}). 
(We again use the minus sign in (\ref{93}), (\ref{94}) here.)
\par
The proof of Proposition \ref{Ainftyhomobyint} is now complete.
\end{proof}
\begin{prp}\label{gtrascyclic}
The filtered $A_{\infty}$ homomorphism $\{\frak c(k,\beta)\}$ is cyclic.
\end{prp}
\begin{proof}
Let $(k,\beta) \ne (2,(0,0))$. We will prove:
\begin{equation}\label{innnerhomocond}
\sum_{k_1+k_2=k \atop \beta_1+\beta_2=\beta}
\left\langle
\frak c(k_1,\beta_1)(x_1,\ldots,x_{k_1}),
\frak c(k_2,\beta_2)(x_{k_1+1},\ldots,x_{k})
\right\rangle = 0.
\end{equation}
A term of (\ref{innnerhomocond}) is written as 
\begin{equation}\label{innerprodint}
\int_{\tau \in \frak M(\Gamma_1)}\int_{\tau' \in \frak M(\Gamma_2)}
\left\langle
\frak c(\Gamma_1;\tau)(x_1,\ldots,x_{k_1}),
\frak c(\Gamma_2;\tau')(x_{k_1+1},\ldots,x_{k})
\right\rangle d\tau d\tau'.
\end{equation}
Here $\Gamma_i =(T_i,v^i_0,\beta_i(\ldots)) \in Gr(k_i,\beta_i)$ with 
$k_1+k_2=k$, $\beta_1+\beta_2=\beta$.
\par
We put
$$
\tau_{\text{max}} = \max\{ \tau(v) \mid v \in C_0^{\text{int}}(T_1)\}
= \tau(v_0^{1\prime}).
$$
Here $v_0^{1\prime}$ is the unique interior vertex which is joined with $v_0^1$. 
We define 
$$
\tau'_{\text{max}} = \max\{ \tau'(v) \mid v \in C_0^{\text{int}}(T_2)\}
= \tau'(v_0^{2\prime}),
$$
in the same way. We divide the domain of integration (\ref{innerprodint}) 
into two:
\smallskip
\begin{enumerate}
\item $\tau_{\text{max}} \ge \tau'_{\text{max}}$.
\item $\tau_{\text{max}} \le \tau'_{\text{max}}$.
\end{enumerate}
\smallskip\par
Integration on the domain 1, is the sum of the terms:
\begin{equation}
\aligned
-\int_{0}^1 
\left\langle
(\frak c^t_{\ell,\beta_{(0)}}\circ
(\frak c(\Gamma(1);t,\tau_0)\otimes \ldots,\right.\otimes &\frak c(\Gamma(\ell);t,\tau_0)))
(x_1,\ldots,x_{k_1}),\\
&\frak c(\Gamma(0);t,\tau_0)(x_{k_1+1},\ldots,x_{k})
\big\rangle dt.
\endaligned
\end{equation}
Here $\Gamma(i) = (T_i,v^i_0,\beta^i(\cdot)) \in Gr(k(i),\beta^i)$
such that $\sum_{i=1}^{\ell}k(i) = k_1$, $k(0) = k_2$, $k_1+k_2=k$, 
$\sum_{i=0}^{\ell} \beta^i + \beta^0 + \beta_{(0)} = \beta$.
\par 
In a similar way, Integration on the domain 2, is the sum of the terms:
\begin{equation}
\aligned
-\int_{0}^1 
\big\langle
&\frak c(\Gamma(0);t,\tau_0)
(x_1,\ldots,x_{k_2}),\\
&(\frak c^t_{\ell,\beta_{(0)}}\circ
(\frak c(\Gamma(1);t,\tau_0)\otimes \ldots,\otimes \frak c(\Gamma(\ell);t,\tau_0)))(x_{k_2+1},\ldots,x_{k})
\big\rangle dt.
\endaligned
\end{equation}
Therefore (\ref{innnerhomocond}) follows from the following:
\begin{lmm}
\begin{equation}
\langle 
\frak c_{\ell,\beta}^t(x_1,\ldots,x_{\ell}),x_0
\rangle
+ 
\langle 
x_1,\frak c_{\ell,\beta}^t(x_2,\ldots,x_{\ell},x_0)
\rangle
= 0.
\end{equation}
\end{lmm}
\begin{proof}
$$
\aligned
&\langle 
\frak c_{\ell,\beta}^t(x_1,\ldots,x_{\ell}),x_0
\rangle \\
&= (-1)^{(\deg'x_1)\left(\sum_{i\ne 1} \deg'x_i\right)}
\langle 
\frak c_{\ell,\beta}^t(x_2,\ldots,x_{\ell},x_0),x_1
\rangle\\
&= - \langle 
x_1,\frak c_{\ell,\beta}^t(x_2,\ldots,x_{\ell},x_0)
\rangle.
\endaligned
$$
Here we use cyclic symmetry of $\frak c_{\ell,\beta}^t$ in the first 
equality and (\ref{gradedansym})
in the second equality.
\end{proof}
The proof of Proposition \ref{gtrascyclic} and  Theorem \ref{pisotopyextention} are complete.
\end{proof}
\begin{proof}[Proof of Theorem \ref{pisotohomotopyequiv}]
We may assume that $E(G) \cap [E_0,E_1] = \{E_0,E_1\}$.
(In the general case we can divide the interval $[E_0,E_1]$ into 
the pieces so that the above assumption holds.)
\par
We use the modulo $T^{E_0}$ version of Theorem  \ref{pisotopyextention} we proved above
and obtain a cyclic filtered $A_{\infty}$ homomorphism $\{\frak c_{k,\beta}\}$ modulo $T^{E_0}$.
\par
Let $E(\beta) = E_0$. We put $\frak c_{k,\beta}^t = 0$.
We then define
$\frak m^t_{k,\beta}$ by solving  (\ref{isotopymaineq}). Namely we put
\begin{equation}\label{mtintdef}
\aligned
&\frak m^{\tau}_{k,\beta}(x_1,\ldots,x_k) \\
=&\frak m^1_{k,\beta}(x_1,\ldots,x_k) \\
&- \sum_{k_1+k_2=k \atop \beta_1+\beta_2=\beta }\sum_{i=1}^{k-k_2+1}
(-1)^{*_i}\int_{\tau}^1\frak c^t_{k_1,\beta_1}(\ldots, \frak m_{k_2,\beta_2}^t(x_i,\ldots),\ldots) \,dt\\
&+\sum_{k_1+k_2=k \atop \beta_1+\beta_2=\beta}\sum_{i=1}^{k-k_2+1}
\int_{\tau}^1\frak m^t_{k_1,\beta_1}(\ldots, \frak c_{k_2,\beta_2}^t(x_i,\ldots),\ldots)
\, dt.
\endaligned
\end{equation}
Here $*_i = \deg' x_1 + \ldots + \deg'x_{i-1}$.
\par
We remark that if $\frak c^t_{k,\beta} \ne 0$ then $E(\beta) > 0$.
Therefore the right hand side of (\ref{mtintdef}) is already defined by
induction hypothesis.
\par
1,3,4,5 of Definition \ref{pisotopydef} is obvious.
\begin{lmm}\label{checkAiniftybydiff}
$\frak m^{\tau}_{k,\beta}$ in $(\ref{mtintdef})$ satisfies the filtered 
$A_{\infty}$ relation $(\ref{Ainfinityrelbeta})$.
\end{lmm}
\begin{proof}
We remark that $\frak m^{1}_{k,\beta}$ satisfies $(\ref{Ainfinityrelbeta})$ by assumption.
We prove $(\ref{Ainfinityrelbeta})$ by induction on $E(\beta)$. 
Since $\frak m^{\tau}_{k,(0,0)}$ is independent of $\tau$, 
$(\ref{Ainfinityrelbeta})$ holds for $E(\beta) = 0$.
We assume that it is satisfied for $\beta'$ with $E(\beta') < E(\beta)$ and 
consider the case of $\beta$.
We calculate
$$
\aligned
&\frac{d}{dt} \left(
\sum (-1)^{*_i} \frak m_{k_1,\beta_1}^t(\ldots, \frak m_{k_2,\beta_2}^t(x_i,\ldots),
\ldots)
\right)\\
= &\sum(-1)^{*^1_{i,j}}
\frak c_{k_1,\beta_1}^t(\ldots, \frak m_{k_2,\beta_2}^t(x_j,\ldots), 
\ldots,\frak m_{k_3,\beta_3}^t(x_i,\ldots),
\ldots) \\
&+\sum(-1)^{*^2_{i,j}}
\frak c_{k_1,\beta_1}^t(\ldots, \frak m_{k_2,\beta_2}^t(x_i,\ldots), 
\ldots,\frak m_{k_3,\beta_3}^t(x_j,\ldots),
\ldots) \\
&+\sum(-1)^{*^3_{i,j}}
\frak c_{k_1,\beta_1}^t(\ldots, \frak m_{k_2,\beta_2}^t(x_j,\ldots\frak m_{k_3,\beta_3}^t(x_i,\ldots)
,\ldots),
\ldots)\\
&+\sum(-1)^{*^4_{i,j}}
\frak m_{k_1,\beta_1}^t(\ldots, \frak c_{k_2,\beta_2}^t(x_j,\ldots), 
\ldots,\frak m_{k_3,\beta_3}^t(x_i,\ldots),
\ldots)\\
&+\sum(-1)^{*^5_{i,j}}
\frak m_{k_1,\beta_1}^t(\ldots, \frak m_{k_2,\beta_2}^t(x_i,\ldots), 
\ldots,\frak c_{k_3,\beta_3}^t(x_j,\ldots),
\ldots)\\
&+\sum(-1)^{*^6_{i,j}}
\frak m_{k_1,\beta_1}^t(\ldots, \frak c_{k_2,\beta_2}^t(x_j,\ldots\frak m_{k_3,\beta_3}^t(x_i,\ldots)
,\ldots),
\ldots)\\
&+\sum(-1)^{*^7_{i,j}}
\frak m_{k_1,\beta_1}^t(\ldots, \frak c_{k_2,\beta_2}^t(x_i,\ldots\frak m_{k_3,\beta_3}^t(x_j,\ldots)
,\ldots),
\ldots) \\
&+\sum(-1)^{*^8_{i,j}}
\frak m_{k_1,\beta_1}^t(\ldots, \frak m_{k_2,\beta_2}^t(x_i,\ldots\frak c_{k_3,\beta_3}^t(x_j,\ldots)
,\ldots),
\ldots).
\endaligned
$$
Here the first 6 terms are obtained by differentiating $\frak m_{k_1,\beta_1}^t$ and the 
last 2 terms are obtained by differentiating $\frak m_{k_2,\beta_2}^t$.
The signs are given by
$$\aligned
*^1_{i,j} &=  \deg'x_1+\ldots+\deg'x_{i-1}+ \deg'x_1+\ldots+\deg'x_{j-1}\\
*^2_{i,j} &=  \deg'x_1+\ldots+\deg'x_{i-1}+ \deg'x_1+\ldots+\deg'x_{j-1}+1\\
*^3_{i,j} &=  \deg'x_j+\ldots+\deg'x_{i-1}\\
*^4_{i,j} &=  \deg'x_1+\ldots+\deg'x_{i-1}+1\\
*^5_{i,j} &=  \deg'x_1+\ldots+\deg'x_{i-1}+1\\
*^6_{i,j} &=  \deg'x_1+\ldots+\deg'x_{i-1}+1\\
*^7_{i,j} &=  \deg'x_1+\ldots+\deg'x_{i-1}+ \deg'x_1+\ldots+\deg'x_{j-1}\\
*^8_{i,j} &=  \deg'x_1+\ldots+\deg'x_{i-1}+1.
\endaligned$$
Now the 1st and 2nd terms cancels. 
The 3rd term is $0$ by induction hypothesis. ($A_{\infty}$  relation for $\frak m$.
We remark that $\frak c_{k,\beta} \ne 0$ only if $E(\beta) >0$.)
The sum of 4th, 5th and 8th terms are $0$ by induction hypothesis also.
6th and 7th terms cancel.
The proof of Lemma \ref{checkAiniftybydiff} is now complete. 
\end{proof}
\begin{lmm}\label{mcancyclic}
$\frak m^{\tau}_{k,\beta}$ is cyclically symmetric.
\end{lmm}
\begin{proof}
We consider the following formulas:
\begin{equation}\label{mccylc}
\sum_{k_1+k_2=k \atop \beta_1+\beta_2=\beta }\sum_{i=1}^{k-k_2+1}
(-1)^{*_i+1}
\left\langle
\frak c^t_{k_1,\beta_1}(\ldots, \frak m_{k_2,\beta_2}^t(x_i,\ldots),\ldots),
x_0
\right\rangle,
\end{equation}
where $*_i = \deg'\rho_1+\ldots+\deg'\rho_{i-1}$ and
\begin{equation}\label{cmcylc}
\sum_{k_1+k_2=k \atop \beta_1+\beta_2=\beta}\sum_{i=1}^{k-k_2+1}
\left\langle
\frak m^t_{k_1,\beta_1}(\ldots, \frak c_{k_2,\beta_2}^t(x_i,\ldots),\ldots),
x_0
\right\rangle.
\end{equation}
We denote (\ref{mccylc}) as $\frak P(x_1,\ldots,x_k,x_0)$ and 
(\ref{cmcylc}) as $\frak Q(x_1,\ldots,x_k,x_0)$.
\par
We will prove
\begin{equation}\label{PQcyclic}
\aligned
&\frak P(x_0,x_1,\ldots,x_k)+ \frak Q(x_0,x_1,\ldots,x_k)=
\\
& (-1)^{(\deg'x_0)(\deg'x_1+\ldots+\deg'x_k)}
(\frak P(x_1,\ldots,x_k,x_0)+ \frak Q(x_1,\ldots,x_k,x_0)).
\endaligned
\end{equation}
We  have
\begin{equation}\label{923}
\aligned
&\frak P(x_0,x_1,\ldots,x_k) \\
=& -\sum 
\left\langle
\frak c^t_{k_1,\beta_1}(\frak m_{k_2,\beta_2}^t(x_0,\ldots),\ldots),
x_k
\right\rangle \\
&+ \sum (-1)^{\deg'x_0+*_i+1}
\left\langle
\frak c^t_{k_1,\beta_1}(x_0,\ldots,\frak m_{k_2,\beta_2}^t(x_i,\ldots),\ldots),
x_k
\right\rangle
\endaligned\end{equation}
and
\begin{equation}\label{924}
\aligned
&\frak Q(x_0,x_1,\ldots,x_k) \\
=& \sum 
\left\langle
\frak m^t_{k_1,\beta_1}(\frak c_{k_2,\beta_2}^t(x_0,\ldots),\ldots),
x_k
\right\rangle \\
&+ \sum 
\left\langle
\frak m^t_{k_1,\beta_1}(x_0,\ldots,\frak c_{k_2,\beta_2}^t(\ldots),\ldots),
x_k
\right\rangle.
\endaligned\end{equation}
Moreover
\begin{equation}\label{925}
\aligned
&(-1)^{(\deg'x_0)(\deg'x_1+\ldots+\deg'x_k)}
\frak P(x_1,\ldots,x_k,x_0) \\
=& \sum (-1)^{*+ 1+ *_i}
\left\langle
\frak c^t_{k_1,\beta_1}(\ldots,\frak m_{k_2,\beta_2}^t(x_i,\ldots),\ldots,x_k),
x_0
\right\rangle \\
&+ \sum (-1)^{*+*_i+1}
\left\langle
\frak c^t_{k_1,\beta_1}(\ldots,\frak m_{k_2,\beta_2}^t(x_i,\ldots,x_k)),
x_0
\right\rangle
\endaligned\end{equation}
with $*=(\deg'x_0)(\deg'x_1+\ldots+\deg'x_k)$ and
\begin{equation}\label{926}
\aligned
&(-1)^{(\deg'x_0)(\deg'x_1+\ldots+\deg'x_k)}\frak Q(x_1,\ldots,x_k,x_0) \\
=& \sum 
(-1)^{*}\left\langle
\frak m^t_{k_1,\beta_1}(\ldots,\frak c_{k_2,\beta_2}^t(x_i,\ldots),\ldots,x_k),
x_0
\right\rangle \\
&+ \sum 
(-1)^{*}\left\langle
\frak m^t_{k_1,\beta_1}(\ldots,\frak c_{k_2,\beta_2}^t(x_i,\ldots,x_k)),
x_0
\right\rangle.
\endaligned\end{equation}
The 2nd term of (\ref{923}) coincides with 1st term of (\ref{925}) by 
the cyclic symmetry of $\frak c^t$.
The 2nd term of (\ref{924}) coincides with 1st term of (\ref{926})  by 
the cyclic symmetry of $\frak m^t$.
\par
The 1st term of (\ref{923}) coincides with the 2nd term of (\ref{926}).
In fact
$$\aligned
&-
\left\langle
\frak c^t_{k_1,\beta_1}(\frak m_{k_2,\beta_2}^t(x_0,\ldots,x_{i-1}),\ldots),
x_k \right\rangle\\
&=(-1)^{1+(*_i+1)(\deg'x_i+\ldots+\deg'x_k)}
\left\langle \frak c^t_{k_1,\beta_1}(x_i,\ldots,x_k),
\frak m_{k_2,\beta_2}^t(x_0,\ldots,x_{i-1})
\right\rangle\\
&= 
\left\langle 
\frak m_{k_2,\beta_2}^t(x_0,\ldots,x_{i-1}),\frak c^t_{k_1,\beta_1}(x_i,\ldots,x_k)
\right\rangle,
\endaligned$$
which is equal to the 2nd term of (\ref{926}).
\par
In the same way the 1st term of (\ref{924}) coincides with the 
2nd term of (\ref{925}).
We thus proved (\ref{PQcyclic}).
\end{proof}
We thus constructed cyclic filtered $A_{\infty}$ homomorphism 
$\{\frak c_{k,\beta}\}$. We remark that $\frak c_{1,(0,0)}$ is identity and 
$\frak c_{k,(0,0)} = 0$ for $k\ne 1$ by definition. 
We can use this fact to show that $\{\frak c_{k,\beta}\}$ has an 
inverse, by induction on energy filtration.
The proof of Theorem \ref{pisotohomotopyequiv} is complete.
\end{proof}
\section{Canonical model of cyclic filtered $A_{\infty}$ algebra}
\label{canoalgsec}
In this section we prove Theorems \ref{cancyclic} and \ref{canopseudo}.
We first review the construction of the filtered $A_{\infty}$ structure 
$\frak m^{\text{\rm can}}_{k,\beta}$ on $H$ and 
filtered $A_{\infty}$ homomorphism $\{\frak f_{k,\beta}\} : H \to C$, 
from \cite{FOOO080} Subsection 5.4.4.
\par
We consider the chain complex $(\overline C,\frak m_{1,(0,0)})$ together 
with its inner product.
We take $\R$ linear subspace $\overline H \subset \overline C$ such that 
$\frak m_{1,(0,0)} = 0$ on $\overline H$ and $\overline H$ is
identified with the $\frak m_{1,(0,0)}$ cohomology by an obvious map.
In the case $\overline C$ is the de Rham complex of $L$ we take a 
Riemannian metric on $L$ and $\overline H$ be the space of harmonic forms.
\begin{lmm}[Compare \cite{Ka} Section 5.1.]\label{Gdual}
We assume either $\overline C$ is finite dimensional or is a de Rham complex.
There exists a map $\Pi : \overline C \to \overline C$ of degree $0$ and
$G : \overline C \to \overline C$ of degree $+1$ with the following properties.
\smallskip
\begin{enumerate}
\item $\Pi \circ \Pi = \Pi$.  The image of $\Pi$ is a $\overline H$.
\item 
\begin{equation}\label{GPIhomotopy}
\text{\rm identity} - \Pi 
= 
-\left( \frak m_{1,(0,0)}\circ G + G \circ \frak m_{1,(0,0)} \right).
\end{equation}
\item $G\circ G = 0$.
\item $\langle \Pi x,y\rangle = \langle x,\Pi y\rangle$.
\item
\begin{equation}\label{Gdualformula}
\langle x,Gy\rangle = (-1)^{\deg x\deg y}\langle y,Gx\rangle
\end{equation}
\end{enumerate}
\end{lmm}
\begin{proof}
We first assume that $\overline C$ is finite dimensional.
We remark 
\begin{equation}\label{m1dual}
\langle \frak m_{1,(0,0)} x,y\rangle = (-1)^{\deg'x\deg'y +1}\langle \frak m_{1,(0,0)} y,x\rangle.
\end{equation}
We put $\overline B = \text{\rm Im}\,\,\frak m_{1,(0,0)}$.
(\ref{m1dual}) implies $\langle \overline B,\overline H\rangle = 0$, 
$\langle \overline B,\overline B\rangle = 0$.
We put 
$$
\overline C' = \{ x \in \overline C \mid \langle x,\overline H\rangle = 0\}.
$$
Since $\langle \cdot \rangle$ is nondegenerate on $\overline H$ it follows 
that it is nondegenerate also on $\overline C'$.
By an easy linear algebra we can find $\overline D \subset \overline C'$ such that
$\overline{B}\oplus \overline D = \overline C'$ and
$\langle \overline D,\overline D\rangle = 0$.
We thus have decomposition 
\begin{equation}\label{HKdecom}
\overline C = \overline{B}\oplus \overline D \oplus \overline H.
\end{equation}
We use this decomposition to define projections
$\Pi_B$, $\Pi_D$, $\Pi$ to $\overline{B}$, $\overline D$,  $\overline H$ respectively.
We have
\begin{equation}\label{HKorghogo}
\langle \Pi x,y \rangle = \langle  x,\Pi y \rangle, 
\quad \langle \Pi_B x,y \rangle = \langle x,\Pi_D y \rangle,
\quad \langle \Pi_D x,y \rangle = \langle x,\Pi_B y \rangle.
\end{equation}
Thus we have 1 and 4.
\par
By construction the restriction of $\frak m_{1,(0,0)}$ to $\overline D$ 
induces an isomorphism $: \overline D \to \overline B$.
Let $\frak n$ be its inverse. We put
\begin{equation}
G = -\frak n \circ \Pi_B = -\Pi_D \circ \frak n \circ \Pi_B.
\end{equation}
It is easy to check 2 and 3. We will prove 5.
We first show:
\begin{equation}\label{dualn}
\langle x,\frak n(y) \rangle = (-1)^{\deg x\deg y}\langle y,\frak n(x) \rangle.
\end{equation}
To prove (\ref{dualn}) we may assume $x,y \in \overline B$. 
We put $x = \frak m_{1,(0,0)}a$, $y =\frak m_{1,(0,0)}b$.
Then
$$
\langle x,\frak n(y) \rangle = 
\langle \frak m_{1,(0,0)}a,b \rangle
= (-1)^{\deg' a\deg' b}\langle \frak m_{1,(0,0)}b,a \rangle
= (-1)^{\deg x\deg y}\langle y,\frak n(x) \rangle
$$
as required.
Now we have
$$\aligned
\langle x,G(y) \rangle &= - \langle x,\Pi_D \circ\frak n \circ \Pi_B(y)\rangle 
= - \langle \Pi_B(x),\frak n(\Pi_B(y))\rangle \\
&=-(-1)^{\deg x\deg y}\langle \Pi_B(y),\frak n(\Pi_B(x))\rangle 
=(-1)^{\deg x\deg y}\langle y,G(x)\rangle.
\endaligned
$$
The proof of Lemma \ref{Gdual} is complete in the case $\overline C$ is finite 
dimensional.
In the case of de Rham complex, we take $\delta$ the $L^2$ conjugate to 
$\frak m_{1,(0,0)}$ and let $\overline D$ be the image of it.
(\ref{HKdecom}) is nothing but the Hodge-Kodaira decomposition.
(\ref{HKorghogo}) is well-known. The rest of the proof is the same.
\end{proof}
Let $Gr(k,\beta)$ be as in Section \ref{homisoto}.
For each $\Gamma =(T,v_0,\beta(\cdot)) \in Gr(k,\beta)$ 
we will define
\begin{equation}
\frak f_{\Gamma} ; B_k(\overline H[1]) \to \overline C[1]
\end{equation}
of degree $-\mu(\beta)$ and 
\begin{equation}
\frak m_{\Gamma} : B_k(\overline H[1]) \to \overline H[1]
\end{equation}
of degree $1-\mu(\beta)$ by induction on $\#C_0^{\text{\rm int}}(T)$.
\par
Suppose $\#C_0^{\text{\rm int}}(T)=0$. Then $k=1$. 
We put $\frak f_{\Gamma}=$identity and $\frak m_{\Gamma}=\frak m_{1,(0,0)}$.
\par
Suppose $\#C_0^{\text{\rm int}}(T)=1$. Then $v \in C_0^{\text{\rm int}}(T)$ has 
$k+1$ edges. We put
$$
\left\{
\aligned
\frak f_{\Gamma}(x_1,\ldots,x_k) = G(\frak m_{k,\beta(v)}(x_1,\ldots,x_k)), \\
\frak m_{\Gamma}(x_1,\ldots,x_k) = \Pi(\frak m_{k,\beta(v)}(x_1,\ldots,x_k)).
\endaligned
\right.
$$
\par
We next assume $\#C_0^{\text{\rm int}}(T)\ge 2$. Let $e$ be the edge containing 
$v_0$ and $v$ be another vertex of $e$.  $v$ is necessarily interior. We remove 
$v_0,v,e$ from $T$ and obtain $T_1,\ldots,T_{\ell}$, where $v$ has 
$\ell+1$ edges. We number $T_i$ so that 
 $e,T_1,\ldots,T_{\ell}$ respects the counter clockwise cyclic order induced by 
 the orientation of $\R^2$.
The tree $T_i$, together with the data induced from $\Gamma$ in an obvious way,
 determines $\Gamma_i \in Gr(k_i,\beta_i)$. We have
 $\beta = \beta(v) + \sum \beta(i)$, $k = \sum k_i$. 
 We put
 $$
\left\{
\aligned
\frak f_{\Gamma} = G\circ\frak m_{k,\beta(v)}\circ (\frak f_{\Gamma_1} \otimes \ldots 
\otimes
\frak f_{\Gamma_{\ell}}), \\
\frak m_{\Gamma}= \Pi\circ\frak m_{k,\beta(v)}\circ (\frak f_{\Gamma_1} \otimes \ldots 
\otimes
\frak f_{\Gamma_{\ell}}).
\endaligned
\right.
$$
We now define
\begin{equation}
\frak f_{k,\beta} = \sum_{\Gamma \in Gr(k,\beta)} \frak f_{\Gamma},
\qquad 
\frak m^{\text{\rm can}}_{k,\beta} = \sum_{\Gamma \in Gr(k,\beta)} \frak m_{\Gamma}.
\end{equation}
\begin{lmm}\label{canhomoOK}
$\{\frak m^{\text{\rm can}}_{k,\beta}\}$ defines a structure of 
filtered $A_{\infty}$ algebra on $H$. 
\par
$\{\frak f^{\text{\rm can}}_{k,\beta}\}$ defines a 
filtered $A_{\infty}$ homomorphism $: H \to C$. 
\par
Unital and/or mod $T^{E_0}$ versions also hold.
\end{lmm}
We omit the proof and refer \cite{FOOO080} Subsection 5.4.4.
\par
We next prove the cyclicity of $\frak m^{\text{\rm can}}_{k,\beta}$.
We need some  more notations.
Let $\Gamma = (T,v_0,\beta(\cdot)) \in Gr(k,\beta)$.
A {\it flag} of $\Gamma$ is a pair $(v,e)$, where $v$ is an interior vertex of $T$ and 
$e$ is an edge containing $v$. For each $(\Gamma,v,e)$ we will define
\begin{equation}
\frak m(\Gamma,v,e) : B_{k+1}(\overline C[1]) \to \R
\end{equation}
as follows. We remove $v$ from $T$. Let $T_0,T_1,\ldots,T_{\ell}$ be the 
component of the complement. We assume $e \in T_0$ and 
$T_0,T_1,\ldots,T_{\ell}$ respects the counter clockwise cyclic order 
induced by the standard orientation of $\R^2$.
Together with data induced from $\Gamma$ the ribbon tree $T_i$ determine 
$\Gamma_i =(T_i,v,\beta_i(\cdot)) \in Gr(k_i,\beta_i)$ such that 
$\beta(v) + \sum \beta_i  = \beta$ and $\sum k_i = k$.
(We remark that the root of $\Gamma_i$ is always $v$ by convention.)
\par
We enumerate the exterior vertices of $\Gamma$ as $v_0,v_1,\ldots,v_k$ so that 
it respects the counter clockwise cyclic order.
We take $j_i$ such that 
$v_{j_i},\ldots,v_{j_i+k_i-1}$ are vertices of $T_i$. 
(In case $j_i+k_i-1 > k$ we identify $v_{j_i+k_i-1}$ with 
$v_{j_i+k_i-1-k}$.)
\begin{dfn}
\begin{equation}
\aligned
&\frak m(\Gamma,v,e)(x_1,\ldots,x_k,x_0) \\
&= (-1)^* 
\langle \frak m_{\ell,\beta(v)}(
\frak f_{\Gamma_1}(x_{j_1},\ldots),\ldots, \frak f_{\Gamma_{\ell}}(x_{j_{\ell}},\ldots)), 
\frak f_{\Gamma_0}(x_{j_0},\ldots,x_{j_1-1})
\rangle.
\endaligned
\end{equation}
Here 
$$
* = (\deg' x_{j_1} +\deg' x_{j_1+1} + \ldots + \deg' x_{0})
(\deg' x_{1} + \ldots + \deg' x_{j_1-1}).
$$
\end{dfn}
\begin{prp}\label{independofpoints}
$\frak m(\Gamma,v,e)(x_1,\ldots,x_k,x_0)$ is independent of 
the flag $(v,e)$ but depends only on $\Gamma$, 
$x_1,\ldots,x_k,x_0$.
\end{prp}
\begin{proof}
Independence of $e$ is a consequence of the cyclic symmetry of $\frak m_{k,\beta}$.
We will prove independence of $v$. 
Let $v$ and $v'$ be interior vertices. We assume that there exists an edge $e$ 
joining $v$ and $v'$.
\par
We first consider the flag $(v,e)$. We then obtain $\Gamma_0, 
\ldots, \Gamma_{\ell}$ as above. 
We next consider the flag $(v',e)$. We then obtain $\Gamma'_0, 
\ldots, \Gamma'_{\ell'}$ as above. 
It is easy to see the following:
\begin{equation}\label{relationamongtrees}
\left\{
\aligned
& \Gamma_1 \cup \ldots \cup  \Gamma_{\ell} \cup e \cup 
\Gamma'_1 \cup \ldots \cup  \Gamma'_{\ell'} = \Gamma, \\
& e \cup \Gamma'_1 \cup \ldots \cup  \Gamma'_{\ell'} = \Gamma_0 \\
& e \cup \Gamma_1 \cup \ldots \cup  \Gamma_{\ell} = \Gamma'_0.
\endaligned
\right.
\end{equation}
Let $v_{j_i},\ldots,v_{j_i+k_i-1}$ be the vertices of $\Gamma_i$ and 
$v_{j'_i},\ldots,v_{j'_i+k'_i-1}$ the vertices of $\Gamma'_i$.
We put
$$
y_i = \frak f_{\Gamma_i}(x_{j_i},\ldots,x_{j_i+k_i-1}), \qquad
z_i = \frak f_{\Gamma'_i}(x_{j'_i},\ldots,x_{j'_i+k'_i-1}).
$$
Now by definitions and (\ref{relationamongtrees}) we have
\begin{equation}\label{mgammav}
\aligned
&\frak m(\Gamma,v,e)(x_1,\ldots,x_k,x_0) \\
&= (-1)^{*_1} 
\langle \frak m_{\ell,\beta(v)}(y_1,\ldots,y_{\ell}),
G(\frak m_{\ell',\beta(v')}(z_1,\ldots,z_{\ell'}))\rangle 
\endaligned
\end{equation}
where 
$$
*_1 = (\deg' x_{j_1} +\deg' x_{j_1+1} + \ldots + \deg' x_{0})
(\deg' x_{1} + \ldots + \deg' x_{j_1-1}).
$$
On the other hand, we have
\begin{equation}\label{mgammav'}
\aligned
&\frak m(\Gamma,v',e)(x_1,\ldots,x_k,x_0) \\
&= (-1)^{*_2} 
\langle \frak m_{\ell',\beta(v')}(z_1,\ldots,z_{\ell'}),
G(\frak m_{\ell,\beta(v)}(y_1,\ldots,y_{\ell}))\rangle 
\endaligned
\end{equation}
where 
$$
*_2 = (\deg' x_{j'_1} +\deg' x_{j'_1+1} + \ldots + \deg' x_{0})
(\deg' x_{1} + \ldots + \deg' x_{j'_1-1}).
$$
(\ref{mgammav})  coincides with (\ref{mgammav'})  by (\ref{Gdualformula}).
(We use the fact that degree of $\frak m$ is $+1$ here.)
\end{proof}
\begin{lmm}\label{cysycanlem}
$\frak m_{k,\beta}^{\text{\rm can}}$ is cyclically symmetric.
\end{lmm}
\begin{proof}
We consider 
\begin{equation}\label{mcyclkeisan}
\aligned
\langle \frak m_{k,\beta}^{\text{\rm can}}(x_1,\ldots,x_k),x_0\rangle 
&= \sum_{\Gamma \in Gr(k,\beta)} 
\langle \frak m_{\Gamma}^{\text{\rm can}}(x_1,\ldots,x_k),x_0\rangle\\
&= \sum_{\Gamma \in Gr(k,\beta)}
\frak m(\Gamma,(v'_0,e_0))(x_1,\ldots,x_k,x_0).
\endaligned
\end{equation}
Here $e_0$ be the edge containing $v_0$ (the root of $\Gamma$), and 
$v'_0$ be the other edge of $e_0$.
\par
We number exterior edges of $\Gamma$ as $v_0,\ldots,v_k$ so that 
it respects counter clockwise cyclic order. Let $e_i$ be the edge 
containing $v_i$ and let $v'_i$ be the other vertex of $e_i$.
Proposition  \ref{independofpoints} implies that (\ref{mcyclkeisan}) is 
equal to 
$$
(-1)^{*}
\sum_{\Gamma \in Gr(k,\beta)}
\frak m(\Gamma,(v'_i,e_i))(x_1,\ldots,x_k,x_0).
$$
Here 
$* = (\deg'x_{i+1}+\ldots+\deg'x_0)(\deg'x_1+\ldots+\deg'x_{i})$.
Clearly this is equal to 
$$
(-1)^*\langle \frak m_{k,\beta}^{\text{\rm can}}(x_{i+1},\ldots,x_{i-1}),x_{i}\rangle 
$$
The proof of Lemma \ref{cysycanlem} is now complete.
\end{proof}
\begin{lmm}\label{fcycl}
$\{\frak f_{k,\beta}\}$ is cyclic.
\end{lmm}
\begin{proof}
We consider 
\begin{equation}\label{1017}
\langle \frak f(\Gamma_1)(x_1,\ldots,x_{k_1}),
\frak f(\Gamma_2)(x_{k_1+1},\ldots,x_{k}) \rangle
\end{equation}
where $\Gamma_i \in Gr(k_i,\beta_i)$, $k_1+k_2=k+1$, $\beta_1+\beta_2=\beta$.
\par
We remark that the image of $\frak f(\Gamma_i)$ is in $\overline D = \text{\rm Im}\,G$ if 
$(k_i,\beta_i) \ne (1,(0,0))$.   If $(k_i,\beta_i) = (1,(0,0))$ then the image of $\frak f(\Gamma_i)$ is in $\overline H$.
Moreover $\langle \overline D,\overline D\rangle = \langle \overline D,\overline H\rangle = 0$.
Therefore (\ref{1017}) is $0$ unless $(k,\beta) = (2,(0,0))$. In the case 
$(k,\beta) = (2,(0,0))$, (\ref{1017}) is $\langle x_1,x_2\rangle$.
The lemma follows.
\end{proof}
The proof of Theorem \ref{cancyclic} is complete.
\qed 
\begin{proof}[Proof of Theorem \ref{canopseudo}]
Let $(C,\langle \cdot \rangle,\{\frak m_{k,\beta}^t\},\{\frak c_{k,\beta}^t\})$ 
be a pseudo-isotopy. We take $G$, $\Pi$ as in Lemma \ref{Gdual}.
Since $\frak m_{1,(0,0)}^t$ is independent of $t$ we can choose $G$, $\Pi$
to be independent of $t$.
We take canonical model $(H,\langle \cdot \rangle,\{\frak m_{k,\beta}^{t \text{\rm can}}\})$
for each fixed $t$. It is easy to see from construction that $\frak m_{k,\beta}^{t \text{\rm can}}$
is smooth with respect to $t$. We next define $\frak c_{k,\beta}^{t \text{\rm can}}$.
\par
We consider a pair $(\Gamma,v_s)$ of $\Gamma \in Gr(k,\beta)$ and an 
interior vertex $v_s$ of $\Gamma$. We denote by $Gr^+(k,\beta)$ the set of 
all such pairs. We will define
\begin{equation}
\frak c^t(\Gamma,v_s) : B_k(\overline H[1]) \to \overline H[1]
\end{equation}
of degree $-\mu(\beta)$ and 
\begin{equation}
\frak h^t(\Gamma,v_s) : B_k(\overline H[1]) \to \overline c[1]
\end{equation}
of degree $-1-\mu(\beta)$, by induction on $\#C_0^{\text{\rm int}}(\Gamma)$.
\par
Since $v_s \in C_0^{\text{\rm int}}(\Gamma)$ we have 
$\#C_0^{\text{\rm int}}(\Gamma)\ge 1$. 
\par
If $\#C_0^{\text{\rm int}}(\Gamma)= 1$ we put
$$
\aligned
\frak h^t(\Gamma,v_s)(x_1,\ldots,x_k) &= G(\frak c^t_{k,\beta(v_s)}(x_1,\ldots,x_k)) \\
\frak c^t(\Gamma,v_s)(x_1,\ldots,x_k) &= \Pi(\frak c^t_{k,\beta(v_s)}(x_1,\ldots,x_k)).
\endaligned
$$
Suppose $\#C_0^{\text{\rm int}}(\Gamma)>1$.  Let $e_1$ be the edge containing $v_0$ and 
$v$ be the other vertex of $e_1$. 
\par\smallskip
\noindent(Case 1): $v=v_s$.
\par
We  remove $e,v,v_0$ from $\Gamma$ and obtain $\Gamma_1,\ldots,\Gamma_{\ell}$.
(We assume $e,\Gamma_1,\ldots,\Gamma_{\ell}$ respects counter clockwise cyclic order.)
We put
$$
\aligned
\frak h^t(\Gamma,v_s) &= G\circ \frak c^t_{{\ell},\beta(v_s)}\circ 
(\frak f^t(\Gamma_1) \otimes \ldots \otimes \frak f^t(\Gamma_{\ell})),\\
\frak c^t(\Gamma,v_s) &= \Pi\circ \frak c^t_{{\ell},\beta(v_s)}\circ 
(\frak f^t(\Gamma_1) \otimes \ldots \otimes \frak f^t(\Gamma_{\ell})).
\endaligned
$$
Here $\frak f^t(\Gamma)$ is the operator which appeared in the definition of $\frak f^t_{k,\beta}$.
\par\smallskip
\noindent(Case 2): $v\ne v_s$.
\par
We  remove $e,v,v_0$ from $\Gamma$ and obtain $\Gamma_1,\ldots,\Gamma_{\ell}$.
(We assume $e,\Gamma_1,\ldots,\Gamma_{\ell}$ respects counter clockwise cyclic order.)
We assume $v_s \in \Gamma_i$. We now put
$$
\aligned
\frak h^t(\Gamma,v_s) &= G\circ \frak m^t_{{\ell},\beta(v_s)}\circ 
(\frak f^t(\Gamma_1) \otimes \ldots \otimes \frak h^t(\Gamma_i) \otimes \ldots \otimes \frak f^t(\Gamma_{\ell})),\\
\frak c^t(\Gamma,v_s) &= -\Pi\circ \frak m^t_{{\ell},\beta(v_s)}\circ 
(\frak f^t(\Gamma_1) \otimes \ldots \otimes \frak h^t(\Gamma_i) \otimes \ldots \otimes \frak f^t(\Gamma_{\ell})).
\endaligned
$$
Note $\frak f^t(\Gamma)$ is of even degree and  $\frak h^t(\Gamma)$ is of odd degree.
We take their tensor product as follows:
\begin{equation}\label{tenssign}
\aligned
&(\frak f^t(\Gamma_1) \otimes \ldots \otimes \frak h^t(\Gamma_i) \otimes \ldots \otimes \frak f^t(\Gamma_{\ell}))
(x_1,\ldots,x_k) \\
&= (-1)^{*_i} 
\frak f^t(\Gamma_1)(x_1,\ldots,x_{k_1}) \otimes \ldots 
\otimes \frak h^t(\Gamma_i)(x_{j},\ldots,x_{j+k_i-1}) \otimes \ldots \\
&\hskip5cm\otimes \frak f^t(\Gamma_{\ell})
(x_{k-k_{\ell}+1},\ldots,x_k).
\endaligned
\end{equation}
Here $j = k_1+k_2+\ldots+k_{i-1}+1$, 
$
*_i = \deg' x_1+ \ldots + \deg'x_{j-1}.
$
\par
We remark that the minus sign in the definition of $\frak c^t(\Gamma,v_s)$ appears 
since we change the order of $\frak m^T_{\ell,\beta(v_s)}$ and $dt$.
\par
We put
$$
\frak c^{t \text{can}}_{k,\beta} = \sum_{(\Gamma,v_s) \in Gr^+(k,\beta)}
\frak c^t(\Gamma,v_s).
$$
It is easy to see that 
$(H,\langle \cdot \rangle,\{\frak m_{k,\beta}^{t \text{can}}\},\{\frak c_{k,\beta}^{t \text{can}}\})$
satisfies 1,2,5 of Definition \ref{pisotopydef}. We next prove 4 by using Lemma \ref{dtcomiAinf}.
We consider $C^{\infty}([0,1],\overline H)$ and define 
$\hat{\frak m}_{k,\beta}^{\text{\rm can}}$ by (\ref{combineainf}).
We next define
$$
\frak F_{k,\beta} : B_k(C^{\infty}([0,1],\overline H)[1]) \to C^{\infty}([0,1],\overline H)[1]
$$
as follows. 
Let $x_i(t) + dt \wedge y_i(t) = \text{\bf x}_i \in C^{\infty}([0,1],\overline H)$.
We put
$$
\frak F_{k,\beta}(\text{\bf x}_1,\ldots,\text{\bf x}_k)
= x(t) + dt \wedge y(t)
$$
where
\begin{subequations}\label{combineFF}
\begin{equation}
x(t) = \frak f^t_{k,\beta}(x_1(t),\ldots,x_k(t))
\end{equation}
\begin{equation}\label{combineainfmainFF}
\aligned
y(t) 
= 
& \frak h^t_{k,\beta}
(x_1(t),\ldots,x_k(t)) \\
&-\sum_{i=1}^k (-1)^{*_i} \frak m^t_{k,\beta}
(x_1(t),\ldots,x_{i-1}(t),y_i(t),x_{i+1}(t),\ldots,x_k(t))
\endaligned\end{equation}
if $(k,\beta) \ne (1,(0,0))$ and 
\begin{equation}
y(t) = \frac{d}{dt} x_1(t) 
\end{equation}
if $(k,\beta) = (1,(0,0))$. Here $*_i$ in (\ref{combineainfmainFF}) is 
$*_i = \deg' x_1 +\ldots+\deg'x_{i-1}$.
(We remark that $\frak f^t_{1,(0,0)} =$ identity and $\frak h^t_{1,(0,0)} = 0$.)
\end{subequations}
\begin{lmm}\label{isotopcan}
$\{\frak F_{k,\beta} \}$ is a filtered $A_{\infty}$ homomorphism.
\end{lmm}
\begin{proof}
We regard $G$ and $\Pi$ as homomorphisms
$$
C^{\infty}([0,1]\times \overline C) \to C^{\infty}([0,1]\times \overline C).
$$
Then we can apply (the proof of) Lemma \ref{canhomoOK}.
It implies Lemma \ref{isotopcan}.
\end{proof}
\begin{crl}
${\frak c}^{t\text{\rm can}}$ and ${\frak m}^{t\text{\rm can}}$ satisfies 
$(\ref{isotopymaineq})$.
\end{crl}
\begin{proof}
Since $\frak F_{1,(0,0)}$ is injective, 
Lemma \ref{isotopcan} implies that  $\hat{\frak m}^{t\text{\rm can}}$ satisfies 
$A_{\infty}$ relation. The Corollary then follows from Lemma \ref{dtcomiAinf}. 
\end{proof}
\begin{lmm}
${\frak c}^{t\text{\rm can}}$  is cyclically symmetric.
\end{lmm}
The proof is similar to the proof of Lemma \ref{mcancyclic} and so is omitted.
The proof of Theorem \ref{canopseudo} is now complete.
(The proof of unitality is easy.)
\end{proof}
\section{Geometric realization of pseudo-isotopy of cyclic filtered $A_{\infty}$ algebras}
\label{geoisoto}
The main result of this section is the following Theorem \ref{existpseudoisotopy}.
Let $(M,\omega)$ be a symplectic manifold and $L$ be a relatively spin Lagrangian submanifold.
We take two almost complex structures $J_0, J_1$ tamed by $\omega$.
Let $E_0 \le E_1$. We apply Theorem \ref{existifilaAinfmodTE} for $L,J_0,E_0$ and $L,J_1,E_1$.
Then we obtain $G(J_i)$ gapped cyclic filtered $A_{\infty}$ structures 
$(\Lambda(L),\{\frak m_{k,\beta}^{(i)}\},\langle \cdot \rangle)$ module $T^{E_i}$ on the de-Rham complex.
(Note the discrete submonoid $G(J)$ depends on the almost complex structure $J$.
See the begining of the proof of Theorem \ref{existifilaAinfmodTE}.)
\begin{thm}\label{existpseudoisotopy}
There exists $G \supset G(J_0), G(J_1)$ such that
$(\Lambda(L),\langle \cdot \rangle,\{\frak m_{k,\beta}^{(0)}\})$
is pseudo-isotopic to 
$(\Lambda(L),\langle \cdot \rangle,\{\frak m_{k,\beta}^{(1)}\})$
as $G$ gapped cyclic unital filtered $A_{\infty}$ algebras modulo $T^{E_0}$.
\end{thm}
Before proving Theorem \ref{existpseudoisotopy} we clarify the definition of 
pseudo-isotopy in the (inifite dimensional) case of de Rham complex.
We consider $\overline C = \Lambda(L)$. We put
\begin{equation}\label{Cinfty[01]deram}
C^{\infty}([0,1]\times\overline C) = \Lambda([0,1] \times L).
\end{equation}
Note an element of $C^{\infty}([0,1]\times\overline C)$ is uniquely written as 
$$
x(t) + dt \wedge y(t),
$$
where $x(t), y(t) \in \Lambda(L)$.
So this notation is consistent with one in Section \ref{homalgsec}. However 
the assumption on the smoothness here on the map
$t \mapsto x(t)$, $t \mapsto y(t)$ is different from finite dimensional case.
\par
We conider the operations 
$
\frak m^t_{k,\beta} : B_k(\overline C[1]) \to \overline C[1]
$
of degree $-\mu(\beta)+1$ and
$
\frak c^t_{k,\beta} : B_k(\overline C[1]) \to \overline C[1]
$
of degree $-\mu(\beta)$.
\begin{dfn}\label{smoothmtandct2}
We say $\frak m^t_{k,\beta}$ is {\it smooth} if for 
each $x_1,\ldots,x_k$
$$
t\mapsto \frak m^t_{k,\beta}(x_1,\ldots,x_k)
$$
is an element of $C^{\infty}([0,1]\times\overline C)$.
\par
The smoothness of $\frak c^t_{k,\beta}$ is defined in the same way.
\end{dfn}
This is the same as Definition \ref{smoothmtandct2}, except we use 
(\ref{Cinfty[01]deram}) for $C^{\infty}([0,1]\times\overline C)$.
We use this definition of the smoothness and define pseudo-isotopy on 
de-Rham complex in the same way as Definition \ref{pisotopydef}.
Theorems \ref{pisotopyextention}, \ref{pisotohomotopyequiv} can be proved in the same way.
\begin{proof}[Proof of Theorem \ref{existpseudoisotopy}]
We take a path $\mathcal J = \{J_t\}_{t\in[0,1]}$ of tame almost complex structures 
joining $J_0$ to $J_1$. We consider the moduli spaces 
$\mathcal M_k(\beta;J_t)$ of $J_t$-holomorphic discs of homology class
$\beta \in H_2(M,L;\Z)$.
We put:
\begin{equation}
\mathcal M_k(\beta;\mathcal J) = \bigcup_{t \in [0,1]} \{t\} \times
\mathcal M_k(\beta;J_t).
\end{equation}
We have evaluation maps 
$$
ev = (ev_0,\ldots,ev_{k-1}) : \mathcal M_k(\beta;\mathcal J) \to L^k
$$
together with 
$ev_t :\mathcal M_k(\beta;\mathcal J) \to [0,1]$
where $ev_t( \{t\} \times
\mathcal M_k(\beta;J_t)) = \{t\}$.
\begin{lmm}\label{forgetcompKUrafami}
There exists a system of Kuranishi structures on $\mathcal M_{1}(\beta;\mathcal J)$ and 
$\mathcal M_{0}(\beta;\mathcal J)$ with the following properties.
\smallskip
\begin{enumerate}
\item $ev_t$ extends to a strongly continuous and weakly submersive map.
So it induces a Kuranishi structure on $\mathcal M_k(\beta;J_t)$ for each of $t\in [0,1]$, $k=0,1$.
\item The induced Kuranishi structure on $\mathcal M_k(\beta;J_i)$ for $i=0,1$, $k=0,1$ coincides with 
one produced in Theorem \ref{forgetcompKUra}.
\par
\item For any $t\in [0,1]$, $k=0,1$, the induced Kuranshi structure on  $\mathcal M_k(\beta;J_t)$ 
satisfies the conclusions of Theorem \ref{forgetcompKUra}.
\end{enumerate}
\end{lmm}
\begin{proof}
The proof is the same as the proof of Theorem \ref{forgetcompKUra}. Namely we define 
the obstruction bundle, that is a subspace of $C^{\infty}(\Sigma,u^*TM\otimes \Lambda^{01})$, 
enough large 
so that the submersivity of $ev_0$ and $ev_t$ holds. We can do it 
inductively so that the part already defined are untouched.
\end{proof}
\begin{lmm}\label{Corkurafami}
There exists a system of Kuranishi structures on 
$\mathcal M_{\ell,k+1}(\beta;\mathcal J)$ $k\ge 0$,
with the following properties.
\par\smallskip
\begin{enumerate}
\item $ev_t$ extends to a strongly continuous and weakly submersive map.
So it induces a Kuranishi structure on $\mathcal M_k(\beta;J_t)$ for each of $t\in [0,1]$.
\item The induced Kuranishi structure on $\mathcal M_k(\beta;J_i)$ for $i=0,1$ coincides with 
one produced in Corollary \ref{Corkura}.
\item For any $t\in [0,1]$, the induced Kuranshi structure on  $\mathcal M_k(\beta;J_t)$ 
satisfies the conclusion of Corollary \ref{Corkura}.
\end{enumerate}
\end{lmm}
\begin{proof}
The proof is the same as the proof of Corollary \ref{Corkura}.
\end{proof}
\begin{lmm}\label{multicontforgetmainfami}
For each $\epsilon > 0$, there exists a 
compatible systems of familis of multisections 
$(U_{\tilde\alpha,i},W_{\tilde\alpha},\{\frak s_{\tilde\alpha,i,j}\})$, 
$(U_{\alpha,i},W_{\alpha},\{\frak s_{\alpha,i,j}\})$ on 
$\mathcal M_{1}(\beta;\mathcal J)$, $\mathcal M_{0}(\beta;\mathcal J)$ 
for $\beta \cap \omega \le E_0$, with the following properties:
\par\smallskip
\begin{enumerate}
\item At $t=0,1$ they coincide with the family of multisections produced in Theorem 
\ref{multicontforgetmain}.
\item They are $\epsilon$ close to the Kuranishi map.
\item They are transversal to $0$ in the sense of Definition \ref{defnmultisec}.$3$.
\item $(ev_0,ev_t) : \mathcal M_{1}(\beta;\mathcal J) \to L\times [0,1]$ induces submersions 
$(ev_0)_{\tilde\alpha}\vert_{\frak s_{\tilde\alpha}^{-1}(0)} : \frak s_{\tilde\alpha}^{-1}(0) \to L\times [0,1]$.
\item They are compatible with $(\ref{gluemap})$ in the same sense as  Theorem 
\ref{multicontforgetmain}.
\end{enumerate}
\end{lmm}
\begin{proof}
The proof is the same as the proof of  Theorem 
\ref{multicontforgetmain}.
\end{proof}
\begin{rem}
We remark that 4 above implies that submersivity of $ev_0: 
\mathcal M_{1}(\beta;J_t) \to L$ for {\it any} $t$. Since there exists uncountably many $t$'s 
we can not do it in case we are working with multi but finitely many valued sections. 
Since we are working with continuous family of multisections this becomes possible. 
In fact we can take the dimension of our parameter space $W$ as large as we want.
\end{rem}
\begin{lmm}\label{Corkuramultifami}
For each $\epsilon$ and $E_0$, 
there exists a system of continuous family of multisections on 
$\mathcal M_{k+1}(\beta;\mathcal J)$, $k\ge 0$,
$\beta \cap \omega \le E_0$, 
with the following properties.
\par\smallskip
\begin{enumerate}
\item At $t=0,1$ they coincides with the family of multisections produced in Corollary 
\ref{Corkuramulti}.
\item It is $\epsilon$ close to the Kuranishi map.
\item It is compatible with $\mathfrak{forget}_{k+1,1}$.
\item It is invariant under the cyclic permutation of the 
boundary marked points.
\item It is invariant by the permutation of interior marked points.
\item $(ev_0,ev_t) : \mathcal M_{k+1}(\beta;\mathcal J) \to L\times [0,1]$ induces a submersion on its zero set.
\item We consider the decomposition of the boundary:
\begin{equation}\label{bdcompati}
\aligned
\partial \mathcal M_{k+1}(\beta;\mathcal J)
\supset
&\bigcup_{1\le i\le j+1 \le k+1} 
\bigcup_{\beta_1+\beta_2=\beta}\\
&\mathcal M_{j-i+1}(\beta_1;\mathcal J) {}_{(ev_0,ev_t)} \times_{(ev_i,ev_t)} 
\mathcal M_{k-j+i}(\beta_2;\mathcal J).
\endaligned
\end{equation}
Then the restriction of our family of multisections of 
$\mathcal M_{k+1}(\beta;\mathcal J)$ in the left hand side 
coincides with the fiber product family of multisections in 
the right hand side.
\end{enumerate}
\end{lmm}
\begin{proof}
The proof is the same as the proof of  Corollary 
\ref{Corkuramulti}.
\end{proof}
Now we are in the position to complete the proof of Theorem \ref{existpseudoisotopy}.
Let $\rho_1,\ldots,\rho_k \in \Lambda(L)$. We put 
\begin{equation}
\text{\rm Corr}_*(\mathcal M_{k+1}(\beta;\mathcal J);(ev_{1},\ldots,ev_k),ev_0)
(\rho_1\times\ldots\times\rho_k)
=
\rho(t) + dt \wedge \sigma(t).
\end{equation}
Here we use continuous family of multisections produced in Lemma \ref{Corkuramultifami} to define 
the left hand side. We define
\begin{equation}\label{pseudoisoshiki}
\frak m^t_{k,\beta}(\rho_1,\ldots,\rho_k) = \rho(t),  \qquad
\frak c^t_{k,\beta}(\rho_1,\ldots,\rho_k) = \sigma(t).
\end{equation}
Using Lemma \ref{Corkuramultifami} we can prove that they satisfies the required 
properties in the same way as the proof of Theorem \ref{existifilaAinfmodTE}.
\end{proof}
\section{Cyclic filtered $A_{\infty}$ structure on de-Rham 
complex and 
on de-Rham cohomology}
\label{cycAinfseccoh}
In this section, we will use Theorems  \ref{existifilaAinfmodTE} and \ref{existpseudoisotopy}
to produce gapped, cyclic and unital filtered $A_{\infty}$ structure on de Rham cohomology.
We first construct cyclic filtered $A_{\infty}$ structure on de Rham complex.
\begin{thm}\label{maindehamcomplex}
For any relatively spin Lagrangian submanifold $L$, we can 
associate a gapped cyclic unital filtered $A_{\infty}$ algebra, 
$(\Lambda(L),\langle \cdot \rangle,\{\frak m_{k,\beta}\})$
on its de Rham complex.
\par
It is independent of the choices up to 
pseudo-isotoy as cyclic unital filtered $A_{\infty}$  modulo $T^{E}$ for any $E$.
\end{thm}
\begin{proof}
We fix $J$ and take $E_1 < E_2 < \ldots$. For each $E_i$ we apply 
Theorem  \ref{existifilaAinfmodTE} to obtain 
gaped cyclic unital filtered $A_{\infty}$ algebra modulo $T^{E_i}$, 
$(\Lambda(L),\langle \cdot \rangle,\{\frak m_{k,\beta}^{(i)}\})$ 
on de Rham complex.
By Theorem  \ref{existpseudoisotopy} there exists a pseudo-isotopy 
$(\Lambda(L),\langle \cdot \rangle,\{\frak m_{k,\beta}^{t,(i)}\},\{\frak c_{k,\beta}^{t,(i)}\})$
of gaped cyclic unital filtered $A_{\infty}$ algebra modulo $T^{E_i}$
between 
$(\Lambda(L),\langle \cdot \rangle,\{\frak m_{k,\beta}^{(i)}\})$  and 
$(\Lambda(L),\langle \cdot \rangle,\{\frak m_{k,\beta}^{(i+1)}\})$. 
Now we use 
Theorem \ref{pisotohomotopyequiv} 
to extend unital pseudo-isotpy modulo $T^{E_0}$,
$(\Lambda(L),\langle \cdot \rangle,\{\frak m_{k,\beta}^{t,(i)}\},\{\frak c_{k,\beta}^{t,(i)}\})$
and unital and cyclic filtered $A_{\infty}$ algebra modulo $T^{E_0}$,
and 
$(\Lambda(L),\langle \cdot \rangle,\{\frak m_{k,\beta}^{(i)}\})$ to a unital pseudo-isotopy and 
unital and cyclic filtered $A_{\infty}$ algebra.
(See \cite{FOOO080} Subsection 7.2.8.)
\par
Let us take two choices $J_j$ $(j=0,1)$ of $J$ and perturbation etc.
and obtain cyclic unital filtered $A_{\infty}$ structure 
extending one on  $(\Lambda(L),\langle \cdot \rangle,\{\frak m_{k,\beta}^{(i_0,j)}\})$.
We take $i$ such that $E_i > E$. Then 
by construction 
$(\Lambda(L),\langle \cdot \rangle,\{\frak m_{k,\beta}^{(i_0,j)}\})$ is 
pseudo-isotopic to 
$(\Lambda(L),\langle \cdot \rangle,\{\frak m_{k,\beta}^{(i,j)}\})$ as 
cyclic unital filtered $A_{\infty}$ algebra. (Here modulo $T^E$ is superfluous.)
On the other hand, by Theorem  \ref{existpseudoisotopy} 
$(\Lambda(L),\langle \cdot \rangle,\{\frak m_{k,\beta}^{(i,0)}\})$ is 
pseudo-isotopic to 
$(\Lambda(L),\langle \cdot \rangle,\{\frak m_{k,\beta}^{(i,1)}\})$
as cyclic unital filtered $A_{\infty}$ algebra modulo $T^E$.
The uniqueness part of Theorem \ref{maindehamcomplex} follows.
\end{proof} 
Theorems \ref{maindehamcomplex}, \ref{cancyclic} and \ref{canopseudo} immediately 
imply the following.
\begin{crl}\label{maindehamcoh}
For any relatively spin Lagrangian submanifold $L$, we can 
associate a gapped cyclic unital filtered $A_{\infty}$ algebra, 
$(H(L),\langle \cdot \rangle,\{\frak m_{k,\beta}\})$
on its de Rham cohomology.
\par
It is independent of the choices up to 
homotopy equivalence as cyclic unital filtered $A_{\infty}$  modulo $T^{E}$ for any $E$.
\end{crl}
\begin{rem}\label{rem19}
In Corollary \ref{maindehamcoh} we proved that the cyclic filtered $A_{\infty}$ structure on 
the de Rham cohomology is well defined up to homotopy equivalence modulo $T^E$ but 
{\it not} up to up to pseudo-isotopy modulo $T^E$ .
Theorem \ref{canopseudo} implies that it is well defined up to pseudo-isotopy modulo $T^E$ once we fix 
operators $G$ and $\Pi$ satisfying the conclusion of Lemma \ref{Gdual}.
It does not seem to be so immediate to prove its independence of $G$ and $\Pi$ 
up to pseudo-isotopy (modulo $T^E$). The proof up to homotopy equivalence follows from 
the fact $\frak f$ in Theorem \ref{cancyclic} is homotopy equivalence.
\par
So Theorem 
\ref{maindehamcomplex} gives stronger conclusion than Corollary \ref{maindehamcoh}.
\end{rem}
\begin{rem}
The difference between `up to pseudo-isotopy modulo $T^E$ for arbitrary $E$' and 
`up to pseudo-isotopy' is not important for most of the applications. 
To improve the statement of Theorem 
\ref{maindehamcomplex} to `up to pseudo-isotopy' we need to work out 
`pseudo-isotopy of pseudo-isotopies'. We will present the detail of this construction 
in Section \ref{pisoofpisosec}, for completeness.
\end{rem}
\begin{rem}
Here we are using bifurcation method rather than cobordism method.
(See \cite{FOOO080} Subsection 7.2.14 for the comparison between these two methods.)
In \cite{FOOO080} we used cobordism method mainly. 
In \cite{AkJo08} bifurcation method is used.
The reason why we use bifurcation method here is related to Remark \ref{rem19}.
Namely pseudo-isotopy seems stronger than homotopy equivalence.
\par
By carefully looking the proof of Theorem \ref{pisotohomotopyequiv}, we find that 
they finally give the same homotopy equivalence. In fact 
`time ordered product', which was used in \cite{FOOO080} appears 
during the proof of Theorem \ref{pisotohomotopyequiv}.
\end{rem}
\section{Adic convergence of filtered $A_{\infty}$ structure}
\label{adicconv}
In this section we prove Theorem
\ref{main2}. 
We begin with the proof of properties 1 and 2 in Theorem
\ref{main2}.  
We define convergence used here first.
Let $G \subset \R_{\ge 0} \times 2\Z$ be a discrete submonoid.
Let $\overline C$ be a finite dimensional $\R$ vector space and 
$C_G = \overline C \otimes_{\R} \Lambda_0^G$, 
$C = \overline C \otimes_{\R} \Lambda_{0,nov}$.
\begin{dfn}\label{mixconv}
A sequence of elements $v_i \in C_G$ said to converge to 
$v$ if 
$$
v_i = \sum_{\beta \in G} v_{i,\beta} T^{E(\beta)}e^{\mu(\beta)/2}, 
\quad v = \sum_{\beta \in G} v_{\beta} T^{E(\beta)}e^{\mu(\beta)/2} 
$$
and if seach of $v_{i,\beta}$ converges to $v_{\beta}$ in the topology of $\overline C$, 
(induced by the ordinary topology of $\R$.)
\par
 $v_i \in C$ said to converge to 
$v$ if there exists $G$ independent of $i$ such that $v_i,v \in C_G$
and $v_i$ converges to $v$.
\end{dfn}
\par
Let $L$ be a relatively spin Lagrangian submanifold of $M$. 
We take an almost complex structure $J$ on $M$ tamed by $\omega$.
We consider the cyclic unital filtered $A_{\infty}$ algebra 
$(\Lambda(L),\langle \cdot \rangle,\{\frak m_{k,\beta}\})$ in 
Theorem \ref{maindehamcomplex}
and 
$(H(L),\langle \cdot \rangle,\{\frak m_{k,\beta}^{\text{\rm can}}\})$ in 
Corollary \ref{maindehamcoh}. 
Let $\text{\bf e}_1,\ldots,\text{\bf e}_{b_1}$ be a basis of $H^1(L;\Z)$.
We take its representative as a closed one form and denote it by the same symbol.
We put 
\begin{equation}\label{deg1b}
\text{\bf b} = \sum_{i=1}^{b_1} x_i \text{\bf e}_1
\end{equation}
with $x_i \in \Lambda_G^{(0)}$. (Namely $x_i$ does not contain $e$ the grading parameter of 
$\Lambda_{0,nov}$ 
and $\text{\bf x}_i \in \Lambda(L)$ ($i=1,\ldots,k$).
\begin{lmm}\label{m1jikurasireru}
\begin{equation}\label{forgetexponent}
\aligned
\sum_{m_0+\ldots+m_k = m}
&\frak m_{k+m,\beta}(\text{\bf b}^{\otimes m_0},\text{\bf x}_1,\text{\bf b}^{\otimes m_1},\ldots,
\text{\bf b}^{\otimes m_{k-1}},\text{\bf x}_k,\text{\bf b}^{\otimes m_k})
\\
&= \frac{1}{m !}\left(\sum_{i=1}^{b_1} (\partial \beta\cap \text{\bf e}_i) x_i\right)
^{m}\frak m_{k}(\text{\bf x}_1,\ldots,\text{\bf x}_k).
\endaligned
\end{equation}
\end{lmm}
\begin{proof}
We consider the set of $(m_0,\ldots,m_k) \in \Z_{\ge 0}$ with $m_0+\ldots+m_k = m$ and denote it by 
$A(m)$. For each  $\vec{m} = (m_0,\ldots,m_k) \in A(m)$ we take a copy of 
$\mathcal M_{k+m}(\beta)$ and denote it by $\mathcal M_{\vec{m}\,}(\beta)$.
We then cosider the forgetful map
\begin{equation}\label{lforget}
\frak{forget}_{\vec{m}} : \mathcal M_{\vec{m}\,}(\beta)  \to \mathcal M_{k}(\beta).
\end{equation}
which forget
1st,\dots,$m_0$th, $m_0+2$nd, $m_0+m_1+1$st, $m_0+m_1+3$rd, and, 
\dots,$m_0+\ldots+m_i+1$st, $m_0+\ldots+m_i+i+3$rd, \dots marked points.
In other words, we forget the marked points where $\text{\bf b}$ are assigned in the 
left hand side of (\ref{forgetexponent}).
\par
We consider 
\begin{equation}\label{extraeval}
ev_{\vec{m}} : \mathcal M_{\vec{m}\,}(\beta) \to L^{m}
\end{equation}
the evaluation map at the marked points which we forget in (\ref{lforget}).
We take the (continuous family of) perturbations as in Corollary \ref{Corkuramulti}.
We write its zero set as $\mathcal M_{\vec{m}\,}(\beta)^{\frak s}$,
$\mathcal M_{k}(\beta)^{\frak s}$. Since the perturbation is compatible with 
forgetful map there exists a map
\begin{equation}\label{lforget2}
\frak{forget}_{\vec{m}}^{\frak s} : \mathcal M_{\vec{m}\,}(\beta)^{\frak s}  \to \mathcal M_{k}(\beta)^{\frak s}.
\end{equation}
For each $\text{\bf p} \in \mathcal M_{k}(\beta)^{\frak s}$ the fiber 
$(\frak{forget}_{\vec{m}}^{\frak s})^{-1}(\text{\bf p})$ is $m$ dimensional.
Moreover we have the following. We represent $\text{\bf p}$ by $(\Sigma,u)$. Then 
the cycle
$$
\sum_{\vec{m}} (ev_{\vec{m}})_*((\frak{forget}_{\vec{m}}^{\frak s})^{-1}(\text{\bf p}))
$$
is equal to 
$$
\{ (u(t_1),\ldots,u(t_{m})) \mid t_1,\ldots,t_{m} \in [0,1),\,\,\,
t_1 \le \ldots \le t_{m} \}
$$
as currents. Here we identify $\partial\Sigma = [0,1)$ so that $0$ corresponds to the $0$th boundary 
marked point.
\par
Therefore we have
$$
\sum_{\vec{m}} \int_{(\frak{forget}_{\vec{m}}^{\frak s})^{-1}(\text{\bf p})}ev_{\vec{m}\,}^* 
(\text{\bf b} \times \ldots \times \text{\bf b})
=  \frac{1}{m !}\left(\sum_{i=1}^{b_1} (\partial \beta\cap \text{\bf e}_i) x_i\right)^{m}.
$$
(\ref{forgetexponent}) follows.
($1/m!$ is the volume of the domain $\{(t_1,\ldots,t_{m}) \mid t_1,\ldots,t_{m} \in [0,1),\,\,\,
t_1 \le \ldots \le t_{m}\}$.)
\end{proof}
\begin{lmm}\label{m1jikurasireru2}
\begin{equation}\label{forgetexponent2}
\aligned
\sum_{m_0+\ldots+m_k = m}
&\frak m^{\text{\rm can}}_{k+m,\beta}(\text{\bf b}^{\otimes m_0},\text{\bf x}_1,\text{\bf b}^{\otimes m_1},\ldots,
\text{\bf b}^{\otimes m_{k-1}},\text{\bf x}_k,\text{\bf b}^{\otimes m_k})
\\
&= \frac{1}{m !}\left(\sum_{i=1}^{b_1} (\partial \beta\cap \text{\bf e}_i) x_i\right)^{m}\frak m_{k}^{\text{\rm can}}(\text{\bf x}_1,\ldots,\text{\bf x}_k).
\endaligned
\end{equation}
\end{lmm}
\begin{proof}
We have
\begin{equation}\label{forgetexponent3}
\aligned
\sum_{m_0+\ldots+m_k = m}
&\frak f_{k+m,\beta}(\text{\bf b}^{\otimes m_0},\text{\bf x}_1,\text{\bf b}^{\otimes m_1},\ldots,
\text{\bf b}^{\otimes m_{k-1}},\text{\bf x}_k,\text{\bf b}^{\otimes m_k})
\\
&= \frac{1}{m !}\left(\sum_{i=1}^{b_1} (\partial \beta\cap \text{\bf e}_i) x_i\right)^{m}\frak f_{k}(\text{\bf x}_1,\ldots,\text{\bf x}_k).
\endaligned
\end{equation}
by its inductive construction and (\ref{forgetexponent}).
(\ref{forgetexponent2}) then follows from definition, (\ref{forgetexponent}) and  (\ref{forgetexponent3}).
\end{proof}
\begin{crl}\label{m1jikurasireru3}
\begin{equation}\label{forgetexponent22}
\aligned
\lim_{N\to\infty}\sum_{m_0+\ldots+m_k = m \le N}
&\frak m^{\text{\rm can}}_{k+m,\beta}(\text{\bf b}^{\otimes m_0},\text{\bf x}_1,\text{\bf b}^{\otimes m_1},\ldots,
\text{\bf b}^{\otimes m_{k-1}},\text{\bf x}_k,\text{\bf b}^{\otimes m_k})
\\
&= \exp\left(\sum_{i=1}^{b_1} (\partial \beta\cap \text{\bf e}_i) x_i\right)\frak m_{k}^{\text{\rm can}}(\text{\bf x}_1,\ldots,\text{\bf x}_k).
\endaligned
\end{equation}
Namely the left hand side converges to the right hand side in the usual topology 
of $H(L;\R)$, that is the topology induced by the usual topology of $\R$.
\par
The right hand side depends only on $y_i = e^{x_i}$ and $\text{\bf x}_i$. Namely it is independent of 
the change $x_i \mapsto x_i + 2\pi\sqrt{-1}a_i$ for $a_i \in H(L;\Lambda_{0,nov}^{\Z})$.
\end{crl}
This is immediate from Lemma \ref{m1jikurasireru2}.
\par
So far we consider bounding chain $\text{\bf b}$ consisting of cohomology class of degree $1$.
The degree zero class does not appear in the bounding cochain. We next consider 
the class of degree $>1$.
We put
\begin{equation}\label{deghighb}
\text{\bf b}_{\text{\rm high}} = \sum_{i>b_1} x_i \text{\bf e}_i.
\end{equation}
Here $\text{\bf e}_i$, $i=b_1+1,\ldots$ is a basis of $\bigoplus_{d\ge1}H^{2d+1}(L;\Z)$.
\begin{lmm}\label{highconv}
There exists $E(m)$ such that $\lim_{m\to\infty} E(m)= \infty$ and that
\begin{equation}
T^{E(\beta)}\frak m^{\text{\rm can}}_{k+m,\beta}(\text{\bf b}_{\text{\rm high}}^{\otimes m_0},\text{\bf x}_1,\ldots,
\text{\bf x}_k,\text{\bf b}_{\text{\rm high}}^{\otimes m_k}) 
\equiv 0 \mod T^{E(m)}
\end{equation}
if $m = m_0+\ldots+m_k$.
$E(m)$ is independent of $\beta$.
\end{lmm}
\begin{proof}
By degree reason we have
$$
\mu(\beta) > m d + C
$$
where $C$ depends only on $\text{\bf x}_1, \ldots, \text{\bf x}_k$ and $L$.
By Gromov compactness (Definition \ref{discmonoid} 2,3) it implies that 
$E(\beta) \to \infty$ as $m\to \infty$.
\end{proof}
We put
\begin{equation}
\frak m_k = \sum_{\beta \in G} T^{E(\beta)}e^{\mu/2} \frak m_{k,\beta}.
\end{equation}
Then 1,2 of Theorem \ref{main2} follow from Corollary \ref{m1jikurasireru3} and Lemma \ref{highconv}.
\par\smallskip
We turn to the proof of 3,4 of Theorem \ref{main2}.
We take a Weinstein neighborhood $U$ of $L$.
Namely $U$ is symplectomorphic to a neighborhood $U'$ of zero section 
in $T^*L$. 
We choose $\delta_1$ so that for $c = (c_1,\ldots,c_b) \in [-\delta_1,+\delta_1]^b$ 
the graph of the closed one form $\sum_{i=1}^{b_1} c_i \text{\bf e}_i$ is contained in $U'$. 
We send it by the symplectomorphism to $U$ and denote it by 
$L(c)$.
We may take $\delta_2 < \delta_1$ such that if 
$c = (c_1,\ldots,c_b) \in [-\delta_2,+\delta_2]^b$ then there exists a diffeomorphism 
$F_{c} : M \to M$ such that 
\begin{eqnarray}
&&F_c(L) = L(c), \\
&&\text{$(F_c)_*J$ is tamed by $\omega$.}
\end{eqnarray}
\begin{rem}
It is essential here to consider tame almost complex structure rather than compatible almost complex structure.
In fact the compatibility is used 
to prove Gromov compactness which was actually 
proved in \cite{Grom85} for the tame almost complex structure.
In fact we can not take $F_c$ to be symplectomorphism in general.
So in general $(F_c)_*J$ is not compatible with $\omega$.
However it is tamed by $\omega$ if $c$ is sufficiently small.
This is because the condition for almost complex structure to be tame is an open condition.
\end{rem}
We consider the cyclic filtered $A_{\infty}$ algebra 
$(\Lambda(L(c)),\langle \cdot \rangle,\{\frak m_{k,\beta}^{(F_c)_*J}\})$.
We compare it with $(\Lambda(L),\langle \cdot \rangle,\{\frak m_{k,\beta}^{J}\})$.
(Here we include $(F_c)_*J$ and $J$ in the notation to specify the complex structure we use.)
The closed one forms $\text{\bf e}_i$ representing the basis $H^1(L;\Z)$ is transformed 
to a closed one form $\text{\bf e}_i(c)$ on $L(c)$ by the diffeomorphism $F_c$.
For $\text{\bf b}$ in (\ref{deg1b}) we put
$$
\text{\bf b}_c(x_1,\ldots,x_{b_1}) = \text{\bf b} = \sum_{i=1}^{b_i} (x_i  \,\cdot( F_{c})_*(\text{\bf e}_i)).
$$
We define 
$$
\text{\bf b}_{\text{\rm high},c}(x_{b_1+1},\ldots,x_b) = 
\text{\bf b}_{\text{\rm high},c} = \sum_{i>b_1} (x_i \,\cdot ( F_{c})_*(\text{\bf e}_i))
$$
and $\text{\bf b}_{c+} = \text{\bf b}_c + \text{\bf b}_{\text{\rm high},c}$, 
$\text{\bf b}_{+} = \text{\bf b} + \text{\bf b}_{\text{\rm high}}$.
\begin{lmm}\label{134}
We may choose perturbation etc. of the choices entering in the definition of  $\frak m^{(F_c)_*J}_{k,\beta}$, 
$\frak m^{J}_{k,\beta}$ such that the following holds.
\begin{equation}\label{changetocnochange}
\aligned
&\frak m^{(F_c)_*J}_{k+m,(F^{-1}_c)^*\beta}(\text{\bf b}_{c+}^{\otimes m_0},(F^{-1}_c)^*\text{\bf x}_1,
\text{\bf b}_{c+}^{\otimes m_1},\ldots,
\text{\bf b}_{c+}^{\otimes m_{k-1}},(F^{-1}_c)^*\text{\bf x}_k,\text{\bf b}_{c+}^{\otimes m_k})
\\
&=\frak m^{J}_{k+m,\beta}(\text{\bf b}_{+}^{\otimes m_0},\text{\bf x}_1,
\text{\bf b}_{+}^{\otimes m_1},\ldots,\text{\bf b}_{+}^{\otimes \ell_{k-1}},\text{\bf x}_k, 
\text{\bf b}_{+}^{\otimes m_k}).
\endaligned
\end{equation}
\end{lmm}
\begin{proof}
We remark that $F_c$ gives an isomorphism 
\begin{equation}\label{isomoduli}
(F_c)_* : \mathcal M_{k}(\beta;J) \to \mathcal M_{k}((F_c)_*\beta;(F_c)_*J).
\end{equation}
We can extend this isomorphism to one of Kuranishi structures. 
Therefore we can take the continuous family of perturbations in 
Corollary \ref{Corkuramulti} so that it is preserved by (\ref{isomoduli}).
The lemma follows immediately.
\end{proof}
\begin{lmm}\label{weightchange}
If $\partial\beta \cap \text{\bf e}_i = g_i$ then
$$
(F_c)_*(\beta) \cap (F^{-1}_c)^*(\text{\bf e}_i) 
= \beta\cap \text{\bf e}_i + \sum_{j=1}^{b_1} c_ig_i.
$$
\end{lmm}
\begin{proof}
This follows from the fact that $F_c(L) = L(c)$ is the graph of 
$\sum c_i\text{\bf e}_i$.
\end{proof}
Let
$\text{\bf b}_{+} = \sum_{i} x_i\text{\bf e}_i$.
We put $y_i = e^{x_i}$, $i=1,\ldots,b_1$ and 
$\vec x = (y_1,\ldots,y_{b_1},x_{b_1+1},\ldots,x_{b})$.
We define
\begin{equation}
\aligned
&\frak m^{(F_c)_*J,\vec x}_{k,\beta}((F^{-1}_c)^*\text{\bf x}_1,\ldots,(F^{-1}_c)^*\text{\bf x}_k)
\\
&= \sum_m\sum_{m=m_0+\ldots+m_k}T^{(F_c)_*(\beta) \cap\omega}
e^{\mu(\beta)/2}\\
&\qquad\qquad\frak m^{(F_c)_*J}_{k+m,(F^{-1}_c)^*\beta}(\text{\bf b}_{c+}^{\otimes m_0},(F^{-1}_c)^*\text{\bf x}_1,
\ldots,
(F^{-1}_c)^*\text{\bf x}_k,\text{\bf b}_{c+}^{\otimes m_k}).
\endaligned
\end{equation} 
We write
\begin{equation}
\aligned
&\frak m^{J,\vec x}_{k,\beta}(\text{\bf x}_1,\ldots,\text{\bf x}_k)
\\
&= \sum_m\sum_{m=m_0+\ldots+m_k=m}T^{\beta \cap\omega}
e^{\mu(\beta)/2}\frak m^{J}_{k+m,\beta}(\text{\bf b}_{+}^{\otimes m_0},\text{\bf x}_1,
\ldots,
\text{\bf x}_k,\text{\bf b}_{+}^{\otimes m_k}).
\endaligned
\end{equation} 
We put 
$$
\vec x(c) = (T^{c_1}y_1,\ldots,T^{c_{b_1}}y_{b_1},x_{b_1+1},\ldots,x_{b}).
$$
Then Lemmas \ref{134},\ref{weightchange} imply
\begin{equation}\label{1318}
\frak m^{(F_c)_*J,\vec x}_{k,\beta}((F^{-1}_c)^*\text{\bf x}_1,\ldots,(F^{-1}_c)^*\text{\bf x}_k)
=  \frak m^{J,\vec x(c)}_{k,\beta}(\text{\bf x}_1,\ldots,\text{\bf x}_k)
\end{equation}
We apply Corollary \ref{maindehamcoh} to $L(c)$ and $(F_c)_*J$. Then 
the sum of left hand side of (\ref{1318}) over $\beta$ converges for $y_i \in 1+\Lambda_{0,nov}^+$ and 
$v(x_i) > -\delta$ ($i=b_1+1,\ldots$).
Therefore the sum of the right hand side of (\ref{1318}) over $\beta$ also converges there.
Hence by taking various $c_i$ we obtain 3,4 of Theorem \ref{main2}.
\qed
\begin{rem}
By the construction of this section, we can prove a 
similar convergence result for the pseudo-isotopy we constructed in 
Sections \ref{geoisoto},\ref{cycAinfseccoh},\ref{pisoofpisosec}.
Therefore the family of unital cyclic filtered $A_{\infty}$ algebras in Theorem  \ref{main2}.4 
is well defined up to homotopy equivalence of unital cyclic filtered $A_{\infty}$ algebras.
\end{rem}
\section{Pseudo-isotopy of pseudo-isotopies}
\label{pisoofpisosec}
Let $\overline C$ be a finite dimensional $\R$ vector space or 
de Rham complex $\Lambda(L)$.
The vector space $C^{\infty}([0,1]^2,\overline C)$ is the 
set of all smooth maps $[0,1]^2 \to \overline C$.
Let $C^{\infty}([0,1]^2 \times \overline C)$ be the set of formal 
expression
\begin{equation}\label{pspselement}
x(t,s) + dt \wedge y(t,s) + ds \wedge z(t,s) + dt \wedge ds \wedge w(t,s)
\end{equation}
where $x(t,s), y(t,s), z(t,s), w(t,s) \in C^{\infty}([0,1]^2 \times \overline C)$.
We define degree by putting $\deg dt = \deg ds =1$.
\par
In the case of de Rham complex $\overline C=\Lambda(L)$ we put
$$
C^{\infty}([0,1]^2 \times \overline C) = \Lambda([0,1]^2\times L).
$$
(In this case also an element of $\Lambda([0,1]^1\times L)$ can be uniquely written as 
(\ref{pspselement}).)
\par
We assume that, for each $(t,s)\in [0,1]^2$, we have operations:
\begin{equation}\label{param2}
\frak m^{t,s}_{k,\beta} : B_k(\overline C[1]) \to \overline C[1]
\end{equation}
of degree $-\mu(\beta)+1$,
\begin{equation}\label{paracd2}
\frak c^{t,s}_{k,\beta} : B_k(\overline C[1]) \to \overline C[1], \quad
\frak d^{t,s}_{k,\beta} : B_k(\overline C[1]) \to \overline C[1]
\end{equation}
of degree $-\mu(\beta)$ 
and 
\begin{equation}\label{parae2}
\frak e^t_{k,\beta} : B_k(\overline C[1]) \to \overline C[1]
\end{equation}
of degree $-\mu(\beta)-1$.
We assume that they are smooth in the following sense:
$$
(s,t) \mapsto  \frak m^{t,s}_{k,\beta}(x_1,\ldots,x_k) \in C^{\infty}([0,1],\overline C)
$$
and similarly for $\frak c^{t,s}_{k,\beta}, \frak d^{t,s}_{k,\beta}, \frak e^{t,s}_{k,\beta}$.
We use them to define
\begin{equation}
\mathfrak M_{k,\beta} : 
B_k(C^{\infty}([0,1]^2 \times \overline C)[1]) \to C^{\infty}([0,1]^2 \times \overline C),
\end{equation}
as follows.
Let
$$
\text{\bf x}_i =
x_i(t,s) + dt \wedge y_i(t,s) + ds \wedge z_i(t,s) + dt \wedge ds \wedge w_i(t,s)
$$
We put
$$
\mathfrak M_{k,\beta}(\text{\bf x}_1,\ldots,\text{\bf x}_k)
= 
x(t,s) + dt \wedge y(t,s) + ds \wedge z(t,s) + dt \wedge ds \wedge w(t,s)
$$
where $x(t,s), y(t,s), z(t,s), w(t,s)$ are defined as follows.
\begin{subequations}
\begin{equation}
x(t,s) =\frak m^{t,s}_{k,\beta}(x_1(t,s),\ldots,x_k(t,s)).
\end{equation}
\begin{equation}
\aligned
y(s,t)=&\frak c^{t,s}_{k,\beta}(x_1(t,s),\ldots,x_k(t,s)) \\
&+ \sum (-1)^{*_i} \frak m^{t,s}_{k,\beta}(x_1(t,s),\ldots,y_i(t,s),\ldots,x_k(t,s))
\endaligned\end{equation}
where $*_i = \deg'x_1+\ldots+\deg'x_{i-1}+1$.
\begin{equation}
\aligned
z(s,t)=&\frak c^{t,s}_{k,\beta}(x_1(t,s),\ldots,x_k(t,s)) \\
&+ \sum (-1)^{*_i} \frak m^{t,s}_{k,\beta}(x_1(t,s),\ldots,z_i(t,s),\ldots,x_k(t,s)).
\endaligned\end{equation}
\begin{equation}
\aligned
w(s,t) &=\frak e^{t,s}_{k,\beta}(x_1(t,s),\ldots,x_k(t,s)) \\
&+ \sum (-1)^{*^1_i} \frak c^{t,s}_{k,\beta}(x_1(t,s),\ldots,z_i(t,s),\ldots,x_k(t,s)) \\
&+\sum (-1)^{*^2_i} \frak d^{t,s}_{k,\beta}(x_1(t,s),\ldots,y_i(t,s),\ldots,x_k(t,s)) \\
&+ \sum_{i<j} (-1)^{*^3_{ij}}\frak m^{t,s}_{k,\beta}(x_1(t,s),\ldots,y_i(t,s),\ldots, 
z_j(t,s),\ldots,x_k(t,s)) \\
&+ \sum_{i>j} (-1)^{*^4_{ij}}\frak m^{t,s}_{k,\beta}(x_1(t,s),\ldots,z_j(t,s),\ldots, 
y_i(t,s),\ldots,x_k(t,s)).
\endaligned\end{equation}
Here 
$$
\aligned
*^1_i &=  \deg' x_1 + \ldots + \deg'x_{i-1}\\
*^2_i &=  \deg' x_1 + \ldots + \deg'x_{i-1}\\
*^3_{ij} &= \deg' y_i + \deg' x_{i+1} + \ldots + \deg'x_{j-1}\\
*^4_{ij} &= 1 + \deg' z_i + \deg' x_{i+1} + \ldots + \deg'x_{j-1}.
\endaligned$$
\end{subequations}
In the case $(k,\beta) = (1,(0,0))$ we put
$$
\aligned
x(t,s) &= \frak m_{1,(0,0)}(x_1(t,s)) \\
y(t,s) &= \frac{d}{dt} x_1(t,s) - \frak m_{1,(0,0)}(y_1(t,s)) \\
z(t,s) &= \frac{d}{ds} x_1(t,s) - \frak m_{1,(0,0)}(z_1(t,s)) \\
w(s,t) &= \frac{d}{ds} y_1(t,s) - \frac{d}{dt} z_1(t,s) + \frak m_{1,(0,0)}(w_1(t,s)).
\endaligned
$$
\begin{dfn}
We say $(C,\langle \cdot \rangle,\{\frak m^{t,s}_{k,\beta}\},\{\frak c^{t,s}_{k,\beta}\},\{\frak d^{t,s}_{k,\beta}\},\{\frak e^{t,s}_{k,\beta}\})$ is a 
{\it pseudo-isotpy of pseudo-isotopies} if the following holds.
\smallskip
\begin{enumerate}
\item $\mathfrak M_{k,\beta}$ satisfies filtered $A_{\infty}$ formula (\ref{Ainfinityrelbeta}).
\item $\frak m^{t,s}_{k,\beta},\frak c^{t,s}_{k,\beta},\frak d^{t,s}_{k,\beta},\frak e^{t,s}_{k,\beta}$ are all 
cyclically symmetric.
\item $\frak m^{t,s}_{k,(0,0)}$ is independent of $t,s$. Moreover 
$\frak c^{t,s}_{k,(0,0)}$,$\frak e^{t,s}_{k,(0,0)}$,$\frak e^{t,s}_{k,(0,0)}$ are all zero.
\end{enumerate}
\par\smallskip
Unital and/or mod $T^{E_0}$ version is defined in the same way.
\end{dfn}
If $(C,\langle \cdot \rangle,\{\frak m^{t,s}_{k,\beta}\},\{\frak c^{t,s}_{k,\beta}\},\{\frak d^{t,s}_{k,\beta}\},\{\frak e^{t,s}_{k,\beta}\})$ is a 
pseudo-isotpy of pseudo-isotopies.
Then, for each  $s_0 \in [0,1]$, 
$(C,\langle \cdot \rangle,\{\frak m^{t,s_0}_{k,\beta}\},\{\frak c^{t,s_0}_{k,\beta}\})$ is 
a pseudo-isotopy, and, for each 
$t_0 \in [0,1]$, 
$(C,\langle \cdot \rangle,\{\frak m^{t_0,s}_{k,\beta}\},\{\frak d^{t_0,s}_{k,\beta}\})$ is 
also a pseudo-isotopy. We call them the restrictions.
\begin{thm}\label{isotopyisotopyext}
Let $E_0 < E_1$.
Let $(C,\langle \cdot \rangle,\{\frak m^{t_0,s_0}_{k,\beta}\})$ be cyclic filtered 
$A_{\infty}$ algebras modulo $T^{E_1}$ for $s_0,t_0 \in \{0,1\}$.
For $s_0 = 0,1$, let $(C,\langle \cdot \rangle,\{\frak m^{t,s_0}_{k,\beta}\},\{\frak c^{t,s_0}_{k,\beta}\})$ be 
a pseudo-isotopy  modulo $T^{E_1}$.
For $t_0 =1$, let $(C,\langle \cdot \rangle,\{\frak m^{t_0,s}_{k,\beta}\},\{\frak d^{t_0,s}_{k,\beta}\})$
be 
a pseudo-isotopy  modulo $T^{E_1}$.
\par
Let $(C,\langle \cdot \rangle,\{\frak m^{t,s}_{k,\beta}\},\{\frak c^{t,s}_{k,\beta}\},\{\frak d^{t,s}_{k,\beta}\},\{\frak e^{t,s}_{k,\beta}\})$ 
be a pseudo-isotopy of pseudo-isotopies modulo $T^{E_0}$.
\par
We assume that the restriction of  pseudo-isotopy of pseudo-isotopies 
to $s_0 = 0,1$ or to $t_0 =1$ coincides with the above 
pseudo-isotopy as pseudo-isotopies modulo $T^{E_0}$.
\par
We assume Assumption  \ref{assum} below.
\par\smallskip
Then, $(C,\langle \cdot \rangle,\{\frak m^{t,s}_{k,\beta}\},\{\frak c^{t,s}_{k,\beta}\},\{\frak d^{t,s}_{k,\beta}\},\{\frak e^{t,s}_{k,\beta}\})$ 
extends to a pseudo-isotopy of pseudo-isotopies modulo $T^{E_1}$ 
so that its restriction to $s_0 = 0,1$ or to $t_0 =1$ coincides with the above 
pseudo-isotopy as pseudo-isotopies modulo $T^{E_1}$.
\par
The unital version also holds.
\end{thm}
\begin{rem}
This theorem is a cyclic and de Rham version of 
\cite{FOOO080} Theorem 7.2.212.
\end{rem}
\begin{assump}\label{assum}
\begin{enumerate}
\item
If $E(\beta) < E_0$ then $\frak c_{k,\beta}^{t,s}, \frak d_{k,\beta}^{t,s}, \frak e_{k,\beta}^{t,s}$
are $0$ in a neighborhood of $\{0,1\}^2$ and 
$\frak m_{k,\beta}^{t,s}$ is locally constant there.
\item
If $E(\beta) < E_1$ then $\frak c_{k,\beta}^{t,s_0}$ ($s_0=0,1$) is zero 
if $t$ is in a neighborhood of  $\{0,1\}$. 
Moreover $\frak d_{k,\beta}^{0,s}$ is zero 
if $s$ is in a neighborhood of  $\{0,1\}$. 
Furthermore $\frak m^{t,s_0}_{k,\beta}$, $\frak m^{0,s}_{k,\beta}$ are 
locally constant there.
\end{enumerate}
\end{assump}
We remark that by the method of the proof of Lemma \ref{pisoequiv}, 
we may change $(C,\langle \cdot \rangle,\{\frak m^{t,s}_{k,\beta}\},\{\frak c^{t,s}_{k,\beta}\},\{\frak d^{t,s}_{k,\beta}\},\{\frak e^{t,s}_{k,\beta}\})$ so that it satisfies Assumption  \ref{assum}.
\begin{proof}[Proof of Theorem \ref{isotopyisotopyext}]
We take a map 
$$
h = (h_t(u,v),h_s(u,v)) : [0,1]_u \times [0,1]_v \to  [0,1]_t \times [0,1]_s
$$
with the following properties. (We write $[0,1]_s$ etc. to show that the parameter 
of this factor is $s$.)
\smallskip
\begin{enumerate}
\item  $h(u,0) = (1,0)$, $h(u,1) = (1,1)$.
\item $h(1,v) = (1,v)$. 
\item $h(0,1/3) = (0,0)$. $h(0,2/3) = (0,1)$.
\item The restriction of $h$ determines a homeomorphism:
$$
[0,1]^2 \setminus ([0,1] \times \{0\}) \to [0,1]^2 \setminus (([0,1] \times \{0,1\}) \cup
\{0\} \times [0,1])
$$
\item It is a diffeomorphism outside $\{(0,1/3),(0,2/3))\}$ of the domain and 
$\{(0,0),(0,1)\}$ of the target.
\end{enumerate}
\par\smallskip
In particular $h$ determines a homeomorphism
\begin{equation}\label{h1dim}
\{0\} \times [0,1]_v \cong (\{0\} \times [0,1]_s) \cup ([0,1]_t \times \{0,1\}).
\end{equation}
$(C,\langle\cdot\rangle,\{\frak m^{0,s}_{k,\beta}\},\{\frak d^{0,s}_{k,\beta}\})$  and
$(C,\langle\cdot\rangle,\{\frak m^{t,s_0}_{k,\beta}\},\{\frak c^{t,s_0}_{k,\beta}\})$  define  
pseudo-isotopies parametrized by the target of (\ref{h1dim}).
We pull it back by (\ref{h1dim}) and 
obtain a pseudo-isotpy 
$(C,\langle\cdot\rangle,\{\frak m^{0,v}_{k,\beta}\},\{\frak c^{0,v}_{k,\beta}\})$ modulo $T^{E_1}$ 
parametrized by $[0,1]_v$.
(The pull back is defined by a similar formula as (\ref{pisopullback}).
We use Assumption \ref{assum} to show the smoothness of the pull back at $v=1/3,2/3$.)
\par
We next pull back $(C,\langle \cdot \rangle,\{\frak m^{t,s}_{k,\beta}\},\{\frak c^{t,s}_{k,\beta}\},\{\frak d^{t,s}_{k,\beta}\},\{\frak e^{t,s}_{k,\beta}\})$ by $h$ as follows.
We consider
$$\aligned
&dh_t = \frac{dh_t}{du}du + \frac{dh_t}{dv}dv , \quad 
dh_s = \frac{dh_s}{du}du + \frac{dh_s}{dv}dv, \\
&dh_t \wedge dh_s = 
\left(\frac{dh_t}{du}\frac{dh_s}{dv} - \frac{dh_t}{dv}\frac{dh_s}{du}\right) du \wedge dv,
\endaligned
$$
and put
\begin{equation}
\aligned
&\frak c^{u,v}_{k,\beta} = \frac{dh_t}{du}\frak c^{h(u,v)}_{k,\beta}  + \frac{dh_s}{du}\frak d^{h(u,v)}_{k,\beta}, \\
&\frak d^{u,v}_{k,\beta} = \frac{dh_t}{dv}\frak c^{h(u,v)}_{k,\beta}  + \frac{dh_s}{dv}\frak d^{h(u,v)}_{k,\beta},  \\
&\frak e^{u,v}_{k,\beta} = \left(\frac{dh_t}{du}\frac{dh_s}{dv} - \frac{dh_t}{dv}\frac{dh_s}{du}\right)  
\frak e^{h(u,v)}_{k,\beta}.
\endaligned
\end{equation}
(We use Assumption \ref{assum} to show that the pull back is smooth.)
It may be regarded as a pseudo-isotopy modulo $T^{E_0}$ from the pull back of 
$(C^{\infty}([0,1]\times \overline C),\langle \cdot \rangle,\{\frak m^{0,v}_{k,\beta}\},\{\frak c^{0,v}_{k,\beta}\})$
to the restriction of  $(C,\langle \cdot \rangle,\{\frak m^{t,s}_{k,\beta}\},\{\frak c^{t,s}_{k,\beta}\},\{\frak d^{t,s}_{k,\beta}\},\{\frak e^{t,s}_{k,\beta}\})$ to the line $s=1$.
We now apply Theorem \ref{pisotohomotopyequiv} and extend 
$\frak c^{u,v}_{k,\beta}$, $\frak d^{u,v}_{k,\beta}$, $\frak e^{u,v}_{k,\beta}$
so that we obtain isotopy of isotopy module $T^{E_1}$.
\par
Therefore we pull back again by the inverse of $h$ to obtain 
required isotopy of isotopies. The proof of Theorem \ref{isotopyisotopyext} is complete.
\end{proof}
Now the main result of this section is:
\begin{thm}\label{isotopystrong}
The cyclic unital filtered $A_{\infty}$ structure in Theorem \ref{maindehamcomplex} 
is independent of the choices up to pseudo-isotopy.
\end{thm}
\begin{proof}
The proof is similar to \cite{FOOO080} Subsection 7.2.13.
\par
We take tame almost complex structures $J_0$ and $J_1$ and 
$E_0 < \ldots < E_i \ldots$. 
We use $J_0$ and $E_i$ to define cyclic filtered $A_{\infty}$ structures
$(\Lambda(L),\langle \cdot \rangle,\{\frak m_{k,\beta}^{(0,i)}\})$  
modulo $T^{E_i}$ and pseudo-isotopy modulo $T^{E_i}$
$(\Lambda(L),\langle \cdot \rangle,\{\frak m_{k,\beta}^{t,(0,i)}\},\{\frak c_{k,\beta}^{t,(0,i)}\})$
between $(\Lambda(L),\langle \cdot \rangle,\{\frak m_{k,\beta}^{(0,i)}\})$  
and $(\Lambda(L),\langle \cdot \rangle,\{\frak m_{k,\beta}^{(0,i+1)}\})$.
We then extend them to  cyclic filtered $A_{\infty}$ structures and 
pseudo-isotopies.
We next use $J_1$ and $E_i$ to define 
$(\Lambda(L),\langle \cdot \rangle,\{\frak m_{k,\beta}^{(1,i)}\})$  
and $(\Lambda(L),\langle \cdot \rangle,\{\frak m_{k,\beta}^{t,(1,i)}\},\{\frak c_{k,\beta}^{t,(0,i)}\})$.
And extend them to cyclic filtered $A_{\infty}$ structures and 
pseudo-isotopies.
\par
To prove Theorem \ref{isotopystrong} it suffices to 
show that the extension of $(\Lambda(L),\langle \cdot \rangle,\{\frak m_{k,\beta}^{(0,i)}\})$  
is pseudo-isotopic to the extension of $(\Lambda(L),\langle \cdot \rangle,\{\frak m_{k,\beta}^{(1,i)}\})$.
\par
Theorem \ref{existpseudoisotopy} implies that there exists a 
pseudo-isotopy modulo $T^{E_i}$, 
$(\Lambda(L),\langle \cdot \rangle,\{\frak m_{k,\beta}^{s,i}\},\{\frak d_{k,\beta}^{s,i}\})$
between $(\Lambda(L),\langle \cdot \rangle,\{\frak m_{k,\beta}^{(0,i)}\})$  
ando $(\Lambda(L),\langle \cdot \rangle,\{\frak m_{k,\beta}^{(1,i)}\})$.
\begin{lmm}\label{geokouseipipi}
There exists a pseudo-isotopy of pseudo-isotopy modulo $T^{E_i}$, 
\linebreak
$(\Lambda(L),\langle \cdot \rangle,\{\frak m^{t,s;i}_{k,\beta}\},\{\frak c^{t,s;i}_{k,\beta}\},\{\frak d^{t,s;i}_{k,\beta}\},\{\frak e^{t,s;i}_{k,\beta}\})$
so that its restriction to $t =j$ $(j=0,1)$ coincides with 
$(\Lambda(L),\langle \cdot \rangle,\{\frak m_{k,\beta}^{s,i+j}\},\{\frak c_{k,\beta}^{s,i+j}\})$
as pseudo-isotopy modulo $T^{E_{i+j}}$.
Moreover its restriction to $s_j$ $(j=0,1)$ coincides with 
$(\Lambda(L),\langle \cdot \rangle,\{\frak m_{k,\beta}^{s,i}\},\{\frak d_{k,\beta}^{s,i}\})$ 
as pseudo-isotopy modulo $T^{E_i}$.
\end{lmm}
\begin{proof}
We construct two parameter family of Kuranishi structures 
and multisections, in a way similar to Section \ref{geoisoto}.
Then we use it to construct pseudo-isotopy of pseudo-isotopies 
in the same way as (\ref{pseudoisoshiki}).
\end{proof}
By Lemma \ref{geokouseipipi} and Theorem \ref{isotopyisotopyext} 
we can extend the pseudo-isotopy modulo $T^{E_i}$
$(\Lambda(L),\langle \cdot \rangle,\{\frak m_{k,\beta}^{s,i}\},\{\frak d_{k,\beta}^{s,i}\})$
to a pseudo-isotopy. The proof of Theorem \ref{isotopystrong} is complete.
\end{proof}
The uniquness part of Theorem \ref{main1} follows from Theorems \ref{isotopystrong}, 
\ref{pisotopyextention} and \ref{pisotohomotopyequiv}.
{\scshape
\begin{flushright}
\begin{tabular}{l}
Kyoto University \\
Kitashirakawa, Sakyo-ku, \\
Kyoto 602-8502, \\
Japan \\
{\upshape e-mail: fukaya@math.kyoto-u.ac.jp}\\
\end{tabular}
\end{flushright}
}

\end{document}